\documentclass[12pt,reqno]{article}

\usepackage[usenames]{color}
\usepackage{amssymb}
\usepackage{graphicx}
\usepackage{amscd}

\usepackage[colorlinks=true,
linkcolor=webgreen,
filecolor=webbrown,
citecolor=webgreen]{hyperref}

\definecolor{webgreen}{rgb}{0,.5,0}
\definecolor{webbrown}{rgb}{.6,0,0}

\usepackage{color}
\usepackage{fullpage}
\usepackage{float}

\usepackage{graphics,amsmath,amssymb}
\usepackage{amsthm}
\usepackage{amsfonts}
\usepackage{latexsym}
\usepackage{epsf}

\newcommand{\seqnum}[1]{\href{http://oeis.org/#1}{\underline{#1}}}

\usepackage{relsize,mathtools,cancel} 
\usepackage{qsymbols}
\usepackage{url}

\usepackage[font=small,labelfont=bf,textfont=rm,format=hang]{caption,subcaption} 
\renewcommand{\thetable}{\thesection.\arabic{table}} 
 
\newcommand{\tableref}[1]{Table \ref{#1} (page \pageref{#1})} 

\usepackage{rotating} 
\newcommand{\subtablewidth}{\textwidth} 
\newcommand{\tabletopstrut}[0]{\rule{0pt}{3ex}} 
\newcommand{\tablebottomstrut}[0]{\rule{0pt}{3ex}} 

\newcommand{\tagonce}[0]{
     \addtocounter{equation}{1}
     \tag{\theequation}
} 
\newcommand{\tagtext}[1]{\tag{\footnotesize\underline{\emph{#1}}}} 

\newcommand{\itemlabel}[1]{\textbf{#1}: \\ } 
\renewcommand{\labelenumi}{$\mathsmaller{\blacktriangleright}$ } 
\newcommand{\sublabel}[1]{\noindent{\smaller{\underline{\textbf{\textit{#1}}}.\ }} \\ } 
\newcommand{\sublabelII}[1]{\noindent{\smaller{\underline{\textbf{\textit{#1}}}.\ }}\smallskip} 

\newcommand{\quotetext}[1]{``#1''} 
\newcommand{\cf}[0]{cf.\ } 
\newcommand{\ie}[0]{i.e.,\ } 

\newcommand{\OEISII}[1]{{\texttt{#1}}} 
\newcommand{\seqmapsto}[2][]{
     \xrightarrow[\text{ \OEISII{#1} }]{\text{ \OEISII{#2} }}\quad
}  

\newcommand{\defequals}{\ensuremath{\vcentcolon=}} 
\newcommand{\defmapsto}{\ensuremath{\vcentcolon\mapsto}} 
\newcommand{\undersetbrace}[2]{\ensuremath{\underset{\mathlarger{#1}}{\mathsmaller{\underbrace{#2}}}}} 

\newcommand{\TheSummaryNBFileGoogleDriveLink}[0]{https://drive.google.com/file/d/0B6na6iIT7ICZRFltbTVVcmVpVk0/view?usp=drivesdk}  
\newcommand{\TheSummaryNBFile}{ 
     \href{\TheSummaryNBFileGoogleDriveLink}{\texttt{multifact-cfracs-summary.nb}}
}
\newcommand{\Mm}[0]{\emph{Mathematica}} 

\newcommand{\gkpSI}[2]{\ensuremath{\genfrac{\lbrack}{\rbrack}{0pt}{}{#1}{#2}}} 
\newcommand{\gkpSII}[2]{\ensuremath{\genfrac{\{}{\}}{0pt}{}{#1}{#2}}} 
\newcommand{\gkpEI}[2]{\ensuremath{\genfrac{\langle}{\rangle}{0pt}{}{#1}{#2}}} 
\newcommand{\gkpEII}[2]{\ensuremath{\left\langle\genfrac{\langle}{\rangle}{
            0pt}{}{#1}{#2}\right\rangle}} 

\newcommand{\FcfII}[3]{\ensuremath{\gkpSI{#2}{#3}_{#1}}} 
\newcommand{\FFact}[3]{\ensuremath{(#1|#2)^{\underline{#3}}}} 
\newcommand{\FFactII}[2]{\ensuremath{#1^{\underline{#2}}}} 
\newcommand{\RFact}[3]{\ensuremath{(#1|#2)^{\overline{#3}}}} 
\newcommand{\RFactII}[2]{\ensuremath{#1^{\overline{#2}}}} 
\newcommand{\Pochhammer}[2]{\ensuremath{\left(#1\right)_{#2}}} 
\newcommand{\Iverson}[1]{\ensuremath{\left[#1\right]_{\delta}}} 
\newcommand{\MultiFactorial}[2]{\ensuremath{#1!_{\left(#2\right)}}} 
\newcommand{\AlphaFactorial}[2]{\ensuremath{\left(#1\right)!_{\left(#2\right)}}} 
\newcommand{\pn}[3]{\ensuremath{p_{#1}\left(#2, #3\right)}} 
\newcommand{\HypU}[3]{\ensuremath{U\left(#1, #2, #3\right)}} 
\newcommand{\HypM}[3]{\ensuremath{M\left(#1, #2, #3\right)}} 
\newcommand{\ConvGF}[4]{\ensuremath{\Conv_{#1}\left(#2, #3; #4\right)}} 
\newcommand{\ConvFP}[4]{\ensuremath{\FP_{#1}\left(#2, #3; #4\right)}} 
\newcommand{\ConvFQ}[4]{\ensuremath{\FQ_{#1}\left(#2, #3; #4\right)}} 

\newcommand{\FiGF}[1]{\ensuremath{F_{#1}}} 
\newcommand{\HNumGFFactoredDenomFn}[3]{
     \Denom_{#1, \bmod{#2}}\left\llbracket #3 \right\rrbracket 
} 
\def\?{\hbox{!`}} 

\DeclareMathOperator{\FP}{FP} 
\DeclareMathOperator{\FQ}{FQ} 
\DeclareMathOperator{\Conv}{Conv} 
\DeclareMathOperator{\as}{as} 
\DeclareMathOperator{\bs}{bs} 
\DeclareMathOperator{\cs}{cs} 
\DeclareMathOperator{\af}{af} 
\DeclareMathOperator{\bfcf}{bf} 
\DeclareMathOperator{\cfcf}{cf} 
\DeclareMathOperator{\Sf}{sf} 
\DeclareMathOperator{\E}{E} 
\DeclareMathOperator{\Denom}{Denom} 

\allowdisplaybreaks

\begin{document}


\theoremstyle{plain}
\newtheorem{theorem}{Theorem}
\newtheorem{lemma}[theorem]{Lemma}
\newtheorem{prop}[theorem]{Proposition}
\newtheorem{cor}[theorem]{Corollary}

\theoremstyle{definition}
\newtheorem{definition}[theorem]{Definition}
\newtheorem{example}[theorem]{Example}

\theoremstyle{remark}
\newtheorem{remark}[theorem]{Remark}

\begin{center}
\vskip 1cm{\LARGE\bf 
Jacobi-Type Continued Fractions for the \\
\vskip .02in
Ordinary Generating Functions of  \\
\vskip .10in
Generalized Factorial Functions
}
\vskip 1cm
\large
Maxie D. Schmidt \\ 
University of Washington \\ 
Department of Mathematics \\ 
Padelford Hall \\ 
Seattle, WA 98195 \\ 
USA \\ 
\href{mailto:maxieds@gmail.com}{\nolinkurl{maxieds@gmail.com}} 
\end{center}

\vskip .2 in
\begin{abstract} 
The article studies a class of generalized factorial functions 
and symbolic product sequences 
through Jacobi-type continued fractions (J-fractions) that 
formally enumerate the typically 
divergent ordinary generating functions of these sequences. 
The rational convergents of these generalized J-fractions 
provide formal power series approximations to the 
ordinary generating functions that enumerate many specific classes of 
factorial-related integer product sequences. 
The article also provides applications to a number of specific factorial 
sum and product identities, 
new integer congruence relations satisfied by 
generalized factorial-related product sequences, the Stirling numbers of the 
first kind, and the $r$-order harmonic numbers, as well as 
new generating functions for the sequences of binomials, $m^p-1$, 
among several other notable motivating examples given as 
applications of the new results proved in the article. 
\end{abstract} 

\vskip 0.2in 

\section{Notation and other conventions in the article} 
\label{Section_Notation_and_Convs} 

\subsection{Notation and special sequences} 

Most of the conventions in the article are consistent with the 
notation employed within the \emph{Concrete Mathematics} reference, and 
the conventions defined in the introduction to the first article 
\cite{MULTIFACTJIS}. 
These conventions 
include the following particular notational variants: 
\begin{enumerate} 
     \renewcommand{\labelenumi}{$\mathsmaller{\blacktriangleright}$ } 
     \newcommand{\localitemlabel}[1]{\textbf{#1}.\ } 

\item \localitemlabel{Extraction of formal power series coefficients} 
The special notation for formal 
power series coefficient extraction, 
$[z^n] \left( \sum_{k} f_k z^k \right) \defmapsto f_n$; 

\item \localitemlabel{Iverson's convention} 
The more compact usage of Iverson's convention, 
$\Iverson{i = j} \equiv \delta_{i,j}$, in place of 
Kronecker's delta function where 
$\Iverson{n = k = 0} \equiv \delta_{n,0} \delta_{k,0}$; 

\item \localitemlabel{Bracket notation for the Stirling and 
                      Eulerian number triangles} 
We use the alternate bracket notation for the Stirling number triangles, 
$\gkpSI{n}{k} = (-1)^{n-k} s(n, k)$ and 
$\gkpSII{n}{k} = S(n, k)$, as well as 
$\gkpEI{n}{m}$ to denote the first-order Eulerian number triangle, and 
$\gkpEII{n}{m}$ to denote the second-order Eulerian numbers; 

\item \localitemlabel{Harmonic number sequences} 
Use of the notation for the first-order harmonic numbers, $H_n$ or 
$H_n^{(1)}$, which defines the sequence 
\[
H_n \defequals 1+\frac{1}{2}+\frac{1}{3}+\cdots+\frac{1}{n}, 
\] 
and the notation for the partial sums for the more general cases of the 
$r$-order harmonic numbers, $H_n^{(r)}$, defined as 
\[ 
\mathsmaller{H_n^{(r)}} \defequals 1 + 2^{-r} + 3^{-r} + \cdots + n^{-r}, 
\]
when $r, n \geq 1$ are integer-valued and where $H_n^{(r)} \equiv 0$ for all 
$n \leq 0$; 

\item \localitemlabel{Rising and falling factorial functions} 
We use the convention of denoting the 
falling factorial function by $\FFactII{x}{n} = x! / (x-n)!$, the 
rising factorial function as $\RFactII{x}{n} = \Gamma(x+n) / \Gamma(x)$, 
or equivalently by the Pochhammer symbol, 
$\Pochhammer{x}{n} = x(x+1)(x+2) \cdots (x+n-1)$; 

\item \localitemlabel{Shorthand notation in integer congruences and modular arithmetic} 
Within the article the notation 
$g_1(n) \equiv g_2(n) \pmod{N_1, N_2, \ldots, N_k}$ is understood to 
mean that the congruence, $g_1(n) \equiv g_2(n) \pmod{N_j}$, holds 
modulo any of the bases, $N_j$, for $1 \leq j \leq k$. 

\end{enumerate} 
Within the article, the standard set notation for 
$\mathbb{Z}$, $\mathbb{Q}$, and $\mathbb{R}$ 
denote the sets of integers, rational numbers, and real numbers, respectively, 
where the set of natural numbers, $\mathbb{N}$, is defined by 
$\mathbb{N} \defequals \{0, 1, 2, \ldots \} = \mathbb{Z}^{+} \bigcup \{0\}$. 
Other more standard notation for the special functions 
cited within the article is consistent with the definitions 
employed within the \emph{NIST Handbook} reference. 

\subsection{Mathematica summary notebook document and 
            computational reference information} 
\label{subSection_MmSummaryNBInfo} 

The article is prepared with a more extensive set of 
computational data and software routines 
released as open source software to accompany the examples and 
numerous other applications suggested as topics for future 
research and investigation within the article. 
It is highly encouraged, and expected, that the 
interested reader obtain a copy of the summary notebook reference and 
computational documentation prepared in this format to 
assist with computations in a multitude of special case examples cited as 
particular applications of the new results. 

The prepared summary notebook file, \TheSummaryNBFile, 
attached to the submission of this manuscript\footnote{ 
     An official copy of the original summary notebook is also 
     linked \href{https://cs.uwaterloo.ca/journals/JIS/VOL20/Schmidt/schmidt14.nb}{here} 
     on the \emph{Journal of Integer Sequences} website. 
}
contains the working \Mm{} code to verify the formulas, 
propositions, and other identities cited within the article 
\cite{SUMMARYNBREF-STUB}. 
Given the length of the article, 
the \Mm{} summary notebook included with this submission 
is intended to help the reader with verifying and 
modifying the examples presented as 
applications of the new results cited below. 
The summary notebook also contains numerical data 
corresponding to computations of multiple examples and 
congruences specifically referenced in several places by the 
applications and tables given in the next sections of the article. 

\section{Introduction} 
\label{Section_project_overview} 
\label{Section_Intro} 

The primary focus of the new results established by this article is to 
enumerate new properties of the 
generalized symbolic product sequences, $\pn{n}{\alpha}{R}$ defined by 
\eqref{eqn_GenFact_product_form}, 
which are generated by the convergents to 
\emph{Jacobi-type continued fractions} (\emph{J-fractions}) 
that represent formal power series expansions of the 
otherwise divergent ordinary generating functions (OGFs) 
for these sequences. 
\begin{align} 
\label{eqn_GenFact_product_form} 
\pn{n}{\alpha}{R} & \defequals 
     \prod\limits_{0 \leq j < n} (R+\alpha j) \Iverson{n \geq 1} + 
     \Iverson{n = 0} \\ 
\notag 
     & \phantom{:} = 
     R (R + \alpha) (R + 2\alpha) \times \cdots \times (R + (n-1)\alpha) 
     \Iverson{n \geq 1} + \Iverson{n = 0}. 
\end{align} 
The related integer-valued cases of the multiple factorial sequences, 
or \emph{$\alpha$-factorial} functions, $\MultiFactorial{n}{\alpha}$, 
of interest in the applications of 
this article are defined recursively for any fixed 
$\alpha \in \mathbb{Z}^{+}$ and $n \in \mathbb{N}$ by the 
following equation \cite[\S 2]{MULTIFACTJIS}: 
\begin{equation} 
\label{eqn_nAlpha_Multifact_variant_rdef} 
\MultiFactorial{n}{\alpha} = 
     \begin{cases} 
     n \cdot (n-\alpha)!_{(\alpha)}, & \text{ if $n > 0$; } \\ 
     1, & \text{ if $-\alpha < n \leq 0$; } \\ 
     0, & \text{ otherwise. } 
     \end{cases} 
\end{equation} 
The particular new results studied within the article generalize the 
known series proved in the references \cite{FLAJOLET80B,FLAJOLET82}, 
including expansions of the series for generating functions enumerating the 
\emph{rising} and \emph{falling} \emph{factorial} functions, 
$\RFactII{x}{n} = (-1)^{n} \FFactII{(-x)}{n}$ and 
$\FFactII{x}{n} = x! / (x-n)! = \pn{n}{-1}{x}$, and the 
\emph{Pochhammer symbol}, $\Pochhammer{x}{n} = \pn{n}{1}{x}$, 
expanded by the 
\emph{Stirling numbers of the first kind}, $\gkpSI{n}{k}$, as 
(\seqnum{A130534}):
\begin{align*} 
\Pochhammer{x}{n} 
     & = 
     x(x+1)(x+2) \cdots (x+n-1) \Iverson{n \geq 1} + \Iverson{n = 0} \\ 
     & = 
     \left(\sum_{k=1}^{n} \gkpSI{n}{k} x^{k}\right) \times \Iverson{n \geq 1} 
     + \Iverson{n = 0}. 
\end{align*} 
The generalized rising and falling factorial functions 
denote the products, 
$\RFact{x}{\alpha}{n} = \Pochhammer{x}{n,\alpha}$ and 
$\FFact{x}{\alpha}{n} = \Pochhammer{x}{n,-\alpha}$, 
defined in the reference, where the products, 
$\Pochhammer{x}{n} = \RFact{x}{1}{n}$ and 
$\FFactII{x}{n} = \FFact{x}{1}{n}$, 
correspond to the particular special cases of these functions cited above 
\cite[\S 2]{MULTIFACTJIS}. 
Roman's \emph{Umbral Calculus} book employs the 
alternate, and less standard, notation of 
$\left(\frac{x}{a}\right)_n \defequals 
 \frac{x}{a} \left(\frac{x}{a}-1\right) \times \cdots \times 
 \left(\frac{x}{a}-n+1\right)$ to denote the 
sequences of \emph{lower factorial polynomials}, and 
$x^{(n)}$ in place of the Pochhammer symbol 
to denote the \emph{rising factorial polynomials} 
\cite[\S 1.2]{UC}. We do not use this convention within the article. 

The generalized product sequences in \eqref{eqn_GenFact_product_form} 
also correspond to the definition of the \emph{Pochhammer $\alpha$-symbol}, 
$\Pochhammer{x}{n,\alpha} = \pn{n}{\alpha}{x}$, 
defined as in 
(\cite{ON-HYPGEOMFNS-PHKSYMBOL}, \cite[Examples]{CVLPOLYS} and
\cite[\cf \S 5.4]{GKP})
for any fixed $\alpha \neq 0$ and non-zero indeterminate, 
$x \in \mathbb{C}$, by the following analogous expansions involving the 
generalized $\alpha$-factorial coefficient triangles defined in 
\eqref{eqn_Fa_rdef} of the next section of this article 
\cite[\S 3]{MULTIFACTJIS}: 
\begin{align*} 
\Pochhammer{x}{n,\alpha} 
     & = 
     x(x+\alpha)(x+2\alpha) \cdots (x+ (n-1) \alpha) \Iverson{n \geq 1} + 
     \Iverson{n = 0} \\ 
     & = 
     \left(\sum_{k=1}^{n} \gkpSI{n}{k} \alpha^{n-k} x^{k}\right) \times 
     \Iverson{n \geq 1} + \Iverson{n = 0} \\ 
     & = 
     \left(\sum_{k=0}^{n} \FcfII{\alpha}{n+1}{k+1} (x-1)^{k}\right) \times 
     \Iverson{n \geq 1} + \Iverson{n = 0}. 
\end{align*} 
We are especially interested in using the new results established in this 
article to formally generate the 
factorial-function-like product sequences, 
$\pn{n}{\alpha}{R}$ and $\pn{n}{\alpha}{\beta n + \gamma}$, for some fixed 
parameters $\alpha, \beta, \gamma \in \mathbb{Q}$, when the 
initially fixed symbolic indeterminate, $R$, depends linearly on $n$. 
The particular forms of the 
generalized product sequences of interest in the 
applications of this article are related to the \emph{Gould polynomials}, 
$G_n(x; a, b) = \frac{x}{x-an} \cdot \FFactII{\left(\frac{x-an}{b}\right)}{n}$, 
in the form of the following equation 
(\cite[\S 3.4.2]{MULTIFACTJIS},\cite[\S 4.1.4]{UC}): 
\begin{align} 
\label{eqn_pnAlphaBetanpGamma_GouldPolyExp_Ident-stmt_v1} 
\pn{n}{\alpha}{\beta n + \gamma} & = 
     \frac{(-\alpha)^{n+1}}{\gamma-\alpha-\beta} \times 
     G_{n+1}\left(\gamma-\alpha-\beta; -\beta, -\alpha\right). 
\end{align} 
Whereas the first results proved in the articles 
\cite{FLAJOLET80B,FLAJOLET82} 
are focused on establishing properties of divergent forms of the 
ordinary generating functions for a number of special sequence cases 
through more combinatorial interpretations of these 
continued fraction series, the emphasis in this article 
is more enumerative in flavor. 
The new identities involving the integer-valued cases of the 
multiple, $\alpha$-factorial functions, $\MultiFactorial{n}{\alpha}$, 
defined in \eqref{eqn_nAlpha_Multifact_variant_rdef} 
obtained by this article 
extend the study of these sequences motivated by the 
distinct symbolic polynomial expansions of these functions 
originally considered in the reference \cite{MULTIFACTJIS}. 
This article extends a number of the examples considered as 
applications of the results from the 2010 article \cite{MULTIFACTJIS} 
briefly summarized in the next subsection. 

\subsection{Polynomial expansions of generalized $\alpha$-factorial functions} 
\label{subSection_GenAlphaFactorialTriangle_exps} 

For any fixed integer $\alpha \geq 1$ and $n, k \in \mathbb{N}$, the 
coefficients defined by the triangular recurrence relation in 
\eqref{eqn_Fa_rdef} provide one approach to enumerating the 
symbolic polynomial expansions of the 
generalized factorial function product sequences 
defined as special cases of 
\eqref{eqn_GenFact_product_form} and 
\eqref{eqn_nAlpha_Multifact_variant_rdef}. 
\begin{equation} 
\label{eqn_Fa_rdef} 
\FcfII{\alpha}{n}{k} = (\alpha n+1-2\alpha)\FcfII{\alpha}{n-1}{k} + 
     \FcfII{\alpha}{n-1}{k-1} + \Iverson{n = k = 0} 
\end{equation} 
The combinatorial interpretations of these coefficients as generalized 
Stirling numbers of the first kind 
motivated in the reference \cite{MULTIFACTJIS} leads to polynomial 
expansions in $n$ of the multiple factorial function sequence variants in 
\eqref{eqn_nAlpha_Multifact_variant_rdef} that generalize the 
known formulas for the 
single and double factorial functions, $n!$ and $n!!$, 
involving the unsigned Stirling numbers of the first kind, 
$\gkpSI{n}{k} = \FcfII{1}{n}{k} = (-1)^{n-k} s(n, k)$, which are 
expanded in the forms of the following equations 
(\cite[\S 6]{GKP}, \cite[\S 26.8]{NISTHB},
\seqnum{A000142}, \seqnum{A006882}): 
\label{eqn_AlphaNm1pd_AlphaFactFn_PolyCoeffSum_Exp_formula-v1} 
\begin{align*} 
n! & = \sum_{m=0}^{n} \gkpSI{n}{m} (-1)^{n-m} n^m,\ \forall n \geq 1 \\ 
n!! & = \sum_{m=0}^{n} \gkpSI{\lfloor \frac{n+1}{2} \rfloor}{m} 
     (-2)^{\lfloor \frac{n+1}{2} \rfloor - m} n^{m},\ \forall n \geq 1 \\ 
\tagonce\label{eqn_AlphaNm1pd_AlphaFactFn_PolyCoeffSum_Exp_formula-eqns_v3} 
\MultiFactorial{n}{\alpha} & = 
     \sum_{m=0}^{n} 
     \gkpSI{\lceil n / \alpha \rceil}{m} 
     (-\alpha)^{\lceil \frac{n}{\alpha} \rceil - m} n^{m},\ 
     \forall n \geq 1, \alpha \in \mathbb{Z}^{+}. 
\end{align*} 
The Stirling numbers of the first kind similarly provide 
non-polynomial exact 
finite sum formulas for the single and double factorial 
functions in the following forms for $n \geq 1$ 
where $(2n)!! = 2^{n} \times n!$ 
(\cite[\S 6.1]{GKP},\cite[\S 5.3]{DBLFACTFN-COMBIDENTS-SURVEY}): 
\begin{align*} 
n! & = 
\mathsmaller{ 
     \sum\limits_{k=0}^{n} \gkpSI{n}{k} 
}  
\quad \text{ and } \quad 
(2n-1)!! = 
\mathsmaller{ 
     \sum\limits_{k=0}^{n} \gkpSI{n}{k} 2^{n-k} 
}. 
\end{align*} 
Related finite sums generating the 
single factorial, double factorial, and 
$\alpha$-factorial functions are expanded respectively through the 
\emph{first-order Euler numbers}, $\gkpEI{n}{m}$, as 
$n! = \sum_{k=0}^{n-1} \gkpEI{n}{k}$ 
\cite[\S 26.14(iii)]{NISTHB}, 
through the \emph{second-order Euler numbers}, $\gkpEII{n}{m}$, as 
$(2n-1)!! = \sum_{k=0}^{n-1} \gkpEII{n}{k}$ \cite[\S 6.2]{GKP}, and 
by the generalized cases of these triangles \cite[\S 6.2.4]{MULTIFACTJIS}. 

The polynomial expansions of the first two classical sequences 
of the previous equations in 
\eqref{eqn_AlphaNm1pd_AlphaFactFn_PolyCoeffSum_Exp_formula-eqns_v3} 
are generalized to the 
more general $\alpha$-factorial function cases 
through the triangles defined as in \eqref{eqn_Fa_rdef} 
from the reference \cite{MULTIFACTJIS} through the 
next explicit finite sum formulas when $n \geq 1$. 
\begin{align*} 
\tagonce\label{eqn_AlphaNm1pd_AlphaFactFn_PolyCoeffSum_Exp_formula-eqns_v1} 
\MultiFactorial{n}{\alpha} & = 
     \sum_{m=0}^{n} 
     \FcfII{\alpha}{\lfloor \frac{n-1+\alpha}{\alpha} \rfloor + 1}{m+1} 
     (-1)^{\lfloor \frac{n-1+\alpha}{\alpha} \rfloor - m} 
     (n+1)^{m},\ 
     \forall n \geq 1, \alpha \in \mathbb{Z}^{+} 
\end{align*} 
The polynomial expansions in $n$ of the generalized 
$\alpha$-factorial functions, $\AlphaFactorial{\alpha n-d}{\alpha}$, 
for fixed $\alpha \in \mathbb{Z}^{+}$ and integers $0 \leq d < \alpha$, 
are obtained similarly from 
\eqref{eqn_AlphaNm1pd_AlphaFactFn_PolyCoeffSum_Exp_formula-eqns_v1} 
through the generalized coefficients in 
\eqref{eqn_Fa_rdef} as follows \cite[\cf \S 2]{MULTIFACTJIS}: 
\begin{align*} 
\tagonce\label{eqn_AlphaNm1pd_AlphaFactFn_PolyCoeffSum_Exp_formula-eqns_v2} 
(\alpha n-d)!_{(\alpha)} 
     & = 
     (\alpha - d) \times 
     \sum_{m=1}^{n} \FcfII{\alpha}{n}{m} (-1)^{n-m} 
     (\alpha n + 1 - d)^{m-1} \\ 
     & = 
     \phantom{(\alpha - d) \times} 
     \sum_{m=0}^{n} \FcfII{\alpha}{n+1}{m+1} (-1)^{n-m} 
     (\alpha n + 1 - d)^{m},\ 
     \forall n \geq 1, \alpha \in \mathbb{Z}^{+}, 0 \leq d < \alpha. 
\end{align*} 
A binomial-coefficient-themed phrasing of the products underlying the 
expansions of the more general factorial function sequences of this type 
(each formed by dividing through by a normalizing factor of $n!$) 
is suggested in the following expansions of these coefficients by 
the Pochhammer symbol \cite[\S 5]{GKP}: 
\begin{align} 
\label{eqn_BinomCoeff_AlphaNm1pd_AlphaFactFn_Exp_formula-v2} 
\binom{\frac{s-1}{\alpha}}{n} & = 
     \frac{(-1)^{n}}{n!} \cdot \Pochhammer{\frac{s-1}{\alpha}}{n} = 
     \frac{1}{\alpha^{n} \cdot n!} \prod_{j=0}^{n-1} \left( 
     s - 1 -\alpha j 
     \right). 
\end{align} 
When the initially fixed indeterminate $s \defequals s_n$ 
is considered modulo $\alpha$ in the form of 
$s_n \defequals \alpha n + d$ for some fixed least integer residue, 
$0 \leq d < \alpha$, the 
prescribed setting of this offset $d$ completely determines the 
numerical $\alpha$-factorial function sequences of the forms in 
\eqref{eqn_AlphaNm1pd_AlphaFactFn_PolyCoeffSum_Exp_formula-eqns_v2} 
generated by these products 
(see, for example, the examples cited below in 
Section \ref{subsubSection_Intro_Examples_Fact-RelatedSeqs_GenByTheConvFns} 
and the tables given in the reference 
\cite[\cf \S 6.1.2, Table 6.1]{MULTIFACTJIS}). 

For any lower index $n \geq 1$, the 
binomial coefficient formulation of the multiple factorial function 
products in 
\eqref{eqn_BinomCoeff_AlphaNm1pd_AlphaFactFn_Exp_formula-v2} 
provides the next several expansions by the 
exponential generating functions 
for the generalized coefficient triangles in \eqref{eqn_Fa_rdef}, and 
their corresponding generalized 
Stirling polynomial analogs, $\sigma_k^{(\alpha)}(x)$, 
defined in the references 
(\cite[\S 5]{MULTIFACTJIS},\cite[\cf \S 6, \S 7.4]{GKP}): 
\begin{align} 
\label{eqn_binom_FaSPoly_exp_ident} 
\binom{\frac{s-1}{\alpha}}{n} & = 
     \sum_{m=0}^{n} 
     \FcfII{\alpha}{n+1}{n+1-m} \frac{(-1)^{m} s^{n-m}}{\alpha^{n} n!} \\ 
\notag 
     & = 
     \sum_{m=0}^{n} 
     \frac{(-1)^{m} \cdot (n+1) \sigma_m^{(\alpha)}(n+1)}{\alpha^{m}} \times 
     \frac{\left(s / \alpha\right)^{n-m}}{(n-m)!} \\ 
\label{eqn_fx_EGF_result} 
\binom{\frac{s-1}{\alpha}}{n} & = 
     [z^n] \left( 
     e^{(s-1+\alpha) z / \alpha} \left( 
     \frac{-z e^{-z}}{e^{-z} - 1} 
     \right)^{n+1} 
     \right) \\ 
\notag 
     & = 
     [z^n w^n] \left( 
     - \frac{z \cdot e^{(s-1+\alpha) z / \alpha}}{1+wz-e^{z}} 
     \right). 
\end{align} 
The generalized forms of the \emph{Stirling convolution polynomials}, 
$\sigma_n(x) \equiv \sigma_n^{(1)}(x)$, and the 
\emph{$\alpha$-factorial polynomials}, $\sigma_n^{(\alpha)}(x)$, 
studied in the reference are defined for each $n \geq 0$ through the 
triangle in \eqref{eqn_Fa_rdef} as follows 
\cite[\S 5.2]{MULTIFACTJIS}: 
\begin{align*} 
\tagtext{Generalized Stirling Polynomials} 
x \cdot \sigma_n^{(\alpha)}(x) & \defequals 
     \FcfII{\alpha}{x}{x-n} \frac{(x-n-1)!}{(x-1)!} \\ 
     & \phantom{:} = 
     [z^n] \left( 
     e^{(1-\alpha) z} \left(\frac{\alpha z e^{\alpha z}}{e^{\alpha z}-1} 
     \right)^{x} \right). 
\end{align*} 
A more extensive treatment of the properties and generating function 
relations satisfied by the triangular coefficients defined by 
\eqref{eqn_Fa_rdef}, including their similarities to the 
Stirling number triangles, Stirling convolution polynomial sequences, and the 
generalized Bernoulli polynomials, among relations to several other notable 
special sequences, is provided in the references. 

\subsection{Divergent ordinary generating functions 
            approximated by the convergents to 
            infinite Jacobi-type and Stieltjes-type 
            continued fraction expansions} 
\label{subSection_Intro_Examples_DivergentCFracOGFs} 

\subsubsection{Infinite J-fraction expansions generating the 
               rising factorial function} 

Another approach to enumerating the symbolic expansions of the 
generalized $\alpha$-factorial function 
sequences outlined above is constructed 
as a new generalization of the continued fraction series 
representations of the ordinary generating function for the 
rising factorial function, or Pochhammer symbol, 
$\Pochhammer{x}{n} = \Gamma(x+n) / \Gamma(x)$, 
first proved by Flajolet \cite{FLAJOLET80B,FLAJOLET82}. 
For any fixed non-zero indeterminate, $x \in \mathbb{C}$, the 
ordinary power series enumerating the rising factorial 
sequence is defined through the 
next infinite Jacobi-type J-fraction 
expansion \cite[\S 2, p.\ 148]{FLAJOLET80B}: 
\begin{equation} 
\label{eqn_PochhammerSymbol_InfCFrac_series_Rep_example-v1} 
R_0(x, z): = 
     \sum_{n \geq 0} (x)_n z^n = 
     \cfrac{1}{1-xz-\cfrac{1 \cdot x z^2}{1-(x+2)z-
     \cfrac{2(x+1)z^2}{1-(x+4) z - \cfrac{3(x+2) z^2}{\cdots}.}}} 
\end{equation} 
Since we know symbolic polynomial expansions of the functions, $(x)_n$, 
through the Stirling numbers of the first kind, 
we notice that the terms in a convergent 
power series defined by 
\eqref{eqn_PochhammerSymbol_InfCFrac_series_Rep_example-v1} 
correspond to the normalized coefficients of the 
following well-known two-variable 
\quotetext{\emph{double}}, or \quotetext{\emph{super}}, 
exponential generating functions (EGFs) 
for the Stirling number triangle 
when $x$ is taken to be a fixed, formal parameter 
with respect to these series 
(\cite[\S 7.4]{GKP},\cite[\S 26.8(ii)]{NISTHB},
\cite[\cf Prop.\ 9]{FLAJOLET80B}): 
\begin{align*} 
\sum_{n \geq 0} \Pochhammer{x}{n} \frac{z^n}{n!} & = \frac{1}{(1-z)^{x}} 
     \qquad \text{ and } \qquad 
\sum_{n \geq 0} \FFactII{x}{n} \cdot \frac{z^{n}}{n!} = (1+z)^{x}. 
\end{align*} 
For natural numbers $m \geq 1$ and fixed $\alpha \in \mathbb{Z}^{+}$, the 
coefficients defined by the generalized triangles in 
\eqref{eqn_Fa_rdef} are enumerated similarly by the 
generating functions \cite[\cf \S 3.3]{MULTIFACTJIS} 
\begin{align*} 
\tagonce\label{eqn_FcfIIAlphanp1mp1_two-variable_EGFwz-footnote_v1} 
\sum_{m,n \geq 0} \FcfII{\alpha}{n+1}{m+1} \frac{x^m z^n}{n!} & = 
     (1- \alpha z)^{-(x+1) / \alpha}. 
\end{align*} 
When $x$ depends linearly on $n$, the ordinary generating 
functions for the numerical factorial functions formed by 
$(x)_n$ do not converge for $z \neq 0$. 
However, the convergents of the continued fraction representations of 
these series still lead to partial, truncated series approximations 
generating these generalized product sequences, 
which in turn immediately satisfy a number of 
combinatorial properties, recurrence relations, and other 
established integer congruence properties implied by the 
rational convergents to the first continued fraction expansion given in 
\eqref{eqn_PochhammerSymbol_InfCFrac_series_Rep_example-v1}. 

\subsubsection{Examples} 

Two particular divergent ordinary generating functions 
for the single factorial function sequences, 
$f_1(n) \defequals n!$ and $f_2(n) \defequals (n+1)!$, 
are cited in the references as examples of the 
Jacobi-type J-fraction results proved in 
Flajolet's articles 
(\cite{FLAJOLET80B,FLAJOLET82},\cite[\cf \S 5.5]{GFLECT}). 
The next pair of series expansions serve to illustrate the utility to 
enumerating each sequence formally with respect to $z$ 
required by the results in this article 
\cite[Thm.\ 3A; Thm.\ 3B]{FLAJOLET80B}. 
\begin{align*} 
\tagtext{Single Factorial J-Fractions} 
F_{1,\infty}(z) & \defequals 
     \sum_{n \geq 0} n! \cdot z^n && = 
     \cfrac{1}{1-z-\cfrac{1^2 \cdot z^2}{ 
     1-3z-\cfrac{2^2 z^2}{\cdots}}} \\ 
F_{2,\infty}(z) & \defequals 
     \sum_{n \geq 0} (n+1)! \cdot z^n && = 
     \cfrac{1}{1-2z-\cfrac{1 \cdot 2 z^2}{ 
     1-4z-\cfrac{2 \cdot 3 z^2}{\cdots}}} 
\end{align*} 
In each of these respective formal power series expansions, 
we immediately see that for each finite $h \geq 1$, the 
$h^{th}$ convergent functions, denoted $F_{i,h}(z)$ for $i = 1,2$, 
satisfy $f_i(n) = [z^n] F_{i,h}(z)$ whenever $1 \leq n < 2h$. 
We also have that 
$f_i(n) \equiv [z^n] F_{i,h}(z) \pmod{p}$ 
for any $n \geq 0$ whenever $p$ is a divisor of $h$ 
(\cite{FLAJOLET82},\cite[\cf \S 5]{GFLECT}). 

Similar expansions of other factorial-related 
continued fraction series are given in the references 
(\cite{FLAJOLET80B},\cite[\cf \S 5.9]{GFLECT}). 
For example, the next known 
\emph{Stieltjes-type} continued fractions (\emph{S-fractions}), 
formally generating the double factorial function, 
$(2n-1)!!$, and the \emph{Catalan numbers}, 
$C_n$, respectively, 
are expanded through the convergents of the following 
infinite continued fractions 
(\cite[Prop.\ 5; Thm.\ 2]{FLAJOLET80B}, \cite[\S 5.5]{GFLECT},
\seqnum{A001147}, \seqnum{A000108}): 
\begin{align*} 
\tagtext{Double Factorial S-Fractions} 
F_{3,\infty}(z) & \defequals 
     \sum_{n \geq 0} \undersetbrace{(2n-1)!!}{1 \cdot 3 \cdots (2n-1)} 
     \times z^{2n} && = 
     \cfrac{1}{
     1-\cfrac{1 \cdot z^2}{ 
     1-\cfrac{2 \cdot z^2}{
     1-\cfrac{3 \cdot z^2}{\cdots}}}} \\ 
F_{4,\infty}(z) & \defequals 
     \sum_{n \geq 0} \undersetbrace{C_n = \binom{2n}{n} \frac{1}{(n+1)}}{
     \frac{2^{n} (2n-1)!!}{(n+1)!}} 
     \times z^{2n} && = 
     \cfrac{1}{
     1-\cfrac{z^2}{ 
     1-\cfrac{z^2}{
     1-\cfrac{z^2}{\cdots}.}}} 
\end{align*} 
For comparison, some related forms of 
regularized ordinary power series in $z$ generating the 
single and double factorial function sequences from the 
previous examples are stated in terms of the 
\emph{incomplete gamma function}, 
$\Gamma(a, z) = \int_z^{\infty} t^{a-1} e^{-t} dt$, 
as follows \cite[\S 8]{NISTHB}: 
\begin{align*} 
\sum_{n \geq 0} n! \cdot z^{n} & = 
     -\frac{e^{-1/z}}{z} \times 
     \Gamma\left(0, -\frac{1}{z}\right) \\ 
\sum_{n \geq 0} (n+1)! \cdot z^{n} & = 
     -\frac{e^{-1/z}}{z^2} \times 
     \Gamma\left(-1, -\frac{1}{z}\right) \\ 
\tagonce\label{eqn_RelatedFormsOf_SgAndDblFact_OGFs-intro_examples-stmts_v1} 
\sum_{n \geq 1} (2n-1)!! \cdot z^{n} & = 
     -\frac{e^{-1 / 2 z}}{(-2z)^{3/2}} \times 
     \Gamma\left(-\frac{1}{2}, -\frac{1}{2 z}\right). 
\end{align*} 
Since $\pn{n}{\alpha}{R} = \alpha^{n} \Pochhammer{R / \alpha}{n}$, the 
exponential generating function for the 
generalized product sequences corresponds to the series 
(\cite[\cf \S 7.4, eq.\ (7.55)]{GKP},\cite{CVLPOLYS}):
\begin{align*} 
\tagtext{Generalized Product Sequence EGFs} 
\widehat{P}(\alpha, R; z) & \defequals 
     \sum_{n=0}^{\infty} \pn{n}{\alpha}{R} \frac{z^n}{n!} = 
     (1-\alpha z)^{-R / \alpha}, 
\end{align*} 
where for each fixed $\alpha \in \mathbb{Z}^{+}$ and $0 \leq r < \alpha$, 
we have the identities, 
$\AlphaFactorial{\alpha n-r}{\alpha} = \pn{n}{\alpha}{\alpha - r} = 
 \alpha^{n} \Pochhammer{1 - \frac{r}{\alpha}}{n}$. 
The form of this exponential generating function 
then leads to the next forms of the regularized sums by applying a 
\emph{Laplace transform} to the generating functions in the 
previous equation (see Remark \ref{remark_Formal_Laplace-Borel_Transforms}) 
(\cite[\S B.14]{ACOMB-BOOK},\cite[\cf \S 8.6(i)]{NISTHB}). 
\begin{align*} 
\widetilde{B}_{\alpha,-r}(z) & \defequals 
\sum_{n \geq 0} \AlphaFactorial{\alpha n - r}{\alpha} z^n \\ 
     & \phantom{:} = 
\int_0^{\infty} \frac{e^{-t}}{(1-\alpha tz)^{1 - r / \alpha}} dt \\ 
     & \phantom{:} = 
\frac{e^{-\frac{1}{\alpha z}}}{(-\alpha z)^{1 - r / \alpha}} \times 
     \Gamma\left(\frac{r}{\alpha}, -\frac{1}{\alpha z}\right) 
\end{align*} 
The remarks given in 
Section \ref{subSection_AltExps_of_the_GenConvFns} 
suggest similar approximations to the 
$\alpha$-factorial functions generated by the 
generalized convergent functions defined in the next section, and their 
relations to the confluent hypergeometric functions and the 
associated Laguerre polynomial sequences 
(\cite[\cf \S 18.5(ii)]{NISTHB}, \cite[\S 4.3.1]{UC}). 

\subsection{Generalized convergent functions generating 
            factorial-related integer product sequences} 
\label{subSection_Intro_GenConvFn_Defs_and_Properties} 

\subsubsection{Definitions of the generalized J-fraction expansions and 
               the generalized convergent function series} 

We state the next definition as a 
generalization of the result for the 
rising factorial function due to Flajolet cited in 
\eqref{eqn_PochhammerSymbol_InfCFrac_series_Rep_example-v1} 
to form the analogous series enumerating the 
multiple, or $\alpha$-factorial, product 
sequence cases defined by \eqref{eqn_GenFact_product_form} and 
\eqref{eqn_nAlpha_Multifact_variant_rdef}. 

\begin{definition}[Generalized J-Fraction Convergent Functions] 
\label{def_GenConvFns_PFact_Phz_eqn_QFact_Qhz-defs_intro_v1} 
Suppose that the parameters 
$h \in \mathbb{N}$, $\alpha \in \mathbb{Z}^{+}$ and $R \defequals R(n)$ 
are defined in the notation of the product-wise sequences from 
\eqref{eqn_GenFact_product_form}. 
For $h \geq 0$ and $z \in \mathbb{C}$, we let the component 
numerator and denominator convergent functions, denoted  
$\FP_h(\alpha, R; z)$ and $\FQ_h(\alpha, R; z)$, respectively, 
be defined by the recurrence relations in the next equations.
\begin{align} 
\label{eqn_PFact_Phz} 
\FP_h(\alpha, R; z) & \defequals 
     \begin{cases} 
     \mathsmaller{(1-(R+2\alpha(h-1))z)\FP_{h-1}(\alpha, R; z)- 
     \alpha(R+\alpha(h-2))(h-1) z^2 \FP_{h-2}(\alpha, R; z)} & 
     \text{if $h \geq 2$; } \\ 
     1, & \text{if $h = 1$; } \\ 
     0, & \text{otherwise. } 
     \end{cases} \\ 
\label{eqn_QFact_Qhz} 
\FQ_h(\alpha, R; z) & \defequals 
     \begin{cases} 
     \mathsmaller{(1-(R+2\alpha(h-1))z)\FQ_{h-1}(\alpha, R; z)- 
     \alpha(R+\alpha(h-2))(h-1) z^2 \FQ_{h-2}(\alpha, R; z)} & 
     \text{if $h \geq 2$; } \\ 
     1-Rz, & \text{if $h = 1$; } \\ 
     1, & \text{if $h = 0$; } \\ 
     0, & \text{otherwise. } 
     \end{cases} 
\end{align} 
The corresponding convergent functions, 
$\ConvGF{h}{\alpha}{R}{z}$, defined in the 
next equation then provide the 
rational, formal power series approximations in $z$ to the 
divergent ordinary generating functions of many factorial-related sequences 
formed as special cases of the symbolic products in 
\eqref{eqn_GenFact_product_form}. 
\begin{align} 
\notag 
\ConvGF{h}{\alpha}{R}{z} & := 
     \cfrac{1}{1 - R \cdot z - 
     \cfrac{\alpha R \cdot z^2}{ 
            1 - (R+2\alpha) \cdot z -
     \cfrac{2\alpha (R + \alpha) \cdot z^2}{ 
            1 - (R + 4\alpha) \cdot z - 
     \cfrac{3\alpha (R + 2\alpha) \cdot z^2}{ 
     \cfrac{\cdots}{1 - (R + 2 (h-1) \alpha) \cdot z}}}}} \\ 
\label{eqn_ConvGF_notation_def} 
     & \phantom{:} = 
     \frac{\FP_h(\alpha, R; z)}{\FQ_h(\alpha, R; z)} = 
     \sum_{n=0}^{2h-1} p_n(\alpha, R) z^n + 
     \sum_{n=2h}^{\infty} \widetilde{e}_{h,n}(\alpha, R) z^n 
\end{align} 
The first series coefficients on the right-hand-side of 
\eqref{eqn_ConvGF_notation_def} generate the products, 
$p_n(\alpha, R)$, from \eqref{eqn_GenFact_product_form}, where the 
remaining forms of the power series coefficients, 
$\widetilde{e}_{h,n}(\alpha, R)$, 
correspond to \quotetext{error terms} in the 
truncated formal series approximations to the 
exact sequence generating functions 
obtained from these convergent functions, which are defined such that 
$\pn{n}{\alpha}{R} \equiv \widetilde{e}_{h,n}(\alpha, R) \pmod{h}$ 
for all $h \geq 2$ and $n \geq 2h$. 
\end{definition} 

\subsubsection{Properties of the generalized J-fraction convergent functions} 

A number of the immediate, noteworthy properties satisfied by these 
generalized convergent functions 
are apparent from inspection of the first few special cases provided in 
\tableref{table_SpCase_Listings_Of_PhzQhz_ConvFn} and in 
\tableref{table_RelfectedConvNumPolySeqs_sp_cases}. 
The most important of these properties relevant 
to the new interpretations of the 
$\alpha$-factorial function sequences proved in the next sections of the 
article are briefly summarized in the points stated below. 

\begin{enumerate} 

\item \itemlabel{Rationality of the convergent functions in $\alpha$, $R$, and $z$} 
For any fixed $h \geq 1$, it is easy to show that the 
component convergent functions, $\FP_h(z)$ and $\FQ_h(z)$, 
defined by \eqref{eqn_PFact_Phz} and \eqref{eqn_QFact_Qhz}, respectively, 
are polynomials of finite degree in each of $z$, $R$, and $\alpha$ 
satisfying 
\begin{equation} 
\notag 
\deg_{z,R,\alpha}\bigl\{ \FP_h(\alpha, R; z) \bigr\} = h-1 
     \quad \text{ and } \quad 
\deg_{z,R,\alpha}\bigl\{ \FQ_h(\alpha, R; z) \bigr\} = h. 
\end{equation} 
For any $h, n \in \mathbb{Z}^{+}$, if $R \defequals R(n)$ 
denotes some linear function of $n$, the product sequences, 
$p_n(\alpha, R)$, 
generated by the generalized convergent functions 
always correspond to polynomials in $n$ (in $R$) 
of predictably finite degree with integer coefficients determined by the 
choice of $n \geq 1$. 

\item \itemlabel{Expansions of the denominator convergent functions by 
                 special functions} 
For all $h \geq 0$, and fixed non-zero parameters $\alpha$ and $R$, the 
power series in $z$ generated by the generalized $h^{th}$ 
convergents, $\ConvGF{h}{\alpha}{R}{z}$, are characterized by the 
representations of the convergent denominator functions, 
$\FQ_h(\alpha, R; z)$, through the 
confluent hypergeometric functions, 
$\HypU{a}{b}{w}$ and $\HypM{a}{b}{w}$, and the 
associated Laguerre polynomial sequences, $L_n^{(\beta)}(x)$, 
as follows (see Section \ref{subSection_AltExps_of_the_GenConvFns}) 
(\cite[\S 13; \S 18]{NISTHB}, \cite[\S 4.3.1]{UC}): 
\begin{align*} 
\tagonce\label{eqn_PFact_Qhz_Exp_idents-stmts_v1} 
\undersetbrace{\widetilde{\FQ}_h(\alpha, R; z)}{ 
     z^{h} \cdot \FQ_h\left(\alpha, R; z^{-1}\right) 
} & = 
     \alpha^{h} \times \HypU{-h}{\frac{R}{\alpha}}{\frac{z}{\alpha}} \\ 
     & = 
     (-\alpha)^{h} \Pochhammer{R / \alpha}{h} \times 
     \HypM{-h}{\frac{R}{\alpha}}{\frac{z}{\alpha}} \\ 
     & = 
     (-\alpha)^{h} \cdot h! \times 
     L_h^{(R / \alpha - 1)}\left(\frac{z}{\alpha}\right). 
\end{align*} 
The special function expansions of the reflected convergent denominator 
function sequences above lead to the statements of 
addition theorems, multiplication theorems, and 
several additional auxiliary recurrence relations for these functions 
proved in Section \ref{subsubSection_Properties_Of_ConvFn_Qhz}. 

\item \itemlabel{Corollaries: 
       New exact formulas and congruence properties for the 
       $\alpha$-factorial functions and the generalized product sequences} 
If some ordering of the $h$ zeros of 
\eqref{eqn_PFact_Qhz_Exp_idents-stmts_v1} is fixed at each $h \geq 1$, 
we can define the next sequences 
which form special cases of the zeros studied in the references 
\cite{LGWORKS-ASYMP-SPFNZEROS2008,PROPS-ZEROS-CHYPFNS80}. 
In particular, each of the following special zero sequence definitions 
given as ordered sets, or ordered lists of zeros, 
provide factorizations over $z$ of the denominator sequences, 
$\FQ_h(\alpha, R; z)$, parameterized by $\alpha$ and $R$ 
\cite[\cf \S 13.9, \S 18.16]{NISTHB}: 
\begin{align} 
\tagtext{Special Function Zeros} 
\left( \ell_{h,j}(\alpha, R) \right)_{j=1}^{h} & \defequals 
     \left\{ z_j : 
     \alpha^{h} \times \HypU{-h}{R / \alpha}{\frac{z}{\alpha}} = 0,\ 
     1 \leq j \leq h 
     \right\} \\ 
\notag 
     & \phantom{:} = 
     \left\{ z_j : 
     \alpha^{h} \times L_h^{(R / \alpha - 1)}\left(\frac{z}{\alpha}\right) = 0,\ 
     1 \leq j \leq h 
     \right\}. 
\end{align} 
Let the sequences, $c_{h,j}(\alpha, R)$, denote a shorthand 
for the coefficients corresponding to an 
expansion of the generalized convergent functions, 
$\ConvGF{h}{\alpha}{R}{z}$, by 
partial fractions in $z$, \ie the coefficients 
defined so that \cite[\S 1.2(iii)]{NISTHB} 
\begin{align*} 
\ConvGF{h}{\alpha}{R}{z} & \defequals 
     \sum_{j=1}^{h} \frac{c_{h,j}(\alpha, R)}{ 
     (1-\ell_{h,j}(\alpha, R) \cdot z)}. 
\end{align*} 
For $n \geq 1$ and any fixed integer $\alpha \neq 0$, these 
rational convergent functions 
provide the following formulas exactly generating the 
respective sequence cases in \eqref{eqn_GenFact_product_form} and 
\eqref{eqn_nAlpha_Multifact_variant_rdef}: 
\begin{align} 
\label{eqn_AlphaFactFn_Exact_PartialFracsRep_v1} 
p_n(\alpha, R) & = 
     \sum_{j=1}^{n} c_{n,j}(\alpha, R) \times 
     \ell_{n,j}(\alpha, R)^{n} \\ 
\notag 
n!_{(\alpha)} & = 
     \sum_{j=1}^{n} c_{n,j}(-\alpha, n) \times 
     \ell_{n,j}(-\alpha, n)^{\lfloor \frac{n-1}{\alpha} \rfloor}. 
\end{align} 
The corresponding congruences satisfied by each of these 
generalized sequence cases obtained from the $h^{th}$ 
convergent function expansions in $z$ are stated similarly 
modulo any prescribed integers $h \geq 2$ and fixed $\alpha \geq 1$ 
in the next forms. 
\begin{align} 
\label{eqn_AlphaFactFn_Exact_PartialFracsRep_v2} 
p_n(\alpha, R) & \equiv 
     \sum_{j=1}^{h} c_{h,j}(\alpha, R) \times 
     \ell_{h,j}(\alpha, R)^{n} 
     && \pmod{h} \\ 
\notag 
n!_{(\alpha)} & \equiv 
     \sum_{j=1}^{h} c_{h,j}(-\alpha, n) \times 
     \ell_{h,j}(-\alpha, n)^{\lfloor \frac{n-1}{\alpha} \rfloor} 
     && \pmod{h, h\alpha, \cdots, h\alpha^{h}} 
\end{align} 
Section \ref{subsubSection_Examples_NewCongruences} and 
Section \ref{subSection_NewCongruence_Relations_Modulo_Integer_Bases} 
provide several particular special case examples of the new 
congruence properties expanded by 
\eqref{eqn_AlphaFactFn_Exact_PartialFracsRep_v2}. 

\end{enumerate} 

\section{New results proved within the article} 
\label{subSection_Intro_Examples} 

\subsection{A summary of the new results and 
            outline of the article topics} 

\subsubsection*{J-Fractions for generalized factorial function sequences 
                (Section \ref{Section_Proofs_of_the_GenCFracReps} on page 
                 \pageref{Section_Proofs_of_the_GenCFracReps})} 

The article contains a number of new results and new examples of 
applications of the results from 
Section \ref{Section_Proofs_of_the_GenCFracReps} in the next subsections. 
The Jacobi-type continued fraction expansions formally enumerating the 
generalized factorial functions, $p_n(\alpha, R)$, proved in 
Section \ref{Section_Proofs_of_the_GenCFracReps} are new, and 
moreover, follow easily from the known 
continued fractions for the series generating the 
\emph{rising factorial function}, $(x)_n = p_n(1, x)$, established by 
Flajolet \cite{FLAJOLET80B}. 

\subsubsection*{Properties of the generalized convergent functions 
                (Section \ref{Section_Props_Of_CFracExps_OfThe_GenFactFnSeries} 
                 on page 
                 \pageref{Section_Props_Of_CFracExps_OfThe_GenFactFnSeries})} 

In Section \ref{Section_Props_Of_CFracExps_OfThe_GenFactFnSeries} 
we give proofs of new properties, expansions, recurrence relations, and 
exact closed-form representations by special functions 
satisfied by the convergent numerator and denominator subsequences, 
$\FP_h(\alpha, R; z)$ and $\FQ_h(\alpha, R; z)$. 
The consideration of the convergent function approximations to these 
infinite continued fraction expansions is a new topic not previously 
explored in the references which leads to new integer congruence results 
for factorial functions as well as new applications to 
generating function identities enumerating factorial-related 
integer sequences. 

\subsubsection*{Applications and motivating examples 
                (Section \ref{Section_Apps_and_Examples} on page 
                 \pageref{Section_Apps_and_Examples})} 

Specific consequences of the convergent function properties we prove in 
Section \ref{Section_Props_Of_CFracExps_OfThe_GenFactFnSeries} 
of the article include new congruences for and new rational 
generating function representations enumerating the Stirling numbers of the 
first kind, $\gkpSI{n}{k}$, modulo fixed integers $m \geq 2$ and the 
scaled $r$-order harmonic numbers, $n!^{r} \cdot H_n^{(r)}$, modulo $m$, 
which are easily extended to formulate analogous congruence results for the 
$\alpha$-factorial functions, $n!_{(\alpha)}$, and the generalized 
factorial product functions, $p_n(\alpha, R)$. 
These particular special cases of the new congruence results are proved in 
Section \ref{subSection_Congruences_for_Series_ModuloIntegers_p}. 
Section \ref{subSection_Congruences_for_Series_ModuloIntegers_p} also 
contains proofs of new representations of exact formulas for 
factorial functions expanded by the special zeros of the 
\emph{Laguerre polynomials} and \emph{confluent hypergeometric} functions. 

The subsections of Section \ref{Section_Apps_and_Examples} also 
provide specific sequence examples and new sequence generating function 
identities that demonstrate the utility and breadth of new applications 
implied by the convergent-based rational and hybrid-rational 
generating functions we rigorously treat first in 
Section \ref{Section_Proofs_of_the_GenCFracReps} and 
Section \ref{Section_Props_Of_CFracExps_OfThe_GenFactFnSeries}. 
For example, in Section \ref{subSection_DiagonalGFSequences_Apps}, 
we are the first to notice 
several specific integer sequence applications of new 
convergent-function-based \emph{Hadamard product} identities that 
effectively provide truncated series approximations to the formal 
\emph{Laplace-Borel transformation} where multiples of the 
\emph{rational} convergents, $\Conv_n(\alpha, R; z)$, generate the 
sequence multiplier, $n!$, in place of more standard integral 
representations of the transformation. 
In Section \ref{subsubSection_Apps_ArithmeticProgs_of_the_SgFactFns} through 
Section \ref{subsubSection_Apps_Example_SumsOfPowers_Seqs}, 
we focus on expanding particular examples of convergent-function-based 
generating functions enumerating special factorial-related integer 
sequences and combinatorial identities. 
The next few subsections of the article provide several special case 
examples of the new applications, new congruences, and other 
examples of the new results established in 
Section \ref{Section_Apps_and_Examples}. 

\subsection{Examples of factorial-related finite product sequences 
            enumerated by the 
            generalized convergent functions} 
\label{subsubSection_Intro_Examples_Fact-RelatedSeqs_GenByTheConvFns} 
\label{prop_Conv-Based_Defs_for_FactFn_Variants} 
\label{cor_NumericalMultiFactSeqsEnum_Alpha1234} 

\subsubsection{Generating functions for arithmetic progressions of the 
               $\alpha$-factorial functions} 

There are a couple of noteworthy subtleties that arise in defining the specific 
numerical forms of the $\alpha$-factorial function sequences 
defined by 
\eqref{eqn_nAlpha_Multifact_variant_rdef} and 
\eqref{eqn_AlphaNm1pd_AlphaFactFn_PolyCoeffSum_Exp_formula-eqns_v1}. 
First, since the generalized convergent functions generate the distinct 
symbolic products that characterize the forms of these expansions, 
we see that the following convergent-based enumerations of the 
multiple factorial sequence variants hold at each 
$\alpha, n \in \mathbb{Z}^{+}$, and 
some fixed choice of the prescribed offset, $0 \leq d < \alpha$: 
\begin{align*} 
\tagonce\label{eqn_MultFactFn_ConvSeq_def_v1} 
\left(\alpha n-d\right)!_{(\alpha)} & = 
     \undersetbrace{\pn{n}{-\alpha}{\alpha n-d}}{ 
     \left(-\alpha\right)^{n} \cdot \Pochhammer{\frac{d}{\alpha} - n}{n} 
     } = 
     [z^n] \ConvGF{n}{-\alpha}{\alpha n-d}{z} \\ 
     & = 
     \undersetbrace{\pn{n}{\alpha}{\alpha-d}}{ 
     \alpha^{n} \cdot \Pochhammer{1 - \frac{d}{\alpha}}{n} 
     } 
     \phantom{- n} = 
     [z^n] \ConvGF{n}{\alpha}{\alpha - d}{z}. 
\end{align*} 
For example, 
some variants of the arithmetic progression sequences formed by the 
single factorial and double factorial functions, $n!$ and $n!!$, in 
Section \ref{subsubSection_Apps_ArithmeticProgs_of_the_SgFactFns} 
are generated by the particular shifted inputs to these functions 
highlighted by the special cases in the next equations 
(\seqnum{A000142}, \seqnum{A000165}, \seqnum{A001147}): 
\begin{align*} 
\left( n! \right)_{n=1}^{\infty} & = 
     \left( \Pochhammer{1}{n} \right)_{n=1}^{\infty} 
     && \seqmapsto{A000142} 
     \left(1, 2, 6, 24, 120, 720, 5040, \ldots \right) \\ 
\left( (2n)!! \right)_{n=1}^{\infty} & = 
     \left( 2^{n} \cdot \Pochhammer{1}{n} \right)_{n=1}^{\infty} 
     && \seqmapsto{A001147}  
     \left(2, 8, 48, 384, 3840, 46080, \ldots \right) \\ 
\left( (2n-1)!! \right)_{n=1}^{\infty} & = 
     \left( 2^{n} \cdot \Pochhammer{1/2}{n} \right)_{n=1}^{\infty} 
     && \seqmapsto{A000165} 
     \left(1, 3, 15, 105, 945, 10395, \ldots \right). 
\end{align*} 
The next few special case variants of the $\alpha$-factorial function 
sequences corresponding to $\alpha \defequals 3, 4$, also expanded in 
Section \ref{subsubSection_Apps_ArithmeticProgs_of_the_SgFactFns}, 
are given in the following sequence forms 
(\seqnum{A032031}, \seqnum{A008544}, \seqnum{A007559},
\seqnum{A047053}, \seqnum{A007696}): 
\begin{align*} 
\left( (3n)!!! \right)_{n=1}^{\infty} & = 
     \left( 3^{n} \cdot \Pochhammer{1}{n} \right)_{n=1}^{\infty} 
     && \seqmapsto{A032031} 
     \left(3, 18, 162, 1944, 29160, \ldots \right) \\ 
\left( (3n-1)!!! \right)_{n=1}^{\infty} & = 
     \left( 3^{n} \cdot \Pochhammer{2/3}{n} \right)_{n=1}^{\infty} 
     && \seqmapsto{A008544} 
     \left(2, 10, 80, 880, 12320, 209440, \ldots \right) \\ 
\left( (3n-2)!!! \right)_{n=1}^{\infty} & = 
     \left( 3^{n} \cdot \Pochhammer{1/3}{n} \right)_{n=1}^{\infty} 
     && \seqmapsto{A007559} 
     \left(1, 4, 28, 280, 3640, 58240, \ldots \right) \\ 
\left( \AlphaFactorial{4n}{4} \right)_{n=0}^{\infty} & = 
     \left( 4^{n} \cdot \Pochhammer{1}{n} \right)_{n=0}^{\infty} 
     && \seqmapsto{A047053} 
     \left(1, 4, 32, 384, 6144, 122880, \ldots \right) \\ 
\left( \AlphaFactorial{4n+1}{4} \right)_{n=0}^{\infty} & = 
     \left( 4^{n} \cdot \Pochhammer{5/4}{n} \right)_{n=0}^{\infty} 
     && \seqmapsto{A007696} 
     \left(1, 5, 45, 585, 9945, 208845, \ldots \right). 
\end{align*} 
For each $n \in \mathbb{N}$ and prescribed constants 
$r, c \in \mathbb{Z}$ defined such that $c \mid n+r$, we also obtain 
rational convergent-based generating functions enumerating the modified 
single factorial function sequences given by 
\begin{align} 
\label{eqn_PPlusROverAlpha_FactFn_Conv_Ident-stmt_v2} 
\mathlarger{\mathsmaller{\left(\frac{n+r}{c}\right)\mathlarger{!}}} & = 
     [z^{n}] 
     \ConvGF{h}{-c}{n+r}{\frac{z}{c}} + 
     \Iverson{\frac{r}{c} = 0} \Iverson{n = 0}, 
     \text{ $\forall$ $h \geq \lfloor (n+r) / c \rfloor$}. 
\end{align} 

\subsubsection{Generating functions for 
               multi-valued integer product sequences} 

Likewise, 
given any $n \geq 1$ and fixed $\alpha \in \mathbb{Z}^{+}$, 
we can enumerate the somewhat less obvious full forms of the 
generalized $\alpha$-factorial function sequences defined piecewise for the 
distinct residues, $n \in \{0,1,\ldots, \alpha -1\}$, modulo $\alpha$ by 
\eqref{eqn_nAlpha_Multifact_variant_rdef} and in 
\eqref{eqn_MultFactFn_ConvSeq_def_v1}. 
The multi-valued products defined by 
\eqref{eqn_GenFact_product_form} for these functions are 
generated as follows: 
\begin{align*} 
\tagonce\label{eqn_MultFactFn_ConvSeq_def_v2} 
n!_{(\alpha)} & = 
     \left[z^{\lfloor (n+\alpha-1) / \alpha \rfloor}\right] 
     \ConvGF{n}{-\alpha}{n}{z} \\ 
     & = 
     [z^{n}] \left( 
     \sum_{0 \leq d < \alpha} z^{-d} \cdot 
     \ConvGF{n}{\alpha}{\alpha - d}{z^{\alpha}} 
     \right) \\ 
\tagonce\label{eqn_MultFactFn_ConvSeq_def_v3} 
    & = 
     [z^{n+\alpha-1}] \left( 
     \frac{1-z^{\alpha}}{1-z} \times 
     \ConvGF{n}{-\alpha}{n}{z^{\alpha}} 
     \right). 
\end{align*} 
The complete sequences over the multi-valued symbolic products 
formed by the special cases of the 
double factorial function, the \emph{triple factorial} function, $n!!!$, the 
\emph{quadruple factorial} function, 
$n!!!! = \MultiFactorial{n}{4}$, the 
\emph{quintuple factorial} (\emph{$5$-factorial}) function, 
$\MultiFactorial{n}{5}$, and the 
\emph{$6$-factorial} function, $\MultiFactorial{n}{6}$, 
respectively, are generated by the convergent generating function 
approximations expanded in the next equations 
(\seqnum{A006882}, \seqnum{A007661}, \seqnum{A007662}, \seqnum{A085157},
\seqnum{A085158}). 
\begin{align*} 
\left( n!! \right)_{n=1}^{\infty} & = \left( 
     \left[z^{\lfloor (n+1) / 2 \rfloor}\right] 
     \ConvGF{n}{-2}{n}{z} 
     \right)_{n=1}^{\infty} && 
     \seqmapsto{A006882} 
     \left(1, 2, 3, 8, 15, 48, 105, 384, \ldots\right) \\ 
\left( n!!! \right)_{n=1}^{\infty} & = \left( 
     \left[z^{\lfloor (n+2) / 3 \rfloor}\right] 
     \ConvGF{n}{-3}{n}{z} 
     \right)_{n=1}^{\infty} && 
     \seqmapsto{A007661} 
     \left(1, 2, 3, 4, 10, 18, 28, 80, 162, \ldots\right) \\ 
\left( \MultiFactorial{n}{4} \right)_{n=1}^{\infty} & = 
     \left( 
     \left[z^{\lfloor (n+3) / 4 \rfloor}\right] 
     \ConvGF{n}{-4}{n}{z} 
     \right)_{n=1}^{\infty} && 
     \seqmapsto{A007662} 
     \left(1, 2, 3, 4, 5, 12, 21, 32, 45, \ldots\right) \\ 
\left( \MultiFactorial{n}{5} \right)_{n=1}^{\infty} & = 
     \left( 
     \left[z^{\lfloor (n+4) / 5 \rfloor}\right] 
     \ConvGF{n}{-5}{n}{z} 
     \right)_{n=1}^{\infty} && 
     \seqmapsto{A085157} 
     \left(1, 2, 3, 4, 5, 6, 14, 24, 36, 50, \ldots\right) \\ 
\left( \MultiFactorial{n}{6} \right)_{n=1}^{\infty} & = 
     \left( 
     \left[z^{\lfloor (n+5) / 6 \rfloor}\right] 
     \ConvGF{n}{-6}{n}{z} 
     \right)_{n=1}^{\infty} && 
     \seqmapsto{A085158} 
     \left(1, 2, 3, 4, 5, 6, 7, 16, 27, 40, \ldots\right) 
\end{align*} 

\subsubsection{Examples of new convergent-based generating function identities 
               for binomial coefficients} 

The rationality in $z$ of the generalized convergent functions, 
$\ConvGF{h}{\alpha}{R}{z}$, for all $h \geq 1$ also provides 
several of the new forms of generating function identities for many 
factorial-related product sequences and 
related expansions of the binomial coefficients that are easily 
proved from the 
diagonal-coefficient, or Hadamard product, generating function 
results established in 
Section \ref{subSection_DiagonalGFSequences_Apps}. 
For example, for natural numbers $n \geq 1$, the next variants of the 
binomial-coefficient-related product sequences are enumerated by the 
following coefficient identities 
(\seqnum{A009120}, \seqnum{A001448}): 
\begin{align*} 
\tagonce\label{eqn_HybridDiagCoeffHPGFs_BinomCoeff_Examples-exps_v1} 
\frac{(4n)!}{(2n)!} & = 
     \frac{4^{4n} 
           \bcancel{\Pochhammer{1}{n}} \bcancel{\Pochhammer{\frac{2}{4}}{n}} 
           \Pochhammer{\frac{1}{4}}{n} \Pochhammer{\frac{3}{4}}{n}}{ 
           2^{2n} \bcancel{\Pochhammer{1}{n}} 
           \bcancel{\Pochhammer{\frac{1}{2}}{n}}} \\ 
     & = 
     4^{n} \times 4^{n} \Pochhammer{1/4}{n} \times 4^{n} \Pochhammer{3/4}{n} \\ 
     & = 
     [x_1^0 z^n]\left( 
     \ConvGF{n}{4}{3}{\frac{4z}{x_1}} \ConvGF{n}{4}{1}{x_1} 
     \right) \\ 
     & = 
     4^{n} \times \AlphaFactorial{4n-3}{4} \AlphaFactorial{4n-1}{4} \\ 
     & = 
     [x_1^0 z^n]\left( 
     \ConvGF{n}{-4}{4n-3}{\frac{4z}{x_1}} \ConvGF{n}{-4}{4n-1}{x_1} 
     \right) \\ 
\binom{4n}{2n} & = 
     [x_1^0 x_2^0 z^n]\Biggl( 
     \ConvGF{n}{4}{3}{\frac{4z}{x_2}} \ConvGF{n}{4}{1}{\frac{x_2}{x_1}} \times 
     \undersetbrace{\widehat{E}_2(x_1) = E_{2,1}(x_1)}{ 
     \cosh\left(\sqrt{x_1}\right) 
     } 
     \quad\Biggr) \\ 
     & = 
     [x_1^0 x_2^0 z^n]\Biggl( 
     \ConvGF{n}{-4}{4n-3}{\frac{4z}{x_2}} 
     \ConvGF{n}{-4}{4n-1}{\frac{x_2}{x_1}} \times 
     \undersetbrace{E_{2,1}(x_1)}{ 
     \cosh\left(\sqrt{x_1}\right) 
     } 
     \quad\Biggr). 
\end{align*} 
The examples given in 
Section \ref{subsubSection_remark_HybridDiagonalHPGFs} 
provide examples of related constructions 
of the hybrid rational convergent-based generating function products that 
generate the central binomial coefficients and 
several other notable cases of related sequence expansions. 
We can similarly generate the sequences of binomials, 
$(a+b)^n$ and $c^n-1$ for fixed non-zero $a,b,c \in \mathbb{R}$, 
which we consider in 
Section \ref{subsubSection_Apps_Example_SumsOfPowers_Seqs}, using the 
binomial theorem and a rational convergent-based approximation to the formal 
Laplace-Borel transform as follows: 
\begin{align*} 
(a+b)^n & = n! \times \sum_{k=0}^{n} \frac{a^k}{k!} \cdot 
     \frac{b^{n-k}}{(n-k)!} \\ 
     & = 
     [x^0][z^n] \ConvGF{n}{1}{1}{\frac{z}{x}} e^{(a+b) x} \\ 
     & = 
     [x^0][z^n] \ConvGF{n}{-1}{n}{\frac{z}{x}} e^{(a+b) x} \\ 
c^n-1 & = n! \times \sum_{k=0}^{n-1} \frac{(c-1)^{k+1}}{(k+1)!} \cdot 
     \frac{1}{(n-1-k)!} \\ 
     & = 
     [x^0][z^n] \ConvGF{n}{1}{1}{\frac{z}{x}} \left(e^{(c-1)x} - 1\right) e^x \\ 
     & = 
     [x^0][z^n] \ConvGF{n}{-1}{n}{\frac{z}{x}} \left(e^{(c-1)x} - 1\right) e^x. 
\end{align*} 

\subsection{Examples of new congruences for the 
            $\alpha$-factorial functions, the 
            Stirling numbers of the first kind, and the 
            $r$-order harmonic number sequences} 
\label{subsubSection_Examples_NewCongruences} 

\subsubsection{Congruences for the $\alpha$-factorial functions modulo $2$} 

For any fixed $\alpha \in \mathbb{Z}^{+}$ and natural numbers $n \geq 1$, the 
generalized multiple, $\alpha$-factorial functions, $n!_{(\alpha)}$, 
defined by \eqref{eqn_nAlpha_Multifact_variant_rdef} 
satisfy the following congruences modulo $2$ (and $2\alpha$): 
\begin{align} 
\notag 
n!_{(\alpha)} & \equiv 
     \frac{n}{2} \left(\left(n-\alpha + \sqrt{\alpha 
     (\alpha -n)}\right)^{\left\lfloor \frac{n-1}{\alpha }\right\rfloor } + 
     \left(n-\alpha - \sqrt{\alpha (\alpha -n)}\right)^{\left\lfloor 
     \frac{n-1}{\alpha }\right\rfloor }\right) \pmod{2, 2\alpha} \\ 
\label{eqn_cor_Congruences_for_AlphaFactFns_modulo2} 
     & = 
     [z^n] \left( 
     \frac{(z^{\alpha}-1)(1+(2\alpha-n)z^\alpha)}{z(1-z)\left( 
     (\alpha-n)(n \cdot z^{\alpha} - 2) z^{\alpha}-1\right)} 
     \right). 
\end{align} 
Given that the definition of the single factorial function implies that 
$n! \equiv 0 \pmod{2}$ whenever $n \geq 2$, the statement of 
\eqref{eqn_cor_Congruences_for_AlphaFactFns_modulo2} 
provides somewhat less obvious results for the 
generalized $\alpha$-factorial function sequence cases when $\alpha \geq 2$. 
\tableref{table245}
provides specific listings of the result in 
\eqref{eqn_cor_Congruences_for_AlphaFactFns_modulo2} satisfied by the 
$\alpha$-factorial functions, $\MultiFactorial{n}{\alpha}$, 
for $\alpha \defequals 1, 2, 3, 4$. 
The corresponding, closely-related new forms of 
congruence properties satisfied by 
these functions expanded by exact algebraic formulas 
modulo $3$ ($3\alpha$) and modulo $4$ ($4\alpha$) are also cited 
as special cases in the examples given in the next subsection 
(see Section \ref{subsubSection_GenpnAlphaR_Congr_Mod234}). 

\subsubsection{New forms of congruences for the $\alpha$-factorial functions 
               modulo $3$, modulo $4$, and modulo $5$} 

To simplify notation, we first define the next shorthand for the respective 
(distinct) roots, $r_{p,i}^{(\alpha)}(n)$ for $1 \leq i \leq p$, 
corresponding to the special cases of the convergent denominator 
functions, $\ConvFQ{p}{\alpha}{R}{z}$, 
factorized over $z$ for any fixed integers $n, \alpha \geq 1$ 
when $p \defequals 3, 4, 5$ 
\cite[\S 1.11(iii); \cf \S 4.43]{NISTHB}: 
\begin{align} 
\notag 
\left( r_{3,i}^{(\alpha)}(n) \right)_{i=1}^{3} & \defequals 
     \bigl\{ z_i : z_i^3 -3 z_i^2 (2 \alpha +n)+3 z_i (\alpha +n) 
     (2 \alpha +n) \\ 
\notag 
     & \phantom{\defequals \bigl\{ z_i : z_i^3\ } - 
     n (\alpha +n) (2 \alpha +n) = 0,\ 
     1 \leq i \leq 3 \bigr\} \\ 
\notag 
\left( r_{4,j}^{(\alpha)}(n) \right)_{j=1}^{4} & \defequals 
     \bigl\{ z_j : 
     z_j^4 - 4 z_j^3 (3 \alpha +n) + 
     6 z_j^2 (2 \alpha +n) (3 \alpha +n) - 
     4 z_j (\alpha +n) (2 \alpha +n) (3 \alpha +n) \\ 
\notag 
     & \phantom{\defequals \bigl( z_j : z_j^4\ } + 
     n (\alpha +n) (2 \alpha +n) (3 \alpha +n) = 0,\ 
     1 \leq j \leq 4 \bigr\} \\ 
\notag 
\left( r_{5,k}^{(\alpha)}(n) \right)_{k=1}^{5} & \defequals 
     \bigl\{ z_k : 
     z_k^5 - 5 (4 \alpha + n) z_k^4 + 
     10 (3 \alpha + n) (4 \alpha + n) z_k^3 \\ 
\notag 
     & \phantom{\defequals \bigl( z_k : z_k^5\ } - 
     10 (2 \alpha + n) (3 \alpha + n) (4 \alpha + n) z_k^2 \\ 
\notag 
     & \phantom{\defequals \bigl( z_k : z_k^5\ } + 
     5 (\alpha + n) (2 \alpha + n) (3 \alpha + n) (4 \alpha + n) z_k \\ 
\label{eqn_cor_AlphaFactMod3_orig_roots_eqn} 
     & \phantom{\defequals \bigl( z_k : z_k^5\ } - 
     n (\alpha + n) (2 \alpha + n) (3 \alpha + n) (4 \alpha + n) = 0,\ 
     1 \leq k \leq 5 \bigr\}. 
\end{align} 
Similarly, we define the following functions for any fixed 
$\alpha \in \mathbb{Z}^{+}$ and $n \geq 1$ to 
simplify the notation in stating next the congruences in 
\eqref{eqn_AlphaFactFnModulo3_congruence_stmts} below: 
\newcommand{\rootri}[2]{\ensuremath{r_{#1,#2}^{(-\alpha)}(n)}} 
\begin{align*} 
\widetilde{R}_3^{(\alpha)}(n) & \defequals 
     \frac{\left(6 \alpha ^2+\alpha 
     \left(6 \rootri{3}{1}-4 n\right)+\left(n-\rootri{3}{1}
     \right){}^2\right) 
     \rootri{3}{1}{}^{\left\lfloor \frac{n-1}{\alpha }\right\rfloor+1}}{ 
     \left(\rootri{3}{1}-\rootri{3}{2}\right) 
     \left(\rootri{3}{1}-\rootri{3}{3}\right)} \\ 
   & \phantom{\equiv\ } \quad + 
     \frac{\left(6 \alpha ^2+\alpha  \left(6 \rootri{3}{3}-4 n\right)+ 
     \left(n-\rootri{3}{3}\right){}^2\right)
      \rootri{3}{3}{}^{\left\lfloor \frac{n-1}{\alpha }\right\rfloor +1}}{ 
      \left(\rootri{3}{3}-\rootri{3}{1}\right) 
      \left(\rootri{3}{3}-\rootri{3}{2}\right)} \\ 
   & \phantom{\equiv\ } \quad + 
     \frac{\left(6 \alpha ^2+\alpha \left(6 \rootri{3}{2}-4 n\right) + 
     \left(n-\rootri{3}{2}\right){}^2\right) \rootri{3}{2}{}^{\left\lfloor 
     \frac{n-1}{\alpha }\right\rfloor +1}}{ 
     \left(\rootri{3}{2}-\rootri{3}{1}\right)
     \left(\rootri{3}{2}-\rootri{3}{3}\right)} \\ 
C_{4,i}^{(\alpha)}(n) & \defequals 
     24 \alpha^3-18 \alpha ^2 \left(n-2 \cdot \rootri{4}{i} 
     \right)+\alpha  \left(7 n-12 \cdot \rootri{4}{i}\right) 
     \left(n-\rootri{4}{i}\right) \\ 
   & \phantom{\defequals 24 \alpha^3\ } - 
     \left(n-\rootri{4}{i}\right)^3,\ 
     \text{ for } 
     1 \leq i \leq 4 \\ 
C_{5,k}^{(\alpha)}(n) & \defequals 
     120 \alpha^4 + 2 \alpha^2 \left(23 n^2-79 n \cdot \rootri{5}{k} + 
     60 \cdot \rootri{5}{k}^2\right) + 
     48 \alpha ^3 (2n - 5 \cdot \rootri{5}{k}) \\ 
     & \phantom{\defequals 120 \alpha^4\ } + 
     \alpha  (11 n-20 \rootri{5}{k}) (n-\rootri{5}{k})^2 + 
     (n-\rootri{5}{k})^4,\ 
     \text{ for } 
     1 \leq k \leq 5. 
\end{align*} 
For fixed $\alpha \in \mathbb{Z}^{+}$ and $n \geq 0$, 
we obtain the following analogs to the first congruence result modulo $2$ 
expanded by \eqref{eqn_cor_Congruences_for_AlphaFactFns_modulo2} 
for the $\alpha$-factorial functions, $n_{(\alpha)}$, when $n \geq 1$ 
(see Section \ref{subsubSection_GenpnAlphaR_Congr_Mod234}): 
\begin{align} 
\label{eqn_AlphaFactFnModulo3_congruence_stmts} 
n!_{(\alpha)} & \equiv \widetilde{R}_3^{(\alpha)}(n) && \pmod{3, 3\alpha} \\ 
\notag 
n!_{(\alpha)} & \equiv 
     \undersetbrace{ \defequals R_4^{(\alpha)}(n)}{
     \sum\limits_{1 \leq i \leq 4} 
     \frac{C_{4,i}^{(\alpha)}(n)}{\prod\limits_{j \neq i} 
     \left(\rootri{4}{i} - \rootri{4}{j}\right)} 
     \rootri{4}{i}^{\left\lfloor \frac{n+\alpha-1}{\alpha} \right\rfloor} 
     } 
     && \pmod{4, 4\alpha} \\ 
\notag 
n!_{(\alpha)} & \equiv 
     \undersetbrace{ \defequals R_5^{(\alpha)}(n)}{ 
     \sum\limits_{1 \leq k \leq 5} 
     \frac{C_{5,k}^{(\alpha)}(n)}{\prod\limits_{j \neq k} 
     \left(\rootri{5}{k} - \rootri{5}{j}\right)} 
     \rootri{5}{k}^{\left\lfloor \frac{n+\alpha-1}{\alpha} \right\rfloor} 
     } 
     && \pmod{5, 5\alpha}. 
\end{align} 
Several particular concrete examples illustrating the 
results cited in 
\eqref{eqn_cor_Congruences_for_AlphaFactFns_modulo2} 
modulo $2$ (and $2\alpha$), and in 
\eqref{eqn_AlphaFactFnModulo3_congruence_stmts} 
modulo $p$ (and $p\alpha$) for $p \defequals 3, 4, 5$, 
corresponding to the first few cases of 
$\alpha \geq 1$ and $n \geq 1$ appear in 
\tableref{table245}
Further computations of the congruences given in 
\eqref{eqn_AlphaFactFnModulo3_congruence_stmts} 
modulo $p\alpha^{i}$ (for some $0 \leq i \leq p$) are contained in the 
\Mm{} summary notebook 
included as a supplementary file with the 
submission of this article 
(see Section \ref{subSection_MmSummaryNBInfo} and the 
reference document \cite{SUMMARYNBREF-STUB}). 
The results in 
Section \ref{subSection_NewCongruence_Relations_Modulo_Integer_Bases} 
provide statements of these new integer congruences 
for fixed $\alpha \neq 0$ modulo any integers $p \geq 2$. 
The analogous formulations of 
the new relations for the factorial-related product sequences 
modulo any $p$ and $p\alpha$ 
are easily established for the subsequent cases of integers $p \geq 6$ 
from the partial fraction expansions of the 
convergent functions, $\ConvGF{h}{\alpha}{R}{z}$, 
cited in the particular listings in 
\tableref{table_SpCase_Listings_Of_PhzQhz_ConvFn} and in 
\tableref{table_RelfectedConvNumPolySeqs_sp_cases}, and 
through the generalized rational convergent function properties proved in 
Section \ref{Section_Props_Of_CFracExps_OfThe_GenFactFnSeries}. 

\subsubsection{New congruence properties for the 
               Stirling numbers of the first kind} 

The results given in 
Section \ref{subSection_NewCongruence_Relations_Modulo_Integer_Bases} also 
provide new congruences for the generalized Stirling number triangles 
in \eqref{eqn_Fa_rdef}, as well as several new forms of rational generating 
functions that enumerate the scaled factorial-power variants of the 
\emph{$r$-order harmonic numbers}, 
$(n!)^{r} \times H_n^{(r)}$, 
modulo integers $p \geq 2$ \cite[\S 6.3]{GKP} 
(see Section \ref{subsubSection_remark_New_Congruences_for_GenS1Triangles_and_HNumSeqs}). 
For example, the known congruences for the 
Stirling numbers of the first kind proved by the 
generating function techniques enumerated in the reference 
\cite[\S 4.6]{GFOLOGY} imply the next new 
congruence results satisfied by the binomial coefficients modulo $2$ 
(\seqnum{A087755}). 
\begin{align*} 
\binom{\lfloor \frac{n}{2} \rfloor}{m - \lceil \frac{n}{2} \rceil} & 
     \Iverson{1 \leq m \leq 6} & \\ 
     & \equiv 
     \Iverson{n > m} \times 
     \begin{cases} 
     \frac{2^{n}}{4}, & 
     \text{ if $m = 1$; } \\ 
     \frac{3 \cdot 2^{n}}{16} (n-1), & 
     \text{ if $m = 2$; } \\ 
     \frac{2^{n}}{128} (9n-20) (n-1), & 
     \text{ if $m = 3$; } \\ 
     \frac{2^{n}}{512} (3n-10) (3n-7) (n-1), & 
     \text{ if $m = 4$; } \\ 
     \frac{2^{n}}{8192} (27n^3-279n^2+934n-1008) (n-1), & 
     \text{ if $m = 5$; } \\ 
     \frac{2^{n}}{163840} (9n^2-71n+120) (3n-14) (3n-11) (n-1), & 
     \text{ if $m = 6$; } \\ 
     0, & \text{ otherwise. } 
     \end{cases} \\ 
     & \phantom{\qquad 1 } + 
     \Iverson{1 \leq m \leq 6} \Iverson{n = m} 
     \pmod{2}
\end{align*} 

\subsubsection{New congruences and rational generating functions for the 
               $r$-order harmonic numbers} 
The next results state several additional new congruence properties 
satisfied by the 
first-order, second-order, and third-order harmonic number sequences, 
each expanded by the rational generating functions enumerating these 
sequences modulo the first few small cases of integer-valued $p$ 
constructed from the generalized convergent functions in 
Section \ref{subsubSection_remark_New_Congruences_for_GenS1Triangles_and_HNumSeqs} (\seqnum{A001008}, \seqnum{A002805}, \seqnum{A007406},
\seqnum{A007407}, \seqnum{A007408}, \seqnum{A007409}). 
\begin{align*} 
(n!)^{3} \times H_n^{(3)} 
     & \equiv 
     [z^{n}] \left( 
     \mathsmaller{ 
     \frac{z \left(1-7 z+49 z^2-144 z^3+192 z^4\right)}{(1-8z)^2} 
     } 
     \right) 
     && \hspace{-6mm} \pmod{2} \\ 
(n!)^{2} \times H_n^{(2)} 
     & \equiv 
     [z^{n}] \left( 
     \mathsmaller{ 
     \frac{z \left(1-61 z+1339 z^2-13106 z^3+62284 z^4-144264 z^5+ 
     151776 z^6-124416 z^7+41472 z^8\right)}{(1-6 z)^3 
     \left(1-24 z+36 z^2\right)^2}
     } 
     \right) 
     && \hspace{-6mm} \pmod{3} \\ 
(n!) \times H_n^{(1)} 
     & \equiv 
     [z^{n+1}] \left( 
     \mathsmaller{ 
     \frac{36 z^2 - 48z + 325}{576} + 
     \frac{17040 z^2+1782 z+6467}{576 \left(24 z^3-36 z^2+12 z-1\right)}+\frac{78828 z^2-33987 z+3071}{288 \left(24
        z^3-36 z^2+12 z-1\right)^2} 
     } 
     \right) && \hspace{-6mm} \pmod{4} \\ 
     & \equiv 
     [z^{n\phantom{-1}}]\left( 
     \mathsmaller{ 
     \frac{3z-4}{48} + 
     \frac{1300 z^2+890 z+947}{96 \left(24 z^3-36 z^2+12 z-1\right)}+\frac{24568 z^2-10576 z+955}{96 \left(24 z^3-36 z^2+12 z-1\right)^2}
     } 
     \right) && \hspace{-6mm} \pmod{4} \\ 
     & \equiv 
     [z^{n-1}] \left( 
     \mathsmaller{ 
     \frac{1}{16} + 
     \frac{-96 z^2+794 z+397}{48 \left(24 z^3-36 z^2+12 z-1\right)}+\frac{5730 z^2-2453 z+221}{24 \left(24 z^3-36 z^2+12 z-1\right)^2} 
     } 
     \right) && \hspace{-6mm} \pmod{4}. 
\end{align*} 

\section{The Jacobi-type J-fractions for 
         generalized factorial function sequences} 
\label{Section_Proofs_of_the_GenCFracReps} 

\subsection{Enumerative properties of 
            Jacobi-type J-fractions}
\label{subSection_EnumProps_of_JFractions} 

To simplify the exposition in this article, we adopt the 
notation for the Jacobi-type continued fractions, or J-fractions, 
in 
(\cite{FLAJOLET80B,FLAJOLET82}, \cite[\cf \S 3.10]{NISTHB}, \cite[\cf \S 5.5]{GFLECT}). 
Given some application-specific choices of the prescribed 
sequences, $\{a_k, b_k, c_k\}$, we consider the 
formal power series whose coefficients are generated by the 
rational convergents, $J^{[h]}(z) \defequals J^{[h]}(\{a_k, b_k, c_k\}; z)$, 
to the infinite continued fractions, 
$J(z) \defequals J^{[\infty]}(\{a_k, b_k, c_k\}; z)$, defined as follows: 
\begin{align} 
\label{eqn_CF_exp_general_form} 
J(z) & = 
     \cfrac{1}{1-c_0z-\cfrac{a_0b_1 z^2}{1-c_1z- 
     \cfrac{a_1b_2 z^2}{\cdots}.}} 
\end{align} 
We briefly summarize the other enumerative properties from the 
references that are 
relevant in constructing the new factorial-function-related results 
given in the following subsections 
(\cite{FLAJOLET82,FLAJOLET80B}, \cite[\S 5.5]{GFLECT}, \cite[\cf \S 6.7]{GKP}). 

\begin{enumerate} 

\item \itemlabel{Definitions of the $h$-order convergent series} 
When $h \geq 1$, the $h^{th}$ convergent functions, 
given by the equivalent notation of 
$J^{[h]}(z)$ and $J^{[h]}(\{a_k, b_k, c_k\}; z)$ within this section, 
of the infinite continued fraction expansions in 
\eqref{eqn_CF_exp_general_form} 
are defined as the ratios, $J^{[h]}(z) \defequals P_h(z) / Q_h(z)$. 

The component functions corresponding to the convergent 
numerator and denominator sequences, $P_h(z)$ and $Q_h(z)$, 
each satisfy second-order 
finite difference equations (in $h$) of the respective forms 
defined by the next two equations. 
\begin{align*} 
P_h(z) & = (1-c_{h-1} z) P_{h-1}(z)-a_{h-2}b_{h-1} z^2 P_{h-2}(z) + 
           \Iverson{h = 1} \\ 
Q_h(z) & = (1-c_{h-1} z) Q_{h-1}(z)-a_{h-2}b_{h-1} z^2 Q_{h-2}(z) + 
           (1-c_0 z) \Iverson{h = 1}+\Iverson{h=0} 
\end{align*} 

\item \itemlabel{Rationality of truncated convergent function approximations} 
Let $p_n = p_n(\{a_k, b_k, c_k\}) \defequals [z^n] J(z)$ 
denote the expected term corresponding to the coefficient of $z^n$ in the 
formal power series expansion defined by the infinite J-fraction from 
\eqref{eqn_CF_exp_general_form}. 
For all $n \geq 0$, we know that the $h^{th}$ convergent functions 
have truncated power series expansions that satisfy 
\begin{equation*} 
p_n(\{a_k, b_k, c_k\}) = [z^n] J^{[h]}(\{a_k, b_k, c_k\}; z),\ 
     \forall n \leq h. 
\end{equation*} 
In particular, the series coefficients of the $h^{th}$ convergents are 
always at least $h$-order accurate as formal power series 
expansions in $z$ that exactly enumerate the expected sequence terms, 
$\left(p_n\right)_{n \geq 0}$. 

The resulting \quotetext{\emph{eventually periodic}} nature 
suggested by the approximate sequences enumerated by 
the rational convergent functions in $z$ is formalized in the 
congruence properties given below in 
\eqref{eqn_EnumProps_Of_JTypeCFracs_congruence_rels-stmt_v1} 
(\cite{FLAJOLET82}, \cite[See \S 2, \S 5.7]{GFLECT}). 

\item \itemlabel{Congruence properties modulo integer bases} 
Let $\lambda_k \defequals a_{k-1} b_{k}$ and suppose that the 
corresponding bases, $M_h$, are formed by the products 
$M_h \defequals \lambda_1 \lambda_2 \cdots \lambda_h$ for $h \geq 1$. 
Whenever $M_h \in \mathbb{Z}$, $N_h \mid M_h$, and $n \geq 0$, we have that 
\begin{equation} 
\label{eqn_EnumProps_Of_JTypeCFracs_congruence_rels-stmt_v1} 
p_n(\{a_k, b_k, c_k\}) \equiv [z^n] J^{[h]}(\{a_k, b_k, c_k\}; z) \pmod{N_h}, 
\end{equation} 
which is also true of all partial sequence terms enumerated by the 
$h^{th}$ convergent functions modulo any integer divisors of the $M_h$ 
\cite[\cf \S 5.7]{GFLECT}. 

\end{enumerate} 

\subsection{A short direct proof of the J-fraction representations for the 
            generalized product sequence generating functions} 
\label{subSection_GenCFrac_Reps_for_GenFactFns} 

We omit the details to a more combinatorially flavored proof that the 
J-fraction series defined by the convergent functions in 
\eqref{eqn_ConvGF_notation_def} do, in fact, correctly enumerate the expected 
symbolic product sequences in \eqref{eqn_GenFact_product_form}. 
Instead, a short direct proof following from the J-fraction results 
given in Flajolet's first article is sketched below. 
Even further combinatorial interpretations of the sequences generated by 
these continued fraction series, 
their relations to the Stirling number triangles, and other 
properties tied to the coefficient triangles studied in depth by the article 
\cite{MULTIFACTJIS} based on the properties of these new 
J-fractions is suggested as a topic for later investigation. 

\begin{definition} 
The prescribed sequences in the J-fraction expansions defined by 
\eqref{eqn_CF_exp_general_form} in the previous section, corresponding to 
(\emph{i}) the 
Pochhammer symbol, $\Pochhammer{x}{n}$, and 
(\emph{ii}) the 
convergent functions enumerating the generalized products, 
$\pn{n}{\alpha}{R}$, or equivalently, the 
Pochhammer $k$-symbols, $\Pochhammer{R}{n,\alpha}$, 
over any fixed $\alpha \in \mathbb{Z}^{+}$ and 
indeterminate, $R$, are defined as follows: 
\begin{align*} 
\tagtext{Pochhammer Symbol} 
\{\as_k(x), \bs_k(x), \cs_k(x)\} & 
     \quad\overset{\text{ (i)}}{\defequals}\quad 
     \left\{ x+k, k, x+2k \right\} \\ 
\tagtext{Generalized Products} 
\{\af_k(\alpha, R), \bfcf_k(\alpha, R), \cfcf_k(\alpha, R)\} & 
     \quad\overset{\text{ (ii)}}{\defequals}\quad 
     \left\{ 
     R+k\alpha, k, \alpha \cdot (R+2k \alpha) 
     \right\}.  
\end{align*} 
\textbf{\textit{Claim}:} 
We claim that the modified J-fractions, $R_0(R / \alpha, \alpha z)$, 
that generate the corresponding series expansions in 
\eqref{eqn_PochhammerSymbol_InfCFrac_series_Rep_example-v1} 
enumerate the analogous terms of the generalized 
symbolic product sequences, $\pn{n}{\alpha}{R}$, 
defined by \eqref{eqn_GenFact_product_form}. 
\end{definition} 
\begin{proof}[Proof of the Claim.] 
First, an appeal to the polynomial expansions of both the 
Pochhammer symbol, $\Pochhammer{x}{n}$, and 
then of the products, $\pn{n}{\alpha}{R}$, 
defined by \eqref{eqn_GenFact_product_form} by the 
Stirling numbers of the first kind, yields the following sums: 
\begin{align} 
\notag 
p_n(\alpha, R) \times z^n & = 
     \left( 
     \prod_{j=0}^{n-1} (R+j\alpha) \Iverson{n \geq 1} + \Iverson{n = 0} 
     \right) \times z^n \\ 
\notag 
   & = 
     \left( 
     \sum_{k=0}^{n} \gkpSI{n}{k} \alpha^{n-k} R^{k} 
     \right) \times z^{n} \\ 
\label{eqn_gen_pnAlphaR_seq_PochhammerSymbol_exps} 
   & = 
     \undersetbrace{\Pochhammer{R}{n,\alpha} \times z^{n}}{ 
     \alpha^{n} \cdot \Pochhammer{\frac{R}{\alpha}}{n} 
     \times z^{n}
     }. 
\end{align} 
Finally, after a parameter substitution of 
$x \defmapsto R / \alpha$ together with the change of variable 
$z \defmapsto \alpha z$ in the first results from 
Flajolet's article \cite{FLAJOLET80B}, 
we obtain identical forms of the convergent-based definitions for the 
generalized J-fraction definitions given in 
Definition \ref{def_GenConvFns_PFact_Phz_eqn_QFact_Qhz-defs_intro_v1}. 
\end{proof} 

\subsection{Alternate exact expansions of the generalized convergent functions} 
\label{subSection_AltExps_of_the_GenConvFns} 

The notational inconvenience introduced in the inner sums of 
\eqref{eqn_AlphaFactFn_Exact_PartialFracsRep_v1} and 
\eqref{eqn_AlphaFactFn_Exact_PartialFracsRep_v2}, 
implicitly determined by the shorthand for the coefficients, 
$c_{h,j}(\alpha, R)$, 
each of which depend on the more difficult terms of the 
convergent numerator function sequences, $\FP_h(\alpha, R; z)$, 
is avoided in place of alternate recurrence relations for each 
finite $h^{th}$ convergent function, $\ConvGF{h}{\alpha}{R}{z}$, 
involving paired products of the 
denominator polynomials, $\FQ_h(\alpha, R; z)$. 
These denominator function sequences are related to generalized forms of the 
Laguerre polynomial sequences as follows 
(see Section \ref{subsubSection_Properties_Of_ConvFn_Qhz}): 
\begin{align} 
\label{eqn_PFact_Qhz_Uident_LPoly_exp-stmt_v0} 
\FQ_h(\alpha, R; z) & = 
     (-\alpha z)^{h} \cdot h! \cdot 
     L_h^{(R / \alpha - 1)}\left((\alpha z)^{-1}\right). 
\end{align} 
The expansions provided by the formula in 
\eqref{eqn_PFact_Qhz_Uident_LPoly_exp-stmt_v0} 
suggest useful alternate formulations of the 
congruence results given below in 
Section \ref{subSection_NewCongruence_Relations_Modulo_Integer_Bases} 
when the Laguerre polynomial, or corresponding 
confluent hypergeometric function, zeros 
are considered to be less complicated in form than the 
more involved sums expanded through the numerator functions, 
$\FP_h(\alpha, R; z)$. 

\begin{example}[Recurrence Relations for the Convergent Functions and Laguerre Polynomials] 
The well-known cases of the enumerative properties satisfied by the 
expansions of the convergent function sequences given in the references 
immediately yield the following relations 
(\cite[\S 3]{FLAJOLET80B}, \cite[\S 1.12(ii)]{NISTHB}): 
\begin{align} 
\notag 
\ConvGF{h}{\alpha}{R}{z} & = 
     \ConvGF{h-k}{\alpha}{R}{z} \\ 
\notag 
     & \phantom{= \quad } + 
     \sum_{i=0}^{k-1} 
     \frac{\alpha^{h-i-1} (h-i-1)! \cdot p_{h-i-1}(\alpha, R) \cdot 
     z^{2(h-i-1)}}{ 
     \FQ_{h-i}(\alpha, R; z) \FQ_{h-i-1}(\alpha, R; z)},\ 
     h > k \geq 1 \\ 
\label{eqn_AltExps_of_the_GenConvFns-rec_properties-stmts_v1} 
\ConvGF{h}{\alpha}{R}{z} & = 
     \sum_{i=0}^{h-1} 
     \frac{\alpha^{h-i-1} (h-i-1)! \cdot p_{h-i-1}(\alpha, R) \cdot 
     z^{2(h-i-1)}}{ 
     \FQ_{h-i}(\alpha, R; z) \FQ_{h-i-1}(\alpha, R; z)} \\ 
\notag 
     & = 
     \sum_{i=0}^{h-1} \binom{\frac{R}{\alpha}+i-1}{i} \times 
     \frac{(-\alpha z)^{-1}}{(i+1) \cdot 
     L_{i}^{(R / \alpha - 1)}\left((\alpha z)^{-1}\right) 
     L_{i+1}^{(R / \alpha - 1)}\left((\alpha z)^{-1}\right)},\ 
     h \geq 2. 
\end{align} 
The convergent function recurrences expanded by 
\eqref{eqn_AltExps_of_the_GenConvFns-rec_properties-stmts_v1} 
provide identities for particular cases of the generalized 
Laguerre polynomial sequences, $L_n^{(\beta)}(x)$, 
expressed as finite sums over paired reciprocals of the 
sequence at the same choices of the $x$ and $\beta$. 
\end{example} 

Topics aimed at finding new results obtained from other known, strictly 
continued-fraction-related properties of the 
convergent function sequences beyond the proofs given in 
Section \ref{Section_Props_Of_CFracExps_OfThe_GenFactFnSeries} 
are suggested as a further avenue to approach the 
otherwise divergent ordinary generating functions for these 
more general cases of the integer-valued 
factorial-related product sequences 
\cite[\cf \S 10]{HARDYWRIGHTNUMT}.

\begin{remark}[Related Convergent Function Expansions] 
For $h-i \geq 0$, 
we have a noteworthy Rodrigues-type formula 
satisfied by the Laguerre polynomial sequences in 
\eqref{eqn_PFact_Qhz_Uident_LPoly_exp-stmt_v0} 
stated by the reference as follows \cite[\S 18.5(ii)]{NISTHB}: 
\begin{align} 
\notag 
(-\alpha z)^{h-i} \cdot (h-i)! \times 
     L_{h-i}^{(\beta)}\left(\frac{1}{\alpha z}\right) & = 
     \alpha^{h-i} \cdot z^{2h-2i+\beta+1} e^{1 / \alpha z} \times 
     \left\{ 
     \frac{e^{-1 / \alpha z}}{z^{\beta+1}} 
     \right\}^{(h-i)}. 
\end{align} 
The multiple derivatives implicit to the statement in the previous 
equation then have the additional expansions through the 
product rule analog provided by the formula of Halphen 
from the reference given in the form of the following equation 
for natural numbers $n \geq 0$, where the 
notation for the functions, $F(z)$ and $G(z)$, employed in the 
previous equation corresponds to any prescribed choice of these functions 
that are each $n$ times continuously differentiable at the point $z \neq 0$ 
\cite[\S 3 Exercises, p.\ 161]{ADVCOMB}: 
\begin{align*} 
\tagtext{Halphen's Product Rule} 
\Biggl\{ F\left(\frac{1}{z}\right) G(z) \Biggr\}^{(n)} & = 
     \sum_{k=0}^{n} \binom{n}{k} \frac{(-1)^{k}}{z^{k}} \cdot 
     F^{(k)}\left(\frac{1}{z}\right) 
     \left\{ \frac{G(z)}{z^{k}} \right\}^{(n-k)}. 
\end{align*} 
The next particular restatements of 
\eqref{eqn_AltExps_of_the_GenConvFns-rec_properties-stmts_v1} 
then follow easily from the last two equations as 
\begin{align} 
\notag 
& \ConvGF{h}{\alpha}{R}{z} = 
     \ConvGF{h-k}{\alpha}{R}{z} \\ 
\notag 
     & \phantom{\Conv_h } + 
     \sum_{i=0}^{k-1} 
     \frac{(h-i-1)!}{(\alpha z) \cdot \Pochhammer{R / \alpha}{h-i}} \times 
     \frac{1}{_1F_1\left(-(h-i); \frac{R}{\alpha}; \frac{1}{\alpha z}\right) 
     {_1F_1}\left(-(h-i-1); \frac{R}{\alpha}; \frac{1}{\alpha z}\right)},\ 
     1 \leq k < h \\ 
\notag 
& \ConvGF{h}{\alpha}{R}{z} = 
     \sum_{i=0}^{h-1} 
     \frac{(h-i-1)!}{(\alpha z) \cdot \Pochhammer{R / \alpha}{h-i}} \times 
     \frac{1}{_1F_1\left(-(h-i); \frac{R}{\alpha}; \frac{1}{\alpha z}\right) 
     {_1F_1}\left(-(h-i-1); \frac{R}{\alpha}; \frac{1}{\alpha z}\right)} \\ 
\label{eqn_AltExps_of_the_GenConvFns-rec_properties-stmts_v2} 
& \phantom{\ConvGF{h}{\alpha}{R}{z}} = 
     \sum_{i=1}^{h} 
     \binom{\frac{R}{\alpha}+i-1}{i-1}^{-1} \times 
     \frac{(-R z)^{-1}}{
     {_1F_1}\left(1-i; \frac{R}{\alpha}; \frac{1}{\alpha z}\right) 
     {_1F_1}\left(-i; \frac{R}{\alpha}; \frac{1}{\alpha z}\right)}, 
\end{align} 
where the $h^{th}$ convergents, $\ConvGF{h}{\alpha}{R}{z}$, 
are rational in $z$ for all $h \geq 1$. 
The functions $_1F_1(a; b; z)$, or 
$\HypM{a}{b}{z} = \sum_{s \geq 0} \frac{(a)_s}{(b)_s s!} z^{s}$, 
in the previous equations denote 
\emph{Kummer's confluent hypergeometric function} 
(\cite[\S 13.2]{NISTHB},\cite[\S 5.5]{GKP}). 
\end{remark} 

\section{Properties of the generalized convergent functions} 
\label{Section_Props_Of_CFracExps_OfThe_GenFactFnSeries} 

We first focus on the comparatively simple factored expressions for the 
series coefficients of the denominator sequences in the 
results proved by 
Section \ref{subsubSection_Properties_Of_ConvFn_Qhz}. 
The identification of the convergent denominator functions as 
special cases of the confluent hypergeometric function 
yields additional identities providing 
analogous addition and multiplication theorems for these functions 
with respect to the parameter $z$, as well as a number of further, new 
recurrence relations derived from established relations 
stated in the references, 
such as those provided by Kummer's transformations. 
These properties form a superset of 
extended results beyond the immediate, more combinatorial, 
known relations for the J-fractions summarized in 
Section \ref{subSection_EnumProps_of_JFractions} and in 
Section \ref{subSection_AltExps_of_the_GenConvFns}. 

The numerator convergent sequences considered in 
Section \ref{subsubSection_Properties_Of_ConvFn_Phz} 
have less obvious 
expansions through special functions, or otherwise more well-known 
polynomial sequences. 
A point concerning the relative simplicity of the 
expressions of the denominator convergent polynomials 
compared to the numerator convergent sequences is 
also mentioned in Flajolet's article \cite[\S 3.1]{FLAJOLET80B}. 
The last characterization of the generalized convergent denominator 
functions by the Laguerre polynomials in 
Proposition \ref{prop_Closed-Form_Rep_for_ConvFn_Qhz} 
below provides the factorizations over the 
zeros of the classical orthogonal polynomial sequences studied in the 
references \cite{LGWORKS-ASYMP-SPFNZEROS2008,PROPS-ZEROS-CHYPFNS80} 
which are required to state the results provided by 
\eqref{eqn_AlphaFactFn_Exact_PartialFracsRep_v1} and 
\eqref{eqn_AlphaFactFn_Exact_PartialFracsRep_v2} from 
Section \ref{subSection_Intro_GenConvFn_Defs_and_Properties}, and 
more generally by 
\eqref{eqn_AlphaFactFnModulop_congruence_stmts} in 
Section \ref{subSection_Congruences_for_Series_ModuloIntegers_p}. 

\subsection{The convergent denominator function sequences} 
\label{subsubSection_Properties_Of_ConvFn_Qhz} 

In contrast to the convergent numerator functions, 
$\FP_h(\alpha, R; z)$, discussed next in 
Section \ref{subsubSection_Properties_Of_ConvFn_Phz}, the 
corresponding denominator functions, $\FQ_h(\alpha, R; z)$, 
are readily expressed through well-known special functions. 
The first several special cases given in 
\tableref{table_SpCase_Listings_Of_Qhz_ConvFn} 
suggest the next identity, which is proved following 
Proposition \ref{prop_Closed-Form_Rep_for_ConvFn_Qhz} below. 
\begin{align} 
\label{eqn_PFact_Qhz_product_ident} 
\FQ_h(\alpha, R; z) & = 
     \sum_{k=0}^{h} \binom{h}{k} (-1)^{k} \left(\prod_{j=0}^{k-1} 
     (R + (h-1-j)\alpha)\right) z^k 
\end{align} 
The convergent denominator functions are expanded by the 
\emph{confluent hypergeometric functions}, 
$\HypU{-h}{b}{w}$ and $\HypM{-h}{b}{w}$, 
and equivalently by the 
\emph{associated Laguerre polynomials}, $L_h^{(b-1)}(w)$, 
when $b \defmapsto R / \alpha$ and $w \defmapsto (\alpha z)^{-1}$ 
through the relations proved in the next proposition 
\cite[\cf \S 18.6(iv); \S 13.9(ii)]{NISTHB}. 

\begin{prop}[Exact Representations by Special Functions] 
\label{prop_Closed-Form_Rep_for_ConvFn_Qhz} 
The convergent denominator functions, $\FQ_h(\alpha, R; z)$, are 
expanded in terms of the confluent hypergeometric function and the 
associated Laguerre polynomials through the following results: 
\begin{align} 
\label{eqn_PFact_Qhz_Uident} 
\FQ_h(\alpha, R; z) 
     & = 
     \left(\alpha z\right)^{h} \times 
     \HypU{-h}{R / \alpha}{(\alpha z)^{-1}} \\ 
\label{eqn_PFact_Qhz_Mident} 
     & = 
     \left(-\alpha z\right)^{h} \Pochhammer{\frac{R}{\alpha}}{h} \times 
     \HypM{-h}{R / \alpha}{(\alpha z)^{-1}} \\ 
\label{eqn_PFact_Qhz_Uident_LPoly_exp-stmt_v1} 
   & = 
     (-\alpha z)^{h} \cdot h! \times 
     L_h^{(R / \alpha - 1)}\left((\alpha z)^{-1}\right). 
\end{align} 
\end{prop} 
\begin{proof} 
We proceed to prove the first identity in 
\eqref{eqn_PFact_Qhz_Uident} by induction. 
It is easy to verify by computation 
(see Table \ref{table_SpCase_Listings_Of_Qhz_ConvFn}) 
that the left-hand-side and right-hand-sides of 
\eqref{eqn_PFact_Qhz_Uident} coincide when $h = 0$ and $h = 1$. 
For $h \geq 2$, we apply the recurrence relation 
from \eqref{eqn_QFact_Qhz} to write the 
right-hand-side of \eqref{eqn_PFact_Qhz_Uident} as 
\begin{align} 
\label{eqn_FQhz_HyperFactU_proof_intermed_eqn-v1} 
\FQ_h(\alpha, R; z) & = (1-(R+2\alpha (h-1)) z) 
     \HypU{-h+1}{R / \alpha}{(\alpha z)^{-1}} (\alpha z)^{h-1} \\ 
\notag 
   & \phantom{= (1\ \ } - 
     \alpha (R+\alpha (h-2)) (h-1) z^2 
     \HypU{-h+2}{R / \alpha}{(\alpha z)^{-1}} (\alpha z)^{h-2}. 
\end{align} 
The proof is completed using the 
known recurrence relation for the confluent hypergeometric function 
stated in reference as \cite[\S 13.3(i)]{NISTHB} 
\begin{equation} 
\label{eqn_HyperFactU_proof_recurrence_ident} 
\HypU{-h}{b}{u} = (u-b-2(h-1)) \HypU{-h+1}{b}{u} - 
     (h-1) (b+h-2) \HypU{-h+2}{b}{u}. 
\end{equation} 
In particular, we can rewrite 
\eqref{eqn_FQhz_HyperFactU_proof_intermed_eqn-v1} as 
\begin{align} 
\notag 
\FQ_h(\alpha, R; z) & = 
     (\alpha z)^{h} \Biggl[\left((\alpha z)^{-1}-\left( 
     \frac{R}{\alpha} + 2(h-1)\right)\right) 
     \HypU{-h+1}{R / \alpha}{(\alpha z)^{-1}} \\ 
   & \phantom{= (\alpha z)^{h} \Biggl[ \Bigl(} - 
     \left(\frac{R}{\alpha}+h-2\right) (h-1) 
     \HypU{-h+2}{R / \alpha}{(\alpha z)^{-1}}\Biggr], 
\end{align} 
which implies \eqref{eqn_PFact_Qhz_Uident} in the special case of 
\eqref{eqn_HyperFactU_proof_recurrence_ident} where 
$(b, u) \defequals (R / \alpha, (\alpha z)^{-1})$. 
The second characterizations of $\FQ_h(\alpha, R; z)$ by 
Kummer's confluent hypergeometric function, $M(a, b, z)$, 
in \eqref{eqn_PFact_Qhz_Mident}, and by the 
Laguerre polynomials stated in 
\eqref{eqn_PFact_Qhz_Uident_LPoly_exp-stmt_v1}, follow from the 
first result whenever $h \geq 0$ \cite[\S 13.6(v), \S 18.11(i)]{NISTHB}. 
\end{proof} 

\begin{proof}[Proof of Equation \eqref{eqn_PFact_Qhz_product_ident}] 
The first identity for the denominator functions, $\FQ_h(\alpha, R; z)$, 
conjectured from the special case table listings by 
\eqref{eqn_PFact_Qhz_product_ident}, follows from the 
first statement of the previous proposition. 
We cite the particular expansions of $U(-n, b, z)$ when 
$n \geq 0$ is integer-valued involving the Pochhammer symbol, $(x)_n$, 
stated as follows \cite[\S 13.2(i)]{NISTHB}: 
\begin{align} 
\notag 
U(-n, b, z) & = \sum_{k=0}^{n} \binom{n}{k} (b+k)_{n-k} (-1)^n (-z)^k \\ 
\label{eqn_HyperFactU_proof_sum_ident-v2} 
   & = 
     \sum_{k=0}^{n} \binom{n}{k} (b+n-k)_{k} (-1)^{k} z^{n-k}. 
\end{align} 
The second sum for the confluent hypergeometric function given in 
\eqref{eqn_HyperFactU_proof_sum_ident-v2} 
then implies that the right-hand-side of 
\eqref{eqn_PFact_Qhz_Uident} can be expanded as follows: 
\begin{align*} 
\FQ_h(\alpha, R; z) & = 
     \left(\alpha z\right)^{h} 
     \HypU{-h}{R / \alpha}{(\alpha z)^{-1}} \\ 
   & = 
     (\alpha z)^h \sum_{k=0}^{h} \binom{h}{k} (-1)^k 
     \left(\frac{R}{\alpha} + h-k\right)_{k} (\alpha z)^{k-h} \\ 
   & = \phantom{(\alpha z)^h} 
     \sum_{k=0}^{h} \binom{h}{k} 
     \left(\frac{R}{\alpha} + h-k\right)_{k} (-\alpha z)^{k} \\ 
   & = \phantom{(\alpha z)^h} 
     \sum_{k=0}^{h} \binom{h}{k} 
     \underset{
     \mathlarger{(\pm 1)^{k} \pn{k}{\mp \alpha}{\pm R \pm (h-1) \alpha}}
     }{ 
     \underbrace{ 
     \left( 
     (-1)^{k} \times \prod_{j=0}^{k-1} \left(R+(h-1-j) \alpha \right) 
     \right) 
     } 
     } 
     z^{k}. 
\end{align*} 
The last line of previous equations 
provides the required expansion to complete a proof of the 
first identity cited in \eqref{eqn_PFact_Qhz_product_ident}. 
We obtain the alternate restatements of these 
coefficients of $z^k$ provided by the following equations for $p \geq 0$: 
\begin{align*} 
\tagonce\label{eqn_CoeffsOfzk_FQhaRz_restmts_by_GenProductSeqs-v1} 
[z^k] \FQ_p(\alpha, R; z) & = 
     \binom{p}{k} (-1)^{k} 
     p_k(-\alpha, R + (p-1) \alpha) 
     \cdot \Iverson{0 \leq k \leq p} \\ 
     & = 
     \binom{p}{k} \alpha^{k} \Pochhammer{1-p-R / \alpha}{k} 
     \cdot \Iverson{0 \leq k \leq p} \\ 
[z^k] \FQ_p(\alpha, R; z) & = 
     \binom{p}{k} p_k(\alpha, -R - (p-1) \alpha) 
     \cdot \Iverson{0 \leq k \leq p} \\ 
     & = 
     \binom{p}{k} (-\alpha)^{k} \FFactII{\left(R / \alpha + p-1\right)}{k} 
     \cdot \Iverson{0 \leq k \leq p}. 
\qedhere 
\end{align*} 
\end{proof} 

\begin{cor}[Recurrence Relations] 
\label{cor_HypU-Based_recurrences_for_FQhz} 
For $h \geq 0$ and any integers $s > -h$, the convergent denominator functions, 
$\FQ_h(\alpha, R; z)$, satisfy the reflection identity, or 
analog to Kummer's transformation for the 
confluent hypergeometric function, given by 
\begin{equation} 
\label{eqn_RecurrenceRelations_HypU_for_FQhz_stmt-v1} 
\FQ_h(\alpha, \alpha s; z) = \FQ_{h+s-1}(\alpha; \alpha (2-s); z). 
\end{equation} 
Additionally, for $h \geq 0$ these functions satisfy 
recurrence relations of the following forms: 
\begin{align} 
\label{eqn_RecurrenceRelations_HypU_for_FQhz_stmt-v2} 
(R+(h-1)\alpha) z \FP_h(\alpha, R-\alpha; z) + ((\alpha-R) z-1) 
     \FQ_h(\alpha, R; z) + \FQ_h(\alpha, R+\alpha; z) & = 0 \\ 
\notag 
\FQ_h(\alpha, R; z) + \alpha h z \FQ_{h-1}(\alpha, R; z) - 
     \FQ_h(\alpha, R-\alpha; z) & = 0 \\ 
\notag 
(R+\alpha h) z \FQ_h(\alpha, R; z) + \FQ_{h+1}(\alpha, R; z) - 
     \FQ_{h}(\alpha, R+\alpha; z) & = 0 \\ 
\notag 
(1-\alpha h z) \FQ_h(\alpha, R; z) - \FQ_h(\alpha, R+\alpha; z) - 
     \alpha h (R + (h - 1) \alpha) z^2 \FQ_{h-1}(\alpha, R; z) & = 0 \\ 
\notag 
(1-(h+1) \alpha z) \FQ_h(\alpha, R; z) - \FQ_{h+1}(\alpha, R; z) - 
     (R + (h-1) \alpha) z \FQ_h(\alpha, R-\alpha; z) & = 0 \\ 
\notag 
\alpha (h-1) z \FQ_{h-2}(\alpha, R+2\alpha; z) - (1-Rz) 
     \FQ_{h-1}(\alpha, R+\alpha; z) + \FQ_h(\alpha, R; z) & = \delta_{h,0}. 
\end{align} 
\end{cor} 
\begin{proof} 
The first equation results from \textit{Kummer's transformation} 
for the confluent hypergeometric function, $\HypU{a}{b}{z}$, 
given by \cite[\S 13.2(vii)]{NISTHB} 
\begin{equation*} 
\HypU{a}{b}{z} = z^{1-b} \HypU{a-b+1}{2-b}{z}. 
\end{equation*} 
In particular, when $R \defequals \alpha s$ and $h+s-1 \geq 0$ 
Proposition \ref{prop_Closed-Form_Rep_for_ConvFn_Qhz} 
implies that 
\begin{align*} 
\FQ_h(\alpha, R; z) & = \left(\alpha z\right)^{h+R/\alpha - 1} 
     \HypU{-(h+\frac{R}{\alpha}-1)}{\frac{2\alpha-R}{\alpha}}{ 
     (\alpha z)^{-1}} \\ 
   & = 
     \FQ_{h+R/\alpha - 1}(\alpha, 2\alpha-R; z). 
\end{align*} 
The recurrence relations stated in 
\eqref{eqn_RecurrenceRelations_HypU_for_FQhz_stmt-v2} follow 
similarly as consequences of the first proposition by 
applying the known results for the 
confluent hypergeometric functions cited in the reference 
\cite[\S 13.3(i)]{NISTHB}. 
\end{proof} 

\begin{prop}[Addition and Multiplication Theorems] 
\label{cor_HypU_fn_FiniteSums_Involving_FQhz} 
Let $z, w \in \mathbb{C}$ with $z \neq w$ and suppose that $z \neq 0$. 
For a fixed $\alpha \in \mathbb{Z}^{+}$ and $h \geq 0$, the 
following finite sums provide two 
addition theorem analogs satisfied by the 
sequences of convergent denominator functions: 
\begin{align*} 
\tagonce\label{eqn_FQhaRz_HypU_FiniteSum_Idents_stmt} 
\FQ_h(\alpha, R; z-w) & = 
     \sum_{n=0}^{h} \frac{(-h)_n (-w)^n (z-w)^{h-n}}{z^{h} \cdot n!} 
     \FQ_{h-n}(\alpha, R+\alpha n; z) \\ 
\FQ_h(\alpha, R; z-w) & = 
     \sum_{n=0}^{h} \frac{(-h)_n \left(1-h-\frac{R}{\alpha}\right)_n 
     (\alpha w)^{n}}{n!} 
     \FQ_{h-n}(\alpha, R; z) \\ 
     & = 
     \sum_{n=0}^{h} \binom{h}{n} \binom{h+\frac{R}{\alpha}-1}{n} 
     \times (\alpha w)^{n} n! \times 
     \FQ_{h-n}(\alpha, R; z). 
\end{align*} 
The corresponding multiplication theorems for the 
denominator functions are stated similarly for $h \geq 0$ 
in the forms of the following equations: 
\begin{align*} 
\tagonce\label{eqn_FQhaRz_HypU_FiniteSum_Idents-v_MultThms_stmts} 
\FQ_h(\alpha, R; zw) & = 
     \sum_{n=0}^{h} \frac{(-h)_n (w-1)^n w^{h-n}}{n!} 
     \FQ_{h-n}(\alpha, R+\alpha n; z) \\ 
\FQ_h(\alpha, R; zw) & = 
     \sum_{n=0}^{h} \frac{(-h)_n \left(1-h-\frac{R}{\alpha}\right)_n 
     (1-w)^{n} (\alpha z)^{n}}{n!} 
     \FQ_{h-n}(\alpha, R; z) \\ 
     & = 
     \sum_{n=0}^{h} 
     \binom{h}{n} \binom{n-h-\frac{R}{\alpha}}{n} \times 
     \left(\alpha z (w-1)\right)^{n} n! \times 
     \FQ_{h-n}(\alpha, R; z) 
\end{align*} 
\end{prop} 
\begin{proof}[Proof of the Addition Theorems] 
The sums stated in \eqref{eqn_FQhaRz_HypU_FiniteSum_Idents_stmt} 
follow from special cases of established addition theorems for the 
confluent hypergeometric function, $\HypU{a}{b}{x+y}$, cited in the 
reference \cite[\S 13.13(ii)]{NISTHB}. 
The particular addition theorems required in the proof are 
provided as follows: 
\begin{align} 
\label{eqn_AdditionMultThm_for_HypU_proof_stmts-v1} 
\HypU{a}{b}{x+y} & = \sum_{n=0}^{\infty} \frac{(a)_n (-y)^n}{n!} 
     \HypU{a+n}{b+n}{x},\ |y| < |x| \\ 
\notag 
\HypU{a}{b}{x+y} & = \left(\frac{x}{x+y}\right)^{a} \sum_{n=0}^{\infty} 
     \frac{(a)_n (1+a-b)_n y^n}{n! (x+y)^n} \HypU{a+n}{b}{x},\ 
     \Re[y / x] > -\frac{1}{2}. 
\end{align} 
First, 
observe that in the special case inputs to $\HypU{a}{b}{z}$ 
resulting from the application of 
Proposition \ref{prop_Closed-Form_Rep_for_ConvFn_Qhz} 
involving the functions 
$\FQ_h(\alpha, R; z) = \HypU{-h}{R / \alpha}{(\alpha z)^{-1}}$ 
in the infinite sums of 
\eqref{eqn_AdditionMultThm_for_HypU_proof_stmts-v1} 
lead to to finite sum identities corresponding to the inputs, $h$, to 
$\FQ_h(\alpha, R; z)$ where $h \geq 0$. 
More precisely, the definition of the convergent denominator sequences 
provided by \eqref{eqn_QFact_Qhz} requires that 
$\FQ_h(\alpha, R; z) = 0$ whenever $h < 0$. 

To apply the cited results for $\HypU{a}{b}{x+y}$ in these cases, 
let $z \neq w$, assume that both $\alpha, z \neq 0$, and 
suppose the parameters corresponding to $x$ and $y$ 
in \eqref{eqn_FQhaRz_HypU_FiniteSum_Idents_stmt} are defined so that 
\begin{equation} 
\label{eqn_AdditionMultThm_analogues_proof_stmt-v1} 
x \defequals (\alpha z)^{-1},\ 
y \defequals \frac{1}{\alpha}\left((z-w)^{-1} - z^{-1}\right),\ 
x + y = \left(\alpha (z-w)\right)^{-1}. 
\end{equation} 
Since each of the sums in \eqref{eqn_FQhaRz_HypU_FiniteSum_Idents_stmt} 
involve only finitely-many terms, 
we ignore treatment of the convergence conditions given on the 
right-hand-sides of the equations in 
\eqref{eqn_AdditionMultThm_for_HypU_proof_stmts-v1} 
to justify these two restatements of the addition theorem analogs 
provided above. 
\end{proof} 
\begin{proof}[Proof of the Multiplication Theorems] 
The second pair of identities stated in 
\eqref{eqn_FQhaRz_HypU_FiniteSum_Idents-v_MultThms_stmts} 
are formed by the 
multiplication theorems for $\HypU{a}{b}{z}$ noted as in the reference 
\cite[\S 13.13(iii)]{NISTHB}. 
The proof is derived similarly from the first parameter 
definitions of $x$ and $y$ given in the addition theorem proof, 
with an additional adjustment employed in these cases 
corresponding to the change of variable 
$\widehat{y} \mapsto (y-1) x$, which is selected so that 
$x + \widehat{y} \defmapsto xy$ in the above proof. 
The analog to 
\eqref{eqn_AdditionMultThm_analogues_proof_stmt-v1} 
that results in these two cases then 
yields the parameters, 
$x \defequals (\alpha z)^{-1}$ and $y \defequals (w^{-1}-1) \cdot (\alpha z)^{-1}$, 
in the first identities for the 
confluent hypergeometric function, $\HypU{a}{b}{x+y}$, 
given by \eqref{eqn_AdditionMultThm_for_HypU_proof_stmts-v1}. 
\end{proof} 

\begin{remark} 
The expansions of the addition and multiplication theorem analogs to the 
established relations for the confluent hypergeometric function, 
$\HypU{a}{b}{w}$, are also compared to the known expansions of the 
\emph{duplication formula} for the associated Laguerre polynomial 
sequence stated in the following form 
\cite[\S 5.1]{UC} \cite[\cf \S 18.18(iii)]{NISTHB}: 
\begin{align*} 
h! \times L_h^{(\beta)}(wx) & = 
     \sum_{k=0}^{h} \binom{h+\beta}{h-k} \left(\frac{h!}{k!}\right)^{2} 
     \times w^{k} (1-w)^{h-k} \times k! L_k^{(\beta)}(x). 
\end{align*} 
The second expansion of the convergent denominator functions, 
$\ConvFQ{h}{\alpha}{R}{z}$, by the 
confluent hypergeometric function, $\HypM{a}{b}{w}$, stated in 
\eqref{eqn_PFact_Qhz_Mident} of the first proposition in this section 
also suggests additional identities for these sequences generated by the 
multiplication formula analogs in 
\eqref{eqn_FQhaRz_HypU_FiniteSum_Idents-v_MultThms_stmts} 
when the parameter $z \defmapsto \pm 1 / \alpha$, for example, as in the 
simplified cases of the identities expanded in the references 
(\cite[\S 5.5 - \S 5.6; Ex.\ 5.29]{GKP}, \cite[\cf \S 15]{NISTHB}). 
\end{remark} 

\subsection{The convergent numerator function sequences} 
\label{subsubSection_Properties_Of_ConvFn_Phz} 

The most direct expansion of the convergent numerator functions, 
$\FP_h(\alpha, R; z)$, is obtained from the 
\emph{erasing operator}, defined as in Flajolet's first article, 
which performs the formal power series 
truncation operation defined by the next equation 
\cite[\S 3]{FLAJOLET80B}. 
\begin{align*} 
\tagtext{Erasing Operator} 
\E_m\left\llbracket\sum_{i} g_i z^i\right\rrbracket 
     & \defequals 
     \sum_{i \leq m} g_i z^{i} 
\end{align*} 
The numerator polynomials are then given through this notation by the 
expansions in the following equations: 
\begin{align*} 
\FP_h(\alpha, R; z) & = 
     \E_{h-1} \Bigl\llbracket 
     \FQ_h(\alpha, R; z) \cdot \ConvGF{h}{\alpha}{R}{z} 
     \Bigr\rrbracket \\ 
     & = 
     \sum_{n=0}^{h-1} 
     \undersetbrace{ 
     C_{h,n}(\alpha, R) \defequals [z^{n}] \ConvFP{h}{\alpha}{R}{z}}{ 
     \left( 
     \sum_{i=0}^{n} 
     [z^{i}] \ConvFQ{h}{\alpha}{R}{z} \times 
     \pn{n-i}{\alpha}{R} 
     \right) 
     } \times z^{n}. 
\end{align*} 
The coefficients of $z^{n}$ 
expanded in the last equation are rewritten slightly in terms of 
\eqref{eqn_CoeffsOfzk_FQhaRz_restmts_by_GenProductSeqs-v1} and the 
Pochhammer symbol representations of the product sequences, 
$\pn{n}{\alpha}{R}$, to arrive at a pair of formulas 
expanded as follows: 
\begin{subequations} 
\label{eqn_Vandermonde-like_PHSymb_exps_of_PhzCfs} 
\begin{align} 
\label{eqn_Chn_formula_stmt_v1} 
C_{h,n}(\alpha, R) 
     & = 
     \sum_{i=0}^{n} \binom{h}{i} (-1)^{i} 
     \pn{i}{-\alpha}{R + (h-1) \alpha} 
     \pn{n-i}{\alpha}{R}, && 
     h > n \geq 0 \\ 
\label{eqn_Vandermonde-like_PHSymb_exps_of_PhzCfs-stmt_v1} 
     & = 
     \sum_{i=0}^{n} \binom{h}{i} 
     \Pochhammer{1-h-R / \alpha}{i} 
     \Pochhammer{R / \alpha}{n-i} \times \alpha^{n}, && 
     h > n \geq 0. 
\end{align} 
\end{subequations} 
These sums are remarkably similar in form to the next 
binomial-type convolution formula, or \emph{Vandermonde identity}, 
stated as follows \cite[\S 1.13(I)]{ADVCOMB} 
\cite{WOLFRAMFNSSITE-INTRO-FACTBINOMS} 
\cite[\S 1.2.6, Ex.\ 34]{TAOCPV1}: 
\begin{align*} 
\tagtext{Vandermonde Convolution} 
\Pochhammer{x+y}{n} & = 
     \sum_{i=0}^{n} \binom{n}{i} \Pochhammer{x}{i} \Pochhammer{y}{n-i} \\ 
     & = 
     \sum_{i=0}^{n} \binom{n}{i} x \cdot \Pochhammer{x-iz+1}{i-1} 
     \Pochhammer{y+iz}{n-i},\ x \neq 0. 
\end{align*} 
A separate treatment of other properties implicit to the 
more complicated expansions of these convergent function 
subsequences is briefly 
explored through the definitions of the three additional forms of 
auxiliary coefficient sequences, denoted in respective order by 
$C_{h,n}(\alpha, R)$, $R_{h,k}(\alpha; z)$, and $T_h^{(\alpha)}(n, k)$, 
considered in the next subsection below. 

\begin{remark}[Reflected Convergent Numerator Function Sequences] 
The special cases of the reflected numerator polynomials given in 
\tableref{table_RelfectedConvNumPolySeqs_sp_cases} also 
suggest a consideration of the numerator convergent functions 
factored with respect to powers of $\pm (z-R)$ 
by expanding these sequences with respect to 
another formal auxiliary variable, $w$, when $R \defmapsto z \mp w$. 
The tables contained in the attached summary notebook 
\cite{SUMMARYNBREF-STUB} provide working \Mm{} 
code to expand and factor these modified forms of the 
reflected numerator polynomial sequences employed in stating the 
generalized congruence results for the $\alpha$-factorial functions, 
$\MultiFactorial{n}{\alpha}$, from the examples cited in 
Section \ref{subsubSection_Examples_NewCongruences}, and 
more generally by the results proved in 
Section \ref{subSection_NewCongruence_Relations_Modulo_Integer_Bases}, 
satisfied by the $\alpha$-factorial functions and 
generalized product sequence expansions modulo integers $p \geq 2$. 
\end{remark} 

\subsubsection{Alternate forms of the 
               convergent numerator function subsequences}
\label{subsubSection_Properties_Of_ConvFn_Phz-AuxNumFn_Subsequences} 

The next results summarize three semi-triangular 
recurrence relations satisfied by the 
particular variations of the numerator function subsequences considered, 
respectively, as polynomials with respect to $z$ and $R$. 
For $h \geq 2$, fixed $\alpha \in \mathbb{Z}^{+}$, and 
$n, k \geq 0$, we consider the following forms of these auxiliary 
numerator coefficient subsequences: 
\begin{align*} 
\tagonce\label{eqn_Chk_aux_numerator_subseqs-def_v1}
C_{h,n}(\alpha, R) & \defequals [z^n] \ConvFP{h}{\alpha}{R}{z},\ 
       \text{ for \ } 
       0 \leq n \leq h-1 \\ 
       & \phantom{:} = 
C_{h-1,n}(\alpha, R) - (R+2\alpha(h-1)) C_{h-1,n-1}(\alpha, R) \\ 
     & \phantom{= \quad} - 
     \alpha(R+\alpha(h-2))(h-1) C_{h-2,n-2}(\alpha, R) \\ 
R_{h,k}(\alpha; z) & \defequals [R^k] \ConvFP{h}{\alpha}{R}{z},\ 
       \text{ for \ } 
       0 \leq k \leq h-1 \\ 
       & \phantom{:} = 
     (1-2\alpha(h-1)z)R_{h-1,k}(\alpha; z) - 
     \alpha^2(h-1)(h-2)z^2 R_{h-2,k}(\alpha; z) \\ 
       & \phantom{\defequals\ } - 
     zR_{h-1,k-1}(\alpha; z) - 
     \alpha(h-1)z^2 R_{h-2,k-1}(\alpha; z) \\ 
T_h^{(\alpha)}(n, k) & \defequals [z^n R^k] \ConvFP{h}{\alpha}{R}{z},\ 
       \text{ for \ } 
       0 \leq n, k \leq h-1 \\ 
       & \phantom{:}= 
T_{h-1}^{(\alpha)}(n, k) -T_{h-1}^{(\alpha)}(n-1, k-1) - 
     2\alpha (h-1) T_{h-1}^{(\alpha)}(n-1 ,k) \\ 
   & \phantom{\defequals\ } - 
     \alpha (h-1) T_{h-2}^{(\alpha)}(n-2, k-1) - 
     \alpha^2 (h-1)(h-2) T_{h-2}^{(\alpha)}(n-2, k) \\ 
   & \phantom{\defequals\ } + 
     \left([z^n R^0] \FP_h(z)\right) 
     \Iverson{h \geq 1} \Iverson{n \geq 0} \Iverson{k = 0}. 
\end{align*} 
Each of the recurrence relations for the triangles 
cited in the previous equations 
are derived from \eqref{eqn_PFact_Phz} by a 
straightforward application of the coefficient extraction method 
first motivated in \cite{MULTIFACTJIS}. 
\tableref{table_ConvNumFnSeqs_Chn_AlphaR_SpCaseListings} and 
\tableref{table_ConvNumFnSeqs_Rhk_Alphaz_SpCaseListings} 
list the first few special cases of the first two 
auxiliary forms of these component polynomial subsequences. 

We also state, without proof, 
a number of multiple, alternating sums involving the Stirling number 
triangles that generate these auxiliary subsequences 
for reference in the next several equations. 
In particular, for $h \geq 1$ and $0 \leq n < h$, the sequences, 
$C_{h,n}(\alpha, R)$, are expanded by the following sums: 
\begin{subequations} 
\label{eqn_Chn_formula_stmts} 
\begin{align} 
\label{eqn_Chn_formula_stmts-exp_v1}
C_{h,n}(\alpha, R) & = 
     \sum\limits_{\substack{0 \leq m \leq k \leq n \\ 
                            0 \leq s \leq n} 
                 } 
     \left( 
     \binom{h}{k} \binom{m}{s} \gkpSI{k}{m} (-1)^{m} \alpha^{n} 
     \Pochhammer{\frac{R}{\alpha}}{n-k} 
     \left(\frac{R}{\alpha} - 1\right)^{m-s} 
     \right) \times h^{s} \\ 
\label{eqn_Chn_formula_stmts-exp_v2}
  & = 
     \sum\limits_{\substack{0 \leq m \leq k \leq n \\ 
                            0 \leq t \leq s \leq n} 
                 } 
     \left( 
     \binom{h}{k} \binom{m}{t} \gkpSI{k}{m} \gkpSI{n-k}{s-t} 
     (-1)^{m} \alpha^{n-s} (h-1)^{m-t} 
     \right) \times R^{s} \\ 
\label{eqn_Chn_formula_stmts-exp_v3}
   & = 
     \sum\limits_{\substack{0 \leq m \leq k \leq n \\ 
                            0 \leq i \leq s \leq n} 
                 } 
     \binom{h}{k} \binom{h}{i} \binom{m}{s} \gkpSI{k}{m} \gkpSII{s}{i} 
     (-1)^{m} \alpha^{n} 
     \Pochhammer{\frac{R}{\alpha}}{n-k} 
     \left(\frac{R}{\alpha} - 1\right)^{m-s} \times i! \\ 
\label{eqn_Chn_formula_stmts-exp_v4}
   & = 
     \sum\limits_{\substack{0 \leq m \leq k \leq n \\ 
                            0 \leq v \leq i \leq s \leq n} 
                 } 
     \binom{h}{k} \binom{m}{s} \binom{i}{v} \binom{h+v}{v} 
     \gkpSI{k}{m} \gkpSII{s}{i} (-1)^{m+i-v} \alpha^{n} \times \\ 
\notag 
     & \phantom{= \sum \binom{h}{k} \quad } \times 
     \Pochhammer{\frac{R}{\alpha}}{n-k} 
     \left(\frac{R}{\alpha} - 1\right)^{m-s} \times i!. 
\end{align} 
Since the powers of $R$ in the second identity 
are expanded by the Stirling numbers of the second kind as 
\cite[\S 6.1]{GKP} 
\begin{align*} 
R^{p} & = \alpha^{p} \times 
     \sum_{i=0}^{p} \gkpSII{p}{i} (-1)^{p-i} \Pochhammer{\frac{R}{\alpha}}{i}, 
\end{align*} 
for all natural numbers $p \geq 0$, the multiple sum identity in 
\eqref{eqn_Chn_formula_stmts-exp_v2} also implies the next finite multiple sum 
expansion for these auxiliary coefficient subsequences 
(see Table \ref{table_ConvNumFnSeqs_Chn_AlphaR_SpCaseListings}). 
\begin{align} 
\label{eqn_Chn_formula_stmts-exp_v5}
C_{h,n}(\alpha, R) & = 
     \sum_{i=0}^{n} 
     \underset{\mathlarger{\text{polynomial function of $h$ only } \defequals 
               \frac{(-1)^{n} m_{n,h}}{n!} \times \binom{n}{i} p_{n,i}(h)}}{ 
               \underbrace{ 
     \left( 
     \sum\limits_{\substack{0 \leq m \leq k \leq n \\ 
                            0 \leq t \leq s \leq n} 
                 } 
     \binom{h}{k} \binom{m}{t} \gkpSI{k}{m} \gkpSI{n-k}{s-t} \gkpSII{s}{i} 
     (-1)^{m+s-i} (h-1)^{m-t} 
     \right) 
     } 
     } 
     \times \alpha^{n} \Pochhammer{\frac{R}{\alpha}}{i} 
\end{align} 
\end{subequations} 
Similarly, for all $h \geq 1$ and $0 \leq k < h$, the sequences, 
$R_{h,k}(\alpha, R)$, are expanded as follows: 
\begin{align} 
\label{eqn_Rhk_formula_stmts} 
R_{h,k}(\alpha; z) & = 
     \sum\limits_{\substack{0 \leq m \leq i \leq n < h \\ 
                            0 \leq t \leq k} 
                 } 
     \left( 
     \binom{h}{i} \binom{m}{t} \gkpSI{i}{m} \gkpSI{n-i}{k-t} 
     (-1)^{m} \alpha^{n-k} (h-1)^{m-t} 
     \right) \times z^{n} \\ 
\notag 
   & = 
     \sum\limits_{\substack{0 \leq m \leq i \leq n < h \\ 
                            0 \leq t \leq k \\ 0 \leq p \leq m-t} 
                 } 
     \left( 
     \binom{h}{i} \binom{m}{t} \binom{h-1}{p} 
     \gkpSI{i}{m} \gkpSI{n-i}{k-t} \gkpSII{m-t}{p} 
     (-1)^{m} \alpha^{n-k} \times p! 
     \right) \times z^{n}. 
\end{align} 
A more careful immediate 
treatment of the properties satisfied by these subsequences is omitted 
from this section for brevity. 
A number of the new congruence results cited in the next sections do, 
at any rate, 
have alternate expansions given by the more involved termwise structure 
implicit to these finite multiple sums modulo some 
application-specific prescribed functions of $h$. 

\section{Applications and motivating examples} 
\label{Section_Apps_and_Examples} 

\subsection{Lemmas} 
\label{subSection_Apps_and_Examples_StmtsOfLemmas} 

We require the next lemma 
to formally enumerate the generalized products and factorial function sequences 
already stated without proof in the examples from 
Section \ref{subSection_Intro_Examples}. 

\begin{lemma}[Sequences Generated by the Generalized Convergent Functions] 
\label{lemma_GenConvFn_EnumIdents_pnAlphaRSeq_idents_combined_v1} 
For fixed integers $\alpha \neq 0$, $0 \leq d < \alpha$, and each 
$n \geq 1$, the generalized $\alpha$-factorial sequences 
defined in \eqref{eqn_nAlpha_Multifact_variant_rdef} 
satisfy the following expansions by the generalized products in 
\eqref{eqn_GenFact_product_form}: 
\begin{align} 
(\alpha n - d)!_{(\alpha)} 
     & = p_n(-\alpha, \alpha n-d) \\ 
     & = p_n(\alpha, \alpha -d) \\ 
\label{eqn_MultFactFn_ConvSeq_def-stmts_v1.c} 
n!_{(\alpha)} 
     & = p_{\lfloor (n+\alpha-1) / \alpha \rfloor}(-\alpha, n). 
\end{align} 
\end{lemma} 
\begin{proof} 
The related cases of each of these identities cited in the 
equations above correspond to proving to the 
equivalent expansions of the product-wise representations for the 
$\alpha$-factorial functions given in each of the next equations: 
\begin{align*} 
\label{eqn_MultFactFn_ConvSeq_def-proof_stmts_v1.i} 
\tag{a} 
\AlphaFactorial{\alpha n-d}{\alpha} & = 
     \prod_{j=0}^{n-1} \left(\alpha n - d - j\alpha\right) \\ 
\label{eqn_MultFactFn_ConvSeq_def-proof_stmts_v1.ii} 
\tag{b} 
   & = 
     \prod_{j=0}^{n-1} \left(\alpha - d + j \alpha\right) \\ 
\tag{c} 
\label{eqn_MultFactFn_ConvSeq_def-proof_stmts_v2.iii} 
\MultiFactorial{n}{\alpha} & = 
     \prod\limits_{j=0}^{\left\lfloor (n+\alpha-1) / \alpha 
     \right\rfloor - 1} (n-i\alpha). 
\end{align*} 
The first product in \eqref{eqn_MultFactFn_ConvSeq_def-proof_stmts_v1.i} 
is easily obtained from \eqref{eqn_nAlpha_Multifact_variant_rdef} by 
induction on $n$, which then implies the second result in 
\eqref{eqn_MultFactFn_ConvSeq_def-proof_stmts_v1.ii}. 
Similarly, an inductive argument applied to the definition provided by 
\eqref{eqn_nAlpha_Multifact_variant_rdef} 
proves the last product representation given in 
\eqref{eqn_MultFactFn_ConvSeq_def-proof_stmts_v2.iii}. 
\end{proof} 

\begin{proof}[Corollaries] 
The proof of the lemma provides immediate corollaries 
to the special cases of the $\alpha$-factorial functions, 
$(\alpha n-d)!_{(\alpha)}$, expanded by the results from 
\eqref{eqn_MultFactFn_ConvSeq_def_v1}. 
We explicitly state the 
following particular special cases of the lemma corresponding to 
$d \defequals 0$ in \eqref{eqn_AlphaFactFn_anm1_SpCase_SeqIdents-stmts_v0} 
below for later use in 
Section \ref{subsubSection_Apps_ArithmeticProgs_of_the_SgFactFns} 
of the article: 
\begin{align} 
\label{eqn_AlphaFactFn_anm1_SpCase_SeqIdents-stmts_v0} 
(\alpha n)!_{(\alpha)} 
     & = \alpha^{n} \cdot \Pochhammer{1}{n} 
       = [z^n] \ConvGF{n+n_0}{-\alpha}{\alpha n}{z},\ 
       \forall n_0 \geq 0 \\ 
\notag 
    & = \alpha^{n} \cdot n! 
       \phantom{(_{n}} = 
       [z^n] \ConvGF{n+n_0}{-1}{n}{\alpha z},\ 
       \forall n_0 \geq 0. 
       \qedhere 
\end{align} 
\end{proof} 

\begin{remark}[Generating Floored Arithmetic Progressions of $\alpha$-Factorial Functions] 
\label{remark_Lemmas_ArithmeticProgGFs_nOvermFloored} 
Lemma \ref{lemma_GenConvFn_EnumIdents_pnAlphaRSeq_idents_combined_v1} 
provides proofs of the convergent-function-based 
generating function identities enumerating the $\alpha$-factorial 
sequences given in 
\eqref{eqn_MultFactFn_ConvSeq_def_v1} and 
\eqref{eqn_MultFactFn_ConvSeq_def_v2} of the introduction. 
The last convergent-based generating function identity that enumerates the 
$\alpha$-factorial functions, $\MultiFactorial{n}{\alpha}$, when 
$n > \alpha$ expanded in the form of 
\eqref{eqn_MultFactFn_ConvSeq_def_v3} from 
Section \ref{subSection_Intro_Examples} 
follows from the product function expansions provided in 
\eqref{eqn_MultFactFn_ConvSeq_def-stmts_v1.c} of the lemma by 
applying a result proved in the exercises section of the reference 
\cite[\S 7, Ex.\ 7.36; p.\ 569]{GKP}. 
In particular, for any fixed $m \geq 1$ and some sequence, 
$\left(a_n\right)_{n \geq 0}$, a generating function for the 
modified sequences, $\left(a_{\lfloor n/m \rfloor}\right)_{n \geq 0}$, is 
given by 
\begin{align*} 
\widehat{A}_m(z) & \defequals 
     \sum_{n \geq 0} a_{\lfloor \frac{n}{m} \rfloor} z^{n} = 
     \frac{1-z^{m}}{1-z} \times \widetilde{A}\left(z^{m}\right) = 
     \left(1 + z + \cdots + z^{m-2} + z^{m-1}\right) \times 
     \widetilde{A}\left(z^{m}\right), 
\end{align*} 
where $\widetilde{A}(z) \defequals \sum_{n} a_n z^{n}$ denotes the 
ordinary power series generating function formally enumerating the 
prescribed sequence over $n \geq 0$ 
(compare to Remark \ref{remark_GFArithmeticProg_formulas} in 
Section \ref{subsubSection_SpCaseArithProgsOfSgFactFn_alphaEq45}). 
\end{remark} 

\begin{lemma}[Identities and Other Formulas Connecting Pochhammer Symbols] 
\label{lemma_footnote_PHSymbol_BinomIdents} 
The next equations provide statements of several known identities 
from the references involving the 
falling factorial function, the Pochhammer symbol, 
or rising factorial function, and the 
binomial coefficients required by the applications given in the 
next sections of the article. 
\begin{enumerate} 

\item \itemlabel{Relations between rising and falling factorial functions} 
The following identities provide known relations between the 
rising and falling factorial functions for fixed $x \neq \pm 1$ and 
integers $m,n \geq 0$ where the coefficients, 
$L(n, k) = \binom{n-1}{k-1} \frac{n!}{k!}$, in the first connection 
formula denote the \emph{Lah numbers} (\seqnum{A105278}): 
\begin{align*} 
\tagtext{Connection Formulas} 
\FFactII{x}{n} 
     & = 
     \sum_{k=1}^{n} 
     \binom{n-1}{k-1} \frac{n!}{k!} \times \Pochhammer{x}{k} \\ 
     & = 
     (-1)^{n} \Pochhammer{-x}{n} = 
     \Pochhammer{x-n+1}{n} = 
     \frac{1}{\RFactII{(x+1)}{-n}} \\ 
\Pochhammer{x}{n} 
     & = 
     \sum_{k=0}^{n} 
     \binom{n}{k} \FFactII{(n-1)}{n-k} \times \FFactII{x}{k} \\ 
     & = 
     (-1)^{n} \FFactII{(-x)}{n} \\ 
     & = 
     \FFactII{(x+n-1)}{n} = 
     \frac{1}{\FFactII{(x-1)}{-n}} \\ 
\tagtext{Generalized Exponent Laws} 
\FFactII{x}{m+n} 
     & = 
     \FFactII{x}{m} \FFactII{(x-m)}{n} \\ 
\RFactII{x}{m+n} 
     & = 
     \RFactII{x}{m} \RFactII{(x+m)}{n} \\  
\tagtext{Negative Rising and Falling Powers} 
\RFactII{x}{-n} 
     & = 
     \frac{1}{\Pochhammer{x-n}{n}} = 
     \frac{1}{\FFactII{(x-1)}{n}} \\ 
\FFactII{x}{-n} 
     & = 
     \frac{1}{\Pochhammer{x+1}{n}} \\ 
     & = 
     \frac{1}{n! \cdot \binom{x+n}{n}} = 
     \frac{1}{(x+1)(x+2) \cdots (x+n)}. 
\end{align*} 

\item \itemlabel{Expansions of polynomial powers by the 
                 Stirling numbers of the second kind} 
For any fixed $x \neq 0$ and integers $n \geq 0$, the 
polynomial powers of $x^n$ are expanded as follows: 
\begin{align*} 
\tagtext{Expansions of Polynomial Powers} 
x^{n} 
     & = 
     \sum_{k=0}^{n} \gkpSII{n}{n-k} \FFactII{x}{n-k} = 
     \sum_{k=0}^{n} \gkpSII{n}{k} (-1)^{n-k} \RFactII{x}{n}. 
\end{align*} 

\item \itemlabel{Binomial coefficient identities expanded by the 
                 Pochhammer symbol} 
For fixed $x \neq 0$ and integers $n \geq 1$, the 
binomial coefficients are expanded by 
\begin{align*} 
\tagtext{Binomial Coefficient Identities} 
\Pochhammer{x}{n} 
     & = 
     \binom{-x}{n} \times (-1)^{n} n! \\ 
     & = 
     \binom{x+n-1}{n} \times n!. 
\end{align*} 

\item \itemlabel{Connection formulas for products and 
                 ratios of Pochhammer symbols} 
The next identities connecting products and ratios of the 
Pochhammer symbols are stated 
for integers $m, n, i \geq 0$ and fixed $x \neq 0$. 
\begin{align*} 
\tagtext{Connection Formulas for Products} 
\Pochhammer{x}{n} \Pochhammer{x}{m} 
     & = 
     \sum_{k=0}^{\min(m, n)} \binom{m}{k} \binom{n}{k} k! \cdot 
     \Pochhammer{x}{m+n-k}, n \neq m \\ 
\tagtext{Fractions of Pochhammer Symbols} 
\frac{\Pochhammer{x}{n}}{\Pochhammer{x}{i}} 
     & = 
     \Pochhammer{x+i}{n-i},\ n \geq i 
\end{align*} 

\end{enumerate} 
\end{lemma} 
\begin{proof} 
The statements of the formulas given in the lemma 
are contained in the references, which in many cases also provide 
explicit proofs of these results. The first relations between the 
rising and falling factorial functions are given in the references 
\cite[\S 4.1.2, \S 5; \cf \S 4.3.1]{UC} 
\cite[\S 2, Ex.\ 2.17, 2.9, 2.16; \S 5.3; \S 6, Ex.\ 6.31, p.\ 552]{GKP}, the 
expansions of polynomial powers are given in the reference 
\cite[\S 6.1]{GKP}, the binomial coefficient identities for the 
Pochhammer symbol are found in the references 
\cite{CVLPOLYS,WOLFRAMFNSSITE-INTRO-FACTBINOMS}, and the 
connection formulas for products and ratios of Pochhammer symbols are 
given in the references \cite[Ex.\ 1.23, p.\ 83]{ADVCOMB} 
\cite{WOLFRAMFNSSITE-INTRO-FACTBINOMS}. 
\end{proof} 

\subsection{New congruences for the 
            $\alpha$-factorial functions, the 
            generalized Stirling number triangles, and 
            Pochhammer $k$-symbols} 
\label{subSection_NewCongruence_Relations_Modulo_Integer_Bases} 
\label{subSection_Congruences_for_Series_ModuloIntegers_p} 

The new results stated in this subsection follow immediately 
from the congruences properties modulo integer divisors of the $M_h$ 
summarized by Section \ref{subSection_EnumProps_of_JFractions} 
considered as in the references \cite{FLAJOLET80B,FLAJOLET82} 
\cite[\cf \S 5.7]{GFLECT}. 
The particular cases of the J-fraction representations enumerating the 
product sequences defined by \eqref{eqn_GenFact_product_form} 
always yield a factor of $h \defequals N_h \mid M_h$ in the statement of 
\eqref{eqn_EnumProps_Of_JTypeCFracs_congruence_rels-stmt_v1} 
(see Remark \ref{remark_Congruences_for_Rational-Valued_Params} below). 
One consequence of this property implicit to each of the generalized 
factorial-like sequences observed so far, is that it is straightforward to 
formulate new congruence relations for these sequences modulo any 
fixed integers $p \geq 2$. 

\begin{remark}[Congruences for Rational-Valued Parameters] 
\label{remark_Congruences_for_Rational-Valued_Params} 
The J-fraction parameters, 
$\lambda_h = \lambda_h(\alpha, R)$ and $M_h = M_h(\alpha, R)$, 
defined as in the summary of the enumerative properties from 
Section \ref{subSection_EnumProps_of_JFractions}, 
corresponding to the 
expansions of the generalized convergents defined by the proof in 
Section \ref{subSection_GenCFrac_Reps_for_GenFactFns} 
satisfy 
\begin{align*} 
\lambda_k(\alpha, R) & \defequals 
     \as_{k-1}(\alpha, R) \cdot \bs_{k}(\alpha, R) \\ 
   & \phantom{:} = 
     \alpha (R + (k-1) \alpha) \cdot k \\ 
M_h(\alpha, R) & \defequals 
     \lambda_1(\alpha, R) \cdot \lambda_2(\alpha, R) 
     \times \cdots \times 
     \lambda_h(\alpha, R) \\ 
   & \phantom{:} = 
     \alpha^{h} \cdot h! \times p_h(\alpha, R) \\ 
   & \phantom{:} = 
   \alpha^{h} \cdot h! \times \Pochhammer{R}{h,\alpha}, 
\end{align*} 
so that for integer divisors, $N_h(\alpha, R) \mid M_h(\alpha, R)$, 
we have that 
\begin{align*} 
p_n(\alpha, R) & \equiv 
     [z^n] \ConvGF{h}{\alpha}{R}{z} \pmod{N_h(\alpha, R)}. 
\end{align*} 
So far we have restricted ourselves to examples of the 
particular product sequence cases, $p_n(\alpha, R)$, 
where $\alpha \neq 0$ is integer-valued, \ie 
so that $p, p\alpha^{i} \mid M_p(\alpha, R)$ for $1 \leq i \leq p$ 
whenever $p \geq 2$ is a fixed natural number. 
Explicit congruence identities arising in some other related applications 
when the choice of $\alpha \neq 0$ is strictly rational-valued
are intentionally not treated in the examples cited in this section. 
\end{remark} 

\subsubsection{Congruences for the Stirling numbers of the first kind and the 
               $r$-order harmonic number sequences} 
\label{subsubSection_remark_New_Congruences_for_GenS1Triangles_and_HNumSeqs} 

\sublabel{Generating the Stirling numbers 
          as series coefficients of the 
          generalized convergent functions} 
For integers $h \geq 3$ and $n \geq m \geq 1$, the (unsigned) 
Stirling numbers of the first kind, $\gkpSI{n}{m}$, 
are generated by the polynomial expansions of the 
rising factorial function, or Pochhammer symbol, 
$\RFactII{x}{n} = \Pochhammer{x}{n}$, as follows 
(\cite[\S 7.4; \S 6]{GKP},\seqnum{A130534}, \seqnum{A008275}): 
\label{eqn_S1FcfAlpha_GenConvFn_Coeffs_and_CongruencesResultStmts} 
\begin{align*} 
\tagonce 
\gkpSI{n}{m} & = 
     [R^m] R (R+1) \cdots (R+n-1) \\ 
     & = 
     [z^n] [R^m] \ConvGF{h}{1}{R}{z},\ 
     \text{ for } 
     1 \leq m \leq n \leq h. 
\end{align*} 
Analogous formulations of new congruence results for the 
$\alpha$-factorial triangles defined by \eqref{eqn_Fa_rdef}, and the 
corresponding forms of the generalized harmonic number sequences, 
are expanded by noting that 
for all $n,m \geq 1$, and integers $p \geq 2$, we have the 
following expansions 
\cite{MULTIFACTJIS}: 
\begin{align*} 
\tagonce 
\FcfII{\alpha}{n}{m} 
     & = 
     [s^{m-1}] (s+1) (s+1+\alpha) \cdots (s+1 + (n-2) \alpha) \\ 
     & = 
     [s^{m-1}] \pn{n-1}{\alpha}{s+1} \\ 
\FcfII{\alpha}{n}{m} & \equiv 
     [z^{n-1}] [R^{m-1}] \ConvGF{p}{\alpha}{R+1}{z} \pmod{p}. 
\end{align*} 
The coefficients of $R^{m}$ in the 
series expansions of the convergent functions, 
$\ConvGF{h}{1}{R}{z}$, in the formal variable $R$ are 
rational functions of $z$ with denominators given by $m^{th}$ powers of the 
reflected polynomials defined in the next equation. 

\sublabel{Definitions of the quasi-polynomial expansions for the 
          Stirling numbers of the first kind} 
For a fixed $h \geq 3$, let the roots, $\omega_{h,i}$, be 
defined as follows: 
\begin{align*} 
\tagonce\label{eqn_S1CoeffsModuloh_DenomReflectedRoots_defs} 
\left(\omega_{h,i}\right)_{i=1}^{h-1} & \defequals 
     \left\{ \omega_j : 
     \sum_{i=0}^{h-1} \binom{h-1}{i} \frac{h!}{(i+1)!} (-\omega_j)^{i} = 0,\ 
     1 \leq j < h \right\}. 
\end{align*} 
The forms of both exact formulas and congruences for the 
Stirling numbers of the first kind modulo any prescribed integers $h \geq 3$ 
are then expanded as 
\begin{align*} 
\gkpSI{n}{m} & = 
     \left(\sum_{i=0}^{h-1} p_{h,i}^{[m]}(n) \times \omega_{h,i}^{n} 
     \right) \Iverson{n > m} + \Iverson{n = m}\ 
     \phantom{\pmod{h}} 
     m \leq n \leq h \\ 
\gkpSI{n}{m} & \equiv 
     \left(\sum_{i=0}^{h-1} p_{h,i}^{[m]}(n) \times \omega_{h,i}^{n} 
     \right) \Iverson{n > m} + \Iverson{n = m} \pmod{h},\ 
     \forall n \geq m, 
\end{align*} 
where the functions, $p_{h,i}^{[m]}(n)$, denote fixed 
polynomials of degree $m$ in $n$ for each $h$, $m$, and $i$ 
\cite[\S 7.2]{GKP}. 
For example, 
when $h \defequals 2, 3$, the respective reflected roots 
defined by the previous equations in 
\eqref{eqn_S1CoeffsModuloh_DenomReflectedRoots_defs} are given exactly by 
\begin{equation*} 
\left\{ \omega_{2,1} \right\} \defequals \{2\} 
     \qquad \text{ and } \qquad 
\left(\omega_{3,i}\right)_{i=1}^{2} \defequals 
     \left\{3-\sqrt{3}, 3+\sqrt{3}\right\}. 
\end{equation*} 

\sublabel{Comparisons to known congruences for the Stirling numbers} 
The special case of 
\eqref{eqn_cor_Congruences_for_AlphaFactFns_modulo2} from the 
examples given in the introduction 
when $\alpha \defequals 1$ corresponding to the single factorial function, 
$n!$, agrees with the known congruence for the 
Stirling numbers of the first kind derived in the reference 
(\cite[\S 4.6]{GFOLOGY}, \cite[\cf \S 5.8]{ADVCOMB}). 
In particular, for all $n \geq 1$ we can prove that 
\begin{align*} 
n! & \equiv 
     \sum_{m=1}^{n} 
     \binom{\lfloor n/2 \rfloor}{m - \lceil n/2 \rceil} 
     (-1)^{n-m} n^m + \Iverson{n = 0} 
     \pmod{2}. 
\end{align*} 
For comparison with the known result for the Stirling numbers of the 
first kind modulo $2$ expanded as in the result from the reference 
stated above, 
several particular cases of these congruences for the Stirling numbers, 
$\gkpSI{n}{m}$, modulo $2$ are given by 
\begin{align*} 
\gkpSI{n}{1} & \equiv 
     \frac{2^{n}}{4} \Iverson{n \geq 2} + \Iverson{n = 1} && \pmod{2} \\ 
\gkpSI{n}{2} & \equiv 
     \frac{3 \cdot 2^{n}}{16} (n-1) \Iverson{n \geq 3} + 
     \Iverson{n = 2} && \pmod{2} \\ 
\gkpSI{n}{3} & \equiv 
     2^{n-7} (9n-20) (n-1) \Iverson{n \geq 4} + 
     \Iverson{n = 3} && \pmod{2} \\ 
\gkpSI{n}{4} & \equiv 
     2^{n-9} (3n-10) (3n-7) (n-1) \Iverson{n \geq 5} + 
     \Iverson{n = 4} && \pmod{2} \\ 
\gkpSI{n}{5} & \equiv 
     2^{n-13} (27n^3-279n^2+934n-1008) (n-1) \Iverson{n \geq 6} + 
     \Iverson{n = 5} && \pmod{2} \\ 
\gkpSI{n}{6} & \equiv 
     \frac{2^{n-15}}{5} (9n^2-71n+120) (3n-14) (3n-11) (n-1) 
     \Iverson{n \geq 7} + \Iverson{n = 6} && \pmod{2}, 
\end{align*} 
where 
\begin{align*} 
\gkpSI{n}{m} & \equiv 
     \binom{\lfloor n/2 \rfloor}{m - \lceil n/2 \rceil} = 
     [x^m] \left( 
     x^{\lceil n/2 \rceil} (x+1)^{\lfloor n/2 \rfloor} 
     \right) && \pmod{2}, 
\end{align*} 
for all $n \geq m \geq 1$ (\cite[\S 4.6]{GFOLOGY}, \seqnum{A087755}). 
The termwise expansions of the row generating functions, $\Pochhammer{x}{n}$, 
for the Stirling number triangle considered modulo $3$ 
with respect to the non-zero indeterminate $x$ similarly imply the next 
properties of these coefficients for any $n \geq m > 0$. 
\begin{align*} 
\gkpSI{n}{m} & \equiv 
     [x^m] \left( 
     x^{\lceil n/3 \rceil} (x+1)^{\lceil (n-1)/3 \rceil} 
     (x+2)^{\lfloor n/3 \rfloor} 
     \right) && \pmod{3} \\ 
     & \equiv 
     \sum_{k=0}^{m} \binom{\lceil (n-1)/3 \rceil}{k} 
     \binom{\lfloor n/3 \rfloor}{m-k - \lceil n/3 \rceil} \times 
     2^{\lceil n/3 \rceil + \lfloor n/3 \rfloor -(m-k)} && \pmod{3} 
\end{align*} 
The next few particular examples of the 
special case congruences 
satisfied by the Stirling numbers of the first kind modulo $3$ 
obtained from the results in 
\eqref{eqn_S1CoeffsModuloh_DenomReflectedRoots_defs} 
above are expanded in the following forms: 
\begin{align*} 
\gkpSI{n}{1} & \equiv 
     \mathsmaller{ 
     \sum\limits_{j = \pm 1} 
     \frac{1}{36} \left(9-5 j\sqrt{3}\right) 
     \times \left(3+j\sqrt{3}\right)^{n} 
     } 
     \Iverson{n \geq 2} + \Iverson{n = 1} && \hspace{-3mm} \pmod{3} \\ 
\gkpSI{n}{2} & \equiv 
     \mathsmaller{ 
     \sum\limits_{j = \pm 1} 
     \frac{1}{216} \left((44n-41)-(25n-24) \cdot j\sqrt{3}\right) 
     \times \left(3+j\sqrt{3}\right)^{n} 
     } 
     \Iverson{n \geq 3} + \Iverson{n = 2} && \hspace{-3mm} \pmod{3} \\ 
\gkpSI{n}{3} & \equiv 
     \mathsmaller{ 
     \sum\limits_{j = \pm 1} 
     \frac{1}{15552} \left((1299n^2-3837n+2412)- 
     (745n^2-2217n+1418) \cdot j\sqrt{3}\right) 
     } \times && \\ 
     & \phantom{\equiv \sum\limits_{j = \pm 1} \frac{1}{15552}} \times 
     \mathsmaller{ 
     \left(3+j\sqrt{3}\right)^{n} 
     } 
     \Iverson{n \geq 4} + \Iverson{n = 3} && \hspace{-3mm} \pmod{3} \\ 
\gkpSI{n}{4} & \equiv 
     \mathsmaller{ 
     \sum\limits_{j = \pm 1} 
     \frac{1}{179936} \bigl((6409n^3-383778n^2+70901n-37092) 
     } && \\ 
     & \phantom{\equiv\mathsmaller{\sum\limits_{j=\pm 1} 
                \frac{1}{179936}}\bigl(} - 
     \mathsmaller{ 
     (3690n^3-22374n^2+41088n-21708) \cdot j\sqrt{3}\bigr) 
     } \times && \\ 
     & \phantom{\equiv \sum\limits_{j = \pm 1} \quad } \times 
     \mathsmaller{ 
     \left(3+j\sqrt{3}\right)^{n} 
     } 
     \Iverson{n \geq 5} + \Iverson{n = 4} && \hspace{-3mm} \pmod{3}. 
\end{align*} 
Additional congruences for the Stirling numbers of the first kind 
modulo $4$ and modulo $5$ are straightforward to expand by 
related formulas with exact 
algebraic expressions for the roots of the third-degree and 
fourth-degree equations defined as in 
\eqref{eqn_S1CoeffsModuloh_DenomReflectedRoots_defs}. 

\sublabel{Rational generating function expansions enumerating the 
          first-order harmonic numbers} 
The next several particular cases of the congruences for integers $p \geq 4$ 
satisfied by the first-order harmonic numbers, $H_n$ or $H_n^{(1)}$, 
are stated exactly in terms the rational generating functions in $z$ that 
lead to generalized forms of the congruences in the last equations 
modulo the integers $p \defequals 2, 3$ 
(\cite[\S 6.3]{GKP}, \seqnum{A001008}, \seqnum{A002805}). 
\begin{align*} 
n! \times H_n^{(1)} 
     & = 
     \gkpSI{n+1}{2} && \\ 
n! \times H_n^{(1)}     
     & \equiv 
     [z^{n+1}] \left( 
     \mathsmaller{ 
     \frac{36 z^2 - 48z + 325}{576} + 
     \frac{17040 z^2+1782 z+6467}{576 \left(24 z^3-36 z^2+12 z-1\right)}+\frac{78828 z^2-33987 z+3071}{288 \left(24
        z^3-36 z^2+12 z-1\right)^2} 
     } 
     \right) && \pmod{4} \\ 
     & \equiv 
     [z^{n}]\left( 
     \mathsmaller{ 
     \frac{3z-4}{48} + 
     \frac{1300 z^2+890 z+947}{96 \left(24 z^3-36 z^2+12 z-1\right)}+\frac{24568 z^2-10576 z+955}{96 \left(24 z^3-36 z^2+12 z-1\right)^2}
     } 
     \right) && \pmod{4} \\ 
     & \equiv 
     [z^{n-1}] \left( 
     \mathsmaller{ 
     \frac{1}{16} + 
     \frac{-96 z^2+794 z+397}{48 \left(24 z^3-36 z^2+12 z-1\right)}+\frac{5730 z^2-2453 z+221}{24 \left(24 z^3-36 z^2+12 z-1\right)^2} 
     } 
     \right) && \pmod{4} \\ 
n! \times H_n^{(1)} 
     & \equiv 
     [z^{n}]\Bigl( 
     \mathsmaller{ 
     \frac{12z-29}{300} + 
     \frac{80130 z^3+54450 z^2+79113 z+108164}{ 
     900 \left(120 z^4-240 z^3+120 z^2-20 z+1\right)} 
     } && \\ 
     & \phantom{\equiv [z^{n}]\Bigl( \quad} + 
     \mathsmaller{ 
     \frac{17470170 z^3-11428050 z^2+2081551 z-108077}{900
     \left(120 z^4-240 z^3+120 z^2-20 z+1\right)^2} 
     } 
     \Bigr) && \pmod{5} \\ 
n! \times H_n^{(1)} 
     & \equiv 
     [z^{n}]\Bigl( 
     \mathsmaller{ 
     \frac{10z-37}{360} + 
     \frac{1419408 z^4+903312 z^3+1797924 z^2+2950002 z+4780681}{2160 
     \left(720 z^5-1800 z^4+1200 z^3-300 z^2+30 z-1\right)} 
     } && \\ 
     & \phantom{\equiv [z^{n}]\Bigl( \quad } + 
     \mathsmaller{ 
     \frac{5581246248 z^4-4906594848z^3+1347715644 z^2-140481648 z+ 
     4780903}{2160 \left(720 z^5-1800 z^4+1200 z^3-300 z^2+30 z-1\right)^2}
     } 
     \Bigr) && \pmod{6} 
\end{align*} 

\begin{remark} 
For each $h \geq 1$, we have $h$-order finite difference 
equations for the Pochhammer symbol, 
$(R)_n = [z^n] \ConvGF{h \geq n}{1}{R}{z}$, implied by the rationality of the 
convergent functions, $\ConvGF{h}{1}{R}{z}$, in $z$ for all $h$, and by the 
expansions of the coefficients of the component convergent function sequences 
in \eqref{eqn_CoeffsOfzk_FQhaRz_restmts_by_GenProductSeqs-v1} and 
in \eqref{eqn_Chn_formula_stmts}. 
In particular, since $n! \cdot H_n$ is generated by the Pochhammer symbol as 
$\gkpSI{n+1}{2} = [R^2] (R)_{n+1}$, for all $n \geq 2$, we see that we also 
have the following identities following from the coefficients of the 
convergent-based generating functions, $\ConvGF{h}{1}{R}{z}$: 
\begin{align*} 
(n-1)! \cdot H_{n-1} & = [R^2]\left( 
     -\sum_{k=0}^{\min(n, h) - 1} \binom{h}{n-k} (1-h-R)_{n-k} (R)_k + 
     C_{h,n}(1, R) \right),\ 0 < n \leq h \\ 
(n-1)! \cdot H_{n-1} & \equiv [R^2]\left( 
     -\sum_{k=0}^{\min(n, h) - 1} \binom{h}{n-k} (1-h-R)_{n-k} (R)_k\right) 
     \pmod{h},\ n \geq h. 
\end{align*} 
\end{remark} 

\sublabel{Rational generating functions enumerating congruences for the 
          $r$-order harmonic numbers} 
The expansions of the integer-order harmonic number sequences 
cited in the reference \cite[\S 4.3.3]{MULTIFACTJIS} also 
yield additional related expansions of congruences for the 
terms, $(n!)^{r} \times H_n^{(r)}$, 
provided by the noted identities for these functions involving the 
Stirling numbers of the first kind modulo any fixed integers $p \geq 2$. 
The \emph{second-order} and \emph{third-order} \emph{harmonic numbers}, 
$H_n^{(2)}$ and $H_n^{(3)}$, respectively, 
are expanded exactly through the following formulas involving the 
Stirling numbers of the first kind modulo any fixed integers $p \geq 2$, and
where the Stirling number sequences, $\gkpSI{n}{m} \pmod{p}$ at each 
fixed $m \defequals 1,2,3,4$, are generated by the predictably 
rational functions of $z$ generated through the identities stated above 
(\cite[\S 4.3.3]{MULTIFACTJIS},
\seqnum{A007406}, \seqnum{A007407}, \seqnum{A007408}, \seqnum{A007409}): 
\begin{align*} 
(n!)^{2} \times H_n^{(2)} 
     & = 
     (n!)^{2} \times \sum_{k=1}^{n} \frac{1}{k^2} \\ 
     & \equiv 
     \gkpSI{n+1}{2}^{2} - 2 \gkpSI{n+1}{1} \gkpSI{n+1}{3} && \pmod{p} \\ 
(n!)^{3} \times H_n^{(3)} 
     & = 
     (n!)^{3} \times \sum_{k=1}^{n} \frac{1}{k^3} \\ 
     & \equiv 
     \gkpSI{n+1}{2}^{3} - 3 \gkpSI{n+1}{1} \gkpSI{n+1}{2} \gkpSI{n+1}{3} + 
     3 \gkpSI{n+1}{1}^{2} \gkpSI{n+1}{4}. 
     && \pmod{p} 
\end{align*} 
The Hadamard product, or diagonal-coefficient, generating function 
constructions formulated in the examples introduced by 
Section \ref{subSection_DiagonalGFSequences_Apps} below 
give expansions of rational convergent-function-based 
generating functions in $z$ that 
generate these corresponding $r$-order sequence cases modulo any 
fixed integers $p \geq 2$. 
The proof of {Theorem 2.4} given in the reference \cite[\S 2]{GFLECT} 
suggests the direct method for obtaining the next 
rational generating functions for these sequences in the 
working source code documented in the reference \cite{SUMMARYNBREF-STUB}, 
each of which generate series coefficients for these particular 
harmonic number sequence variants (modulo $p$) that 
always satisfy some finite-degree 
linear difference equation with constant coefficients over $n$ 
when $\alpha$ and $R$ are treated as constant parameters. 
\begin{align*} 
(n!)^{2} \times H_n^{(2)} 
     & \equiv 
     [z^{n}] \left( 
     \mathsmaller{ 
     \frac{z \left(1-3z+9z^2-8z^3\right)}{(1-4z)^2} 
     } 
     \right) 
     && \hspace{-5mm} \pmod{2} \\ 
     & = 
     [z^{n}] \left(
     \mathsmaller{ 
     \frac{5z}{16}-\frac{z^2}{2} + \frac{11}{64 (1-4z)^2} - 
     \frac{11}{64 (1-4z)} 
     } 
     \right) 
     \\ 
     & = 
     [z^{n}] \left( 
     \mathsmaller{ 
     z + 5z^2+33z^3+176z^4+880z^5+4224z^6+ \cdots 
     } 
     \right) 
     \\ 
(n!)^{2} \times H_n^{(2)}     
     & \equiv 
     [z^{n}] \left( 
     \mathsmaller{ 
     \frac{z \left(1-61 z+1339 z^2-13106 z^3+62284 z^4-144264 z^5+ 
     151776 z^6-124416 z^7+41472 z^8\right)}{(1-6 z)^3 
     \left(1-24 z+36 z^2\right)^2}
     } 
     \right) 
     && \hspace{-5mm} \pmod{3} \\ 
     & = 
     [z^{n}] \Bigl( 
     \mathsmaller{ 
     -\frac{13}{324}+\frac{14 z}{81}-\frac{4 z^2}{27} + 
     \frac{25}{1944 (-1+6 z)^3}+\frac{115}{1944 (-1+6 z)^2} + 
     \frac{5}{162 (-1+6 z)} 
     } \\ 
     & \phantom{= [z^{n}] \Bigl( \quad} + 
     \mathsmaller{ 
     \frac{-787+17624 z}{216
     \left(1-24 z+36 z^2\right)^2}+\frac{2377+3754 z}{648 
     \left(1-24 z+36 z^2\right)} 
     } 
     \Bigr) 
     \\ 
     & = 
     [z^{n}] \left( 
     \mathsmaller{ 
     z + 5z^2+49z^3+820z^4+18232z^5+437616z^6+ \cdots 
     } 
     \right) 
     \\ 
(n!)^{3} \times H_n^{(3)} 
     & \equiv 
     [z^{n}] \left( 
     \mathsmaller{ 
     \frac{z \left(1-7 z+49 z^2-144 z^3+192 z^4\right)}{(1-8z)^2} 
     } 
     \right) 
     && \hspace{-5mm} \pmod{2} \\ 
     & = 
     [z^{n}] \left( 
     \mathsmaller{ 
     \frac{11z}{32}-\frac{3z^2}{2}+3z^3 + 
     \frac{21}{256 (1-8z)^2} - \frac{21}{256 (1-8z)} 
     } 
     \right) 
     \\ 
     & = 
     [z^{n}] \left( 
     \mathsmaller{ 
     z +9z^2+129z^3+1344z^4+13440z^5+129024z^6+ \cdots 
     } 
     \right) 
     \\ 
(n!)^{3} \times H_n^{(3)}     
     & \equiv 
     [z^{n}] \Bigl( 
     \mathsmaller{ 
     -\frac{143}{5832}+\frac{625 z}{2916}-\frac{4 z^2}{9}+ 
     \frac{4 z^3}{3}+\frac{115 (-6719+711956 z)}{93312 
     \left(1-108 z+216 z^2\right)^2} 
     } 
     && \hspace{-5mm} \pmod{3} \\ 
     & \phantom{\equiv [z^{n}] \Bigl( \quad} + 
     \mathsmaller{ 
     \frac{774079+1459082 z}{93312 \left(1-108 z+216 z^2\right)} - 
     \frac{125 (-11+312 z)}{11664 \left(1-36 z+216 z^2\right)^4} 
     } \\ 
     & \phantom{\equiv [z^{n}] \Bigl(\quad} - 
     \mathsmaller{ 
     \frac{10 (1+306z)}{729 \left(1-36 z+216 z^2\right)^3} + 
     \frac{-20677+269268 z}{93312 \left(1-36 z+216 z^2\right)^2} + 
     \frac{11851+89478 z}{93312 \left(1-36 z+216z^2\right)} 
     } 
     \Bigr) 
     \\ 
     & = 
     [z^{n}] \left( 
     \mathsmaller{ 
     z+9z^2+251z^3+16280z^4+1586800z^5+171547200z^6+\cdots 
     } 
     \right) 
\end{align*} 
The next cases of the rational generating functions enumerating the 
terms of these two sequences modulo $4$ and $5$ 
lead to less compact formulas expanded in partial fractions over $z$, 
roughly approximated in form by the 
generating function expansions from the previous formulas. 
The factored denominators, 
denoted $\HNumGFFactoredDenomFn{r}{p}{z}$ immediately below, 
of the rational generating functions over the 
respective second-order and third-order cases of the $r$-order sequences 
modulo $p \defequals 4,5$ are provided for reference in the 
following equations: 
\begin{align*} 
\HNumGFFactoredDenomFn{2}{4}{z} & = 
     \mathsmaller{ 
     \left(-1+72 z-720 z^2+576 z^3\right)^2 
     \left(-1+36 z-288 z^2+576 z^3\right)^3 
     } \\ 
\HNumGFFactoredDenomFn{2}{5}{z} & = 
     \mathsmaller{ 
     \left(1-160 z+5040 z^2-28800 z^3+14400 z^4\right)^2 \times 
     } \\ 
     & \phantom{=\ } \times 
     \mathsmaller{ 
     \left(1-120 z+4680 z^2-76800 z^3+561600 z^4-1728000 z^5+ 
     1728000 z^6\right)^3 
     } \\ 
\HNumGFFactoredDenomFn{3}{4}{z} & = 
     \mathsmaller{ 
     (1-24z)^4 \left(-1+504 z-17280 z^2+13824 z^3\right)^2 \times 
     } \\ 
     & \phantom{= \quad} \times 
     \mathsmaller{ 
     \left(-1+144 z-5184 z^2+13824 z^3\right)^4 
     } 
     \left(-1+216 z-3456 z^2+13824 z^3\right)^4 \\ 
\HNumGFFactoredDenomFn{3}{5}{z} & = 
     \mathsmaller{ 
     \left(1-1520 z+273600 z^2-4320000 z^3+1728000 z^4\right)^2 \times 
     } \\ 
     & \phantom{= \quad } \times 
     \mathsmaller{ 
     \left(1-240 z+14400 z^2-288000 z^3+1728000 z^4\right)^4 \times 
     } \\ 
     & \phantom{= \quad } \times 
     \mathsmaller{ 
     \bigl(1-1680 z+1051200 z^2-319776000
        z^3+51914304000 z^4 
     } \\ 
     & \phantom{= \quad \times \bigl(1} 
     \mathsmaller{ - 
     4764026880000 z^5+251795865600000 z^6- 
        7537618944000000 z^7 
     } \\ 
     & \phantom{= \quad \times \bigl(1} 
     \mathsmaller{ + 
     121956544512000000 z^8-998751928320000000 z^9 
     } \\ 
     & \phantom{= \quad \times \bigl(1} 
     \mathsmaller{ + 
     4084826112000000000 z^{10} 
     -7739670528000000000 z^{11} 
     } \\ 
     & \phantom{= \quad \times \bigl(1} 
     \mathsmaller{ + 
     5159780352000000000 z^{12} 
     \bigr)^4 
     }. 
\end{align*} 
The summary notebook reference contains further complete expansions of the 
rational generating functions enumerating these $r$-order sequence 
cases for $r \defequals 1,2,3,4$ modulo the next few prescribed cases of the 
integers $p \geq 6$ 
(\cite[\cf \S 4.3.3]{MULTIFACTJIS}, \cite{SUMMARYNBREF-STUB}). 

\subsubsection{Generalized expansions of the new integer congruences for the 
               $\alpha$-factorial functions and the 
               symbolic product sequences} 
\label{subsubSection_GenpnAlphaR_Congr_Mod234} 

\sublabelII{Generalized forms of the special case results 
            expanded in the introduction} 
\begin{example}[The Special Cases Modulo $2$, $3$, and $4$] 
The first congruences for the $\alpha$-factorial functions, 
$n!_{(\alpha)}$, modulo the prescribed integer bases, $2$ and $2\alpha$, 
cited in \eqref{eqn_cor_Congruences_for_AlphaFactFns_modulo2} 
from the introduction result by applying 
Lemma \ref{lemma_GenConvFn_EnumIdents_pnAlphaRSeq_idents_combined_v1} 
to the series for the generalized convergent function, 
$\ConvGF{2}{\alpha}{R}{z}$, expanded by following equations: 
\begin{align*} 
p_n(\alpha, R) & \equiv 
     [z^n] \left( 
     \frac{1-z (2\alpha+R)}{R (\alpha+R) z^2 -2(\alpha+R) z + 1} 
     \right) 
     && \pmod{2, 2\alpha} \\ 
   & = 
     \sum_{b=\pm 1} 
     \frac{\left(\sqrt{\alpha  (\alpha +R)} -b \cdot \alpha\right) 
     \left(\alpha + b \cdot \sqrt{\alpha  (\alpha +R)}+R\right)^n}{ 
     2 \sqrt{\alpha (\alpha +R)}} 
     && \pmod{2, 2\alpha}. 
\end{align*} 
The next congruences for the $\alpha$-factorial function sequences 
modulo $3$ ($3\alpha$) and modulo $4$ ($4\alpha$) 
cited as particular examples in 
Section \ref{subSection_Intro_Examples} 
are established similarly by applying the previous lemma to the 
series coefficients of the next cases of the convergent functions, 
$\ConvGF{p}{\alpha}{R}{z}$, for $p \defequals 3,4$ and where 
$\alpha \defmapsto -\alpha$ and $R \defmapsto n$. 
\begin{align*} 
p_n(\alpha, R) & \equiv 
     [z^n] \left( 
     \frac{1 - 2(3\alpha+R) z + z^2 \left(R^2 + 4\alpha R + 6\alpha^2\right)}{ 
     1 - 3 (2\alpha+R) z + 3 (\alpha+R)(2\alpha+R) z^2 - 
     R (\alpha+R) (2\alpha+R) z^3} 
     \right) \\ 
     & \hspace{0.55\textwidth} \pmod{3, 3\alpha} \\ 
p_n(\alpha, R) & \equiv 
     [z^n] \left( 
     \mathsmaller{ 
     \frac{1-3(R+4 \alpha )z + z^2 \left(3 R^2+19 R \alpha +36 \alpha ^2\right)- (R+4 \alpha ) \left(R^2+3 R \alpha +6 \alpha ^2\right) z^3}{ 
     1-4(R+3 \alpha) z + 6 (R+2 \alpha ) (R+3 \alpha ) z^2 - 4(R+\alpha ) (R+2 \alpha ) (R+3 \alpha ) z^3 + R (R+\alpha ) (R+2 \alpha ) (R+3 \alpha ) z^4} 
     } 
     \right) \\ 
     & \hspace{0.55\textwidth} \pmod{4, 4\alpha} 
\end{align*} 
The particular cases of the new congruence properties satisfied 
modulo $3$ ($3\alpha$) and $4$ ($4\alpha$) cited in 
\eqref{eqn_AlphaFactFnModulo3_congruence_stmts} from 
Section \ref{subSection_Intro_Examples} of the 
introduction also phrase results that are expanded 
through exact algebraic formulas involving the reciprocal zeros 
of the convergent denominator functions, 
$\ConvFQ{3}{-\alpha}{n}{z}$ and $\ConvFQ{4}{-\alpha}{n}{z}$, given in 
\tableref{table_SpCase_Listings_Of_Qhz_ConvFn} 
\cite[\cf \S 1.11(iii), \S 4.43]{NISTHB}. 
The congruences cited in the example cases from the 
first two introduction sections to the article then 
correspond to the respective special cases of the reflected 
numerator polynomial sequences provided in 
\tableref{table_RelfectedConvNumPolySeqs_sp_cases}. 
\end{example} 

\sublabelII{Definitions related to the 
            reflected convergent numerator and denominator function 
            sequences} 
\begin{definition} 
For any $h \geq 1$, let the reflected convergent numerator and 
denominator function sequences be defined as follows: 
\begin{align*} 
\tagtext{Reflected Numerator Polynomials} 
\widetilde{\FP}_h(\alpha, R; z) & \defequals 
     z^{h-1} \times \ConvFP{h}{\alpha}{R}{z^{-1}} \\ 
\tagtext{Reflected Denominator Polynomials} 
\widetilde{\FQ}_h(\alpha, R; z) & \defequals 
     z^{h} \times \ConvFQ{h}{\alpha}{R}{z^{-1}} 
\end{align*} 
The listings given in 
\tableref{table_RelfectedConvNumPolySeqs_sp_cases} 
provide the first few simplified cases of the 
reflected numerator polynomial sequences 
which lead to the explicit formulations of the congruences 
modulo $p$ (and $p\alpha$) 
at each of the particular cases of $p \defequals 4, 5$ given in 
\eqref{eqn_AlphaFactFnModulo3_congruence_stmts}, and 
then for the next few small special cases for subsequent 
cases of the integers $p \geq 6$. 
\end{definition} 

\begin{definition} 
More generally, 
let the respective sequences of $p$-order roots, 
and then the corresponding sequences of 
reflected partial fraction coefficients, 
be defined as ordered sequences over any $p \geq 2$ and each 
$1 \leq i \leq p$ through some fixed ordering of the 
special function zeros defined by the next equations. 
\begin{align*} 
\tagtext{Sequences of Reflected Roots} 
\left( \ell_{p,i}^{(\alpha)}(R) \right)_{i=1}^{p} & \defequals 
     \left\{ z_i : \widetilde{\FQ}_h(\alpha, R; z_i) = 0,\ 
     1 \leq i \leq p \right\} \\ 
\tagtext{Sequences of Reflected Coefficients} 
\left( C_{p,i}^{(\alpha)}(R)  \right)_{i=1}^{p} & \defequals 
     \left( 
     \widetilde{\FP}_h\left(\alpha, R; \ell_{p,i}^{(\alpha)}(R)\right) 
     \right)_{i=1}^{p}. 
\end{align*} 
The sequences of reflected roots 
defined by the previous equations above in terms of the 
reflected denominator polynomials 
correspond to special zeros of the 
confluent hypergeometric function, $\HypU{-h}{b}{w}$, and the 
associated Laguerre polynomials, $L_p^{(\beta)}(w)$, defined as in the 
special zero sets from 
Section \ref{subSection_Intro_GenConvFn_Defs_and_Properties} 
of the introduction 
\cite[\S 18.2(vi), \S 18.16]{NISTHB} 
\cite{LGWORKS-ASYMP-SPFNZEROS2008,PROPS-ZEROS-CHYPFNS80}. 
\end{definition} 

\sublabel{Generalized statements of the exact formulas and 
          new congruence expansions} 
The notation in the previous definition is then employed by the 
next restatements of the 
exact formula expansions and congruence properties cited in the 
initial forms of the expansions given in 
\eqref{eqn_AlphaFactFn_Exact_PartialFracsRep_v1} and 
\eqref{eqn_AlphaFactFn_Exact_PartialFracsRep_v2} of 
Section \ref{subSection_Intro_GenConvFn_Defs_and_Properties}. 
In particular, for any $p \geq 2$, $h \geq n \geq 1$, and 
where $\alpha, R \neq 0$ correspond to fixed 
integer-valued (or symbolic indeterminate) parameters, the 
expansions of the generalized product sequence cases defined by 
\eqref{eqn_GenFact_product_form} satisfy the following relations 
(see Remark \ref{remark_Congruences_for_Rational-Valued_Params} above): 
\renewcommand{\rootri}[4]{ 
     \ensuremath{\ell_{#1,#2}^{\left(#3\right)}\left(#4\right)}
} 
\label{eqn_AlphaFactFnModulop_congruence_stmts} 
\begin{align*} 
\tagonce\label{eqn_AlphaFactFnModulop_congruence_stmts-pn_stmts_v1} 
\pn{n}{\alpha}{R} & = 
     \sum\limits_{1 \leq i \leq h} 
     \frac{C_{h,i}^{(\alpha)}(R)}{\prod\limits_{j \neq i} 
     \left(\rootri{h}{i}{\alpha}{R} - \rootri{h}{j}{\alpha}{R}\right)} 
     \times \left( \rootri{h}{i}{\alpha}{R} \right)^{n+1},\ 
     && \quad \forall h \geq n \geq 1 \\ 
\pn{n}{\alpha}{R} & \equiv 
     \sum\limits_{1 \leq i \leq p} 
     \frac{C_{p,i}^{(\alpha)}(R)}{\prod\limits_{j \neq i} 
     \left(\rootri{p}{i}{\alpha}{R} - \rootri{p}{j}{\alpha}{R}\right)} 
     \times \left( \rootri{p}{i}{\alpha}{R} \right)^{n+1} 
     && \pmod{p, p\alpha, p\alpha^{2}, \ldots, p\alpha^{p}}. 
\end{align*} 
The expansions in the next equations similarly state the 
desired results stating the generalized forms of the 
congruence formulas for the $\alpha$-factorial functions, 
$\MultiFactorial{n}{\alpha}$, 
given by the particular special case expansions in 
\eqref{eqn_AlphaFactFnModulo3_congruence_stmts} from 
Section \ref{subsubSection_Examples_NewCongruences} of the 
introduction to the article. 
\begin{align*} 
\tagonce 
n!_{(\alpha)} & = 
     \sum\limits_{1 \leq i \leq h} 
     \frac{C_{h,i}^{(-\alpha)}(n)}{\prod\limits_{j \neq i} 
     \left(\rootri{h}{i}{-\alpha}{n} - \rootri{h}{j}{-\alpha}{n}\right)} 
     \times \left( \rootri{h}{i}{-\alpha}{n} 
     \right)^{\left\lfloor \frac{n-1}{\alpha} \right\rfloor + 1},\ 
     && \quad \forall h \geq n \geq 1 && \\ 
n!_{(\alpha)} & \equiv 
     \undersetbrace{\defequals R_p^{(\alpha)}(n) 
     \text{ in Section \ref{subsubSection_Examples_NewCongruences} and 
            in \tableref{table245}}}{ 
     \sum\limits_{1 \leq i \leq p} 
     \frac{C_{p,i}^{(-\alpha)}(n)}{\prod\limits_{j \neq i} 
     \left(\rootri{p}{i}{-\alpha}{n} - \rootri{p}{j}{-\alpha}{n}\right)} 
     \times \left( \rootri{p}{i}{-\alpha}{n} 
     \right)^{\left\lfloor \frac{n-1}{\alpha} \right\rfloor + 1} 
     } 
     && \pmod{p, p\alpha, p\alpha^{2}, \ldots, p\alpha^{p}}. 
\end{align*} 
The first pair of expansions given in 
\eqref{eqn_AlphaFactFnModulop_congruence_stmts-pn_stmts_v1} 
for the generalized product sequences, $\pn{n}{\alpha}{R}$, 
provide exact formulas and the corresponding 
new congruence properties for the Pochhammer symbol and 
Pochhammer $k$-symbol, 
in the respective special cases where 
$\Pochhammer{x}{n} \defmapsto \alpha^{-n} \pn{n}{\alpha}{\alpha x}$ and 
$\Pochhammer{x}{n,\alpha} \defmapsto \pn{n}{\alpha}{x}$ in the 
equations above. 

\subsection{Applications of 
            rational diagonal-coefficient generating functions and 
            Hadamard product sequences 
            involving the generalized convergent functions} 
\label{subSection_DiagonalGFSequences_Apps} 

\subsubsection{Generalized definitions and 
               coefficient extraction formulas for 
               sequences involving products of rational generating functions} 

We define the next extended notation for the 
\emph{Hadamard product} generating functions, \\ 
$\left(\FiGF{1} \odot \FiGF{2}\right)(z)$ and 
$\left(\FiGF{1} \odot \cdots \odot \FiGF{k}\right)(z)$, 
at some fixed, formal $z \in \mathbb{C}$. 
Phrased in slightly different wording, 
we define \eqref{eqn_HProdGFs_kGen_def_v1} as an alternate notation for the 
\emph{diagonal} \emph{generating functions} that enumerate the 
corresponding product sequences 
generated by the diagonal coefficients of the multiple-variable 
product series in 
$k$ formal variables treated as in the reference \cite[\S 6.3]{ECV2}. 
\begin{align} 
\label{eqn_HProdGFs_kGen_def_v1} 
\FiGF{1} \odot \FiGF{2} \odot \cdots \odot \FiGF{k} & 
     \defequals 
     \sum_{n \geq 0} f_{1,n} f_{2,n} \cdots f_{k,n} \times z^{n} 
     \quad \text{ where } \quad 
     \FiGF{i}(z) \defequals \sum_{n \geq 0} f_{i,n} z^{n} 
     \text{ for } 
     1 \leq i \leq k 
\end{align} 
When $\FiGF{i}(z)$ is a rational function of $z$ for each $1 \leq i \leq k$, 
we have particularly nice expansions of the 
coefficient extraction formulas of the 
rational diagonal generating functions from the references 
(\cite[\S 6.3]{ECV2}, \cite[\S2.4]{GFLECT}). 
In particular, 
when $\FiGF{i}(z)$ is rational in $z$ at each respective $i$,
these rational generating functions are expanded through the 
next few useful formulas: 
\begin{align*} 
\tagtext{Diagonal Coefficient Extraction Formulas} 
\FiGF{1} \odot \FiGF{2} & = 
     [x_1^0] \Biggl( 
     \FiGF{2}\left(\frac{z}{x_1}\right) \cdot 
     \FiGF{1}(x_1) 
     \Biggr) \\ 
\tagonce\label{eqn_kGenHProdGFs_RationalDiagonalGF_Idents-stmts_v1} 
\FiGF{1} \odot \FiGF{2} \odot \FiGF{3} & = 
     [x_2^0 x_1^0] \Biggl( 
     \FiGF{3}\left(\frac{z}{x_2}\right) \cdot 
     \FiGF{2}\left(\frac{x_2}{x_1}\right) \cdot 
     \FiGF{1}(x_1) 
     \Biggr) \\ 
\FiGF{1} \odot \FiGF{2} \odot \cdots \odot \FiGF{k} & = 
     [x_{k-1}^0 \cdots x_2^0 x_1^0] \Biggl( 
     \FiGF{k}\left(\frac{z}{x_{k-1}}\right) \cdot 
     \FiGF{k-1}\left(\frac{x_{k-1}}{x_{k-2}}\right) 
     \times \cdots \times 
     \FiGF{2}\left(\frac{x_{2}}{x_{1}}\right) \cdot 
     \FiGF{1}(x_1) 
     \Biggr). 
\end{align*} 

\begin{remark}[Integral Representations] 
Analytic formulas for the Hadamard products, 
$\FiGF{1} \odot \FiGF{2} = \FiGF{1}(z) \odot \FiGF{2}(z)$, 
when the component sequence generating functions are 
well enough behaved in some neighborhood of $z_0 = 0$ 
are given in the references 
(\cite[\S 1.12(V); Ex.\ 1.30, p.\ 85]{ADVCOMB}, \cite[\S 6.10]{ACOMB-BOOK}). 
In particular, we compare the first formula in 
\eqref{eqn_kGenHProdGFs_RationalDiagonalGF_Idents-stmts_v1} when $k := 2$ 
to the next known integral formula when the power series for 
both sequence generating functions, 
$\FiGF{1}(z)$ and $\FiGF{2}(z)$, are absolutely convergent in 
some $|z| \leq r < 1$ given by  
\begin{align*} 
\left(\FiGF{1} \odot \FiGF{2}\right)(z^2) & = 
     \frac{1}{2\pi} \int_0^{2\pi} 
     \FiGF{1}\left(z e^{\imath t}\right) 
     \FiGF{2}\left(z e^{-\imath t}\right) dt. 
\end{align*} 
The integral representation in the last equation 
is easily proved directly by noticing that 
\begin{align*} 
\frac{1}{2\pi} \int_0^{2\pi} [z^n] \FiGF{1}\left(z e^{\imath t}\right) 
     \FiGF{2}\left(z e^{-\imath t}\right) dt & = 
     \sum_{0 \leq k \leq n} \frac{1}{2\pi} \int_0^{2\pi} 
     e^{-(n-2k) \imath t} f_1(k) f_2(n-k) dt \\ 
     & = 
     \sum_{0 \leq k \leq n} f_1(k) f_2(n-k) \cdot \Iverson{n = 2k} \\ 
     & = 
     \begin{cases} 
        f_1\left(\frac{n}{2}\right) f_2\left(\frac{n}{2}\right) & 
        \text{ if $n$ is even; } \\ 
        0 & \text{ if $n$ is odd. } 
     \end{cases} 
\end{align*} 
\end{remark} 

We regard the rational convergents approximating the 
otherwise divergent 
ordinary generating functions for the generalized factorial function 
sequences strictly as formal power series in $z$ 
whenever possible in this article. 
The remaining examples in this section illustrate this more 
formal approach using the generating functions enumerating the 
factorial-related product sequences considered here. 
The next several subsections of the article 
aim to provide concrete applications and 
some notable special cases illustrating the utility of this 
approach to the more general formal sequence products enumerated through the 
rational convergent functions, 
especially when combined with other 
generating function techniques discussed elsewhere and in the references. 

\subsubsection{Examples: Constructing hybrid rational 
               generating function approximations from the 
               convergent functions enumerating the generalized 
               factorial product sequences} 
\label{subsubSection_remark_HybridDiagonalHPGFs} 

When one of the generating functions of an individual sequence from the 
Hadamard product representations in 
\eqref{eqn_kGenHProdGFs_RationalDiagonalGF_Idents-stmts_v1} 
is not rational in $z$, we still proceed, however slightly more carefully, to 
formally enumerate the terms of these sequences that arise in applications. 
For example, 
the \emph{central binomial coefficients} 
are enumerated by the next convergent-based generating functions 
whenever $n \geq 1$ 
(\cite[\cf \S 5.3]{GKP}, \seqnum{A000984}). 
\begin{align*} 
\tagtext{Central Binomial Coefficients} 
\binom{2n}{n} & = 
     \frac{2^{2n}}{n!} \times \Pochhammer{1 / 2}{n} \ \ \ = 
     [z^{n}] [x^{0}] \Biggl( 
     e^{2x} \ConvGF{n}{2}{1}{\frac{z}{x}} 
     \Biggr) \\ 
     & = 
     \frac{2^{n}}{n!} \times (2n-1)!! = 
     [z^{n}] [x^{1}] \Biggl( 
     e^{2x} \ConvGF{n}{-2}{2n-1}{\frac{z}{x}} 
     \Biggr). 
\end{align*} 
Since the reciprocal factorial squared terms, $(n!)^{-2}$, 
are generated by the power series for the 
\emph{modified Bessel function of the first kind}, 
$I_0(2\sqrt{z}) = \sum_{n \geq 0} z^{n} / (n!)^{2}$, these 
central binomial coefficients are also enumerated as the 
diagonal coefficients of the 
following convergent function products 
(\cite[\S 10.25(ii)]{NISTHB}, \cite[\S 5.5]{GKP},
\cite[\cf \S 1.13(II)]{ADVCOMB}): 
\begin{align*} 
\binom{2n}{n} & = 
     \frac{2^{2n} \Pochhammer{1}{n} \Pochhammer{\frac{1}{2}}{n}}{(n!)^2} \\ 
     & = 
     [x_1^0 x_2^0 z^n] \left( 
     \ConvGF{n}{2}{2}{\frac{z}{x_2}} \ConvGF{n}{2}{1}{\frac{x_2}{x_1}} 
     I_0\left(2\sqrt{x_1}\right) 
     \right) \\ 
     & = 
     \frac{(2n)!! (2n-1)!!}{(n!)^2} \\ 
     & = 
     [x_1^0 x_2^0 z^n] \left( 
     \ConvGF{n}{-2}{2n}{\frac{z}{x_2}} \ConvGF{n}{-2}{2n-1}{\frac{x_2}{x_1}} 
     I_0\left(2\sqrt{x_1}\right) 
     \right). 
\end{align*} 
The next binomial coefficient product sequence is enumerated through a 
similar construction of the convergent-based generating function 
identities expanded in the previous equations. 
\begin{align*} 
\tagtext{Paired Binomial Coefficient Products} 
\binom{3n}{n} \binom{2n}{n} & = 
     \frac{3^{3n} \Pochhammer{\frac{1}{3}}{n} 
     \Pochhammer{\frac{2}{3}}{n}}{(n!)^2} \\ 
     & = 
     [x_1^0 x_2^0 z^n] \left( 
     \ConvGF{n}{3}{2}{\frac{3z}{x_2}} \ConvGF{n}{3}{1}{\frac{x_2}{x_1}} 
     I_0\left(2\sqrt{x_1}\right) 
     \right) \\ 
     & = 
     \frac{3^n}{(n!)^2} \times 
     \AlphaFactorial{3n-1}{3} \AlphaFactorial{3n-2}{3} \\ 
     & = 
     [x_1^0 x_2^0 z^n] \left( 
     \ConvGF{n}{-3}{3n-1}{\frac{3z}{x_2}} 
     \ConvGF{n}{-3}{3n-2}{\frac{x_2}{x_1}} 
     I_0\left(2\sqrt{x_1}\right) 
     \right) 
\end{align*} 

\sublabel{Generating ratios of factorial functions and 
          binomial coefficients} 
The next few identities for the convergent generating function products 
over the binomial coefficient variants cited in 
\eqref{eqn_HybridDiagCoeffHPGFs_BinomCoeff_Examples-exps_v1} from the 
introduction are generated as the diagonal coefficients 
of the corresponding products of the convergent functions convolved with 
arithmetic progressions extracted from the exponential series 
in the form of the following equation, 
where $\omega_a \defequals \exp\left(2\pi\imath / a\right)$ denotes the 
primitive $a^{th}$ root of unity for integers $a \geq 2$ 
(\cite[\S 1.2.9]{TAOCPV1}, \cite[Ex.\ 1.26, p.\ 84]{ADVCOMB}): 
\begin{align*} 
\tagonce\label{eqn_HybridDiagCoeffHPGFs_BinomCoeff_Examples-exps_ExpGFs_v2} 
\widehat{E}_a(z) & \defequals \sum_{n \geq 0} \frac{z^n}{(an)!} = 
     \frac{1}{a}\left( 
     e^{z^{1/a}} + e^{\omega_a \cdot z^{1/a}} + 
     e^{\omega_a^2 \cdot z^{1/a}} + \cdots + 
     e^{\omega_a^{a-1} \cdot z^{1/a}} 
     \right),\ 
     a > 1. 
\end{align*} 
The modified generating functions, $\widehat{E}_a(z) = E_{a,1}(z)$, 
correspond to special cases of the 
\emph{Mittag-Leffler function} defined as in the 
reference by the following series \cite[\S 10.46]{NISTHB}: 
\begin{align*} 
E_{a,b}(z) & \defequals 
     \sum_{n \geq 0} \frac{z^n}{\Gamma(an+b)},\ a > 0. 
\end{align*} 
These modified exponential series generating functions then denote the 
power series expansions of arithmetic progressions over the 
coefficients of the ordinary generating function 
for the exponential series sequences, $f_n \defequals 1 / n!$ and 
$f_{an} = 1 / (an)!$. 
For $a \defequals 2, 3, 4$, the 
particular cases of these exponential series generating functions 
are given by 
\begin{align*} 
\widehat{E}_2(z) & = \cosh\left(\sqrt{z}\right) \\ 
\widehat{E}_3(z) & = 
     \frac{1}{3}\left( 
     e^{z^{1/3}} + 2 e^{-\frac{z^{1/3}}{2}} 
     \mathsmaller{\cos\left(\frac{\sqrt{3} \cdot z^{1/3}}{2}\right)} 
     \right) \\ 
\widehat{E}_4(z) & = 
     \frac{1}{2}\left( 
     \cos\left(z^{1/4}\right) + \cosh\left(z^{1/4}\right) 
     \right), 
\end{align*} 
where the powers of the $a^{th}$ roots of unity in these 
special cases satisfy 
$\omega_2 = -1$, 
$\omega_3 = \frac{\imath}{2}\left(\imath + \sqrt{3}\right)$, 
$\omega_3^2 = -\frac{\imath}{2}\left(-\imath + \sqrt{3}\right)$, and 
$\left(\omega_4^{m}\right)_{1 \leq m \leq 4} = \left(\imath, -1, -\imath, 1\right)$ 
(see the computations given in the reference \cite{SUMMARYNBREF-STUB}). 

The next particular special cases of these 
diagonal-coefficient generating functions 
corresponding to the binomial coefficient sequence variants from 
\eqref{eqn_HybridDiagCoeffHPGFs_BinomCoeff_Examples-exps_v1} of 
Section \ref{subsubSection_Intro_Examples_Fact-RelatedSeqs_GenByTheConvFns} 
are then given through the following coefficient extraction identities 
provided by \eqref{eqn_kGenHProdGFs_RationalDiagonalGF_Idents-stmts_v1} 
(\seqnum{A166351}, \seqnum{A066802}): 
\begin{align*} 
\frac{(6n)!}{(3n)!} & = 
     \frac{6^{6n} 
     \bcancel{\Pochhammer{1}{n}} 
     \bcancel{\Pochhammer{\frac{2}{6}}{n}} 
     \bcancel{\Pochhammer{\frac{3}{6}}{n}} \times 
     \Pochhammer{\frac{1}{6}}{n} 
     \Pochhammer{\frac{3}{6}}{n} 
     \Pochhammer{\frac{5}{6}}{n} 
     }{ 
     3^{3n} 
     \bcancel{\Pochhammer{1}{n}} 
     \bcancel{\Pochhammer{\frac{1}{3}}{n}} 
     \bcancel{\Pochhammer{\frac{2}{3}}{n}} 
     \phantom{\qquad}} \\ 
     & = 
     24^{n} \times 6^n \Pochhammer{1/6}{n} \times 
     2^{n} \Pochhammer{1/2}{n} \times 6^n \Pochhammer{5/6}{n} \\ 
     & = 
     [x_2^0 x_1^0 z^n] \left( 
     \ConvGF{n}{6}{5}{\frac{24z}{x_2}} \ConvGF{n}{2}{1}{\frac{x_2}{x_1}} 
     \ConvGF{n}{6}{1}{x_1} 
     \right) \\ 
     & = 
     8^n \times \AlphaFactorial{6n-5}{6} 
     \AlphaFactorial{6n-3}{6} \AlphaFactorial{6n-1}{6} \\ 
     & = 
     [x_2^0 x_1^0 z^n] \Biggl( 
     \ConvGF{n}{-6}{6n-5}{\frac{8z}{x_2}} 
     \ConvGF{n}{-6}{6n-3}{\frac{x_2}{x_1}} \times \\ 
     & \phantom{= [x_2^0 x_1^0 z^n] \Biggl(\ } \times 
     \ConvGF{n}{-6}{6n-1}{x_1} 
     \Biggr) \\ 
\binom{6n}{3n} & = 
     [x_3^0 x_2^0 x_1^0 z^n] \Biggl( 
     \ConvGF{n}{6}{5}{\frac{8z}{x_3}} \ConvGF{n}{6}{3}{\frac{x_3}{x_2}} 
     \ConvGF{n}{6}{1}{\frac{x_2}{x_1}} \times \\ 
     & \phantom{= [x_3^0 x_2^0 x_1^0 z^n] \Biggl( \quad} \times 
     \undersetbrace{\widehat{E}_3(x_1) = E_{3,1}(x_1)}{
     \frac{1}{3}\left( 
     e^{x_1^{1/3}} + 2 e^{-\frac{x_1^{1/3}}{2}} 
     \cos\left(\frac{\sqrt{3} \cdot x_1^{1/3}}{2}\right) 
     \right) 
     } 
     \Biggr) \\ 
     & = 
     [x_3^0 x_2^0 x_1^0 z^n] \Biggl( 
     \ConvGF{n}{-6}{6n-5}{\frac{8z}{x_3}} 
     \ConvGF{n}{-6}{6n-3}{\frac{x_3}{x_2}} \times \\ 
     & \phantom{= [x_3^0 x_2^0 x_1^0 z^n] \Biggl( \quad} \times 
     \ConvGF{n}{-6}{6n-1}{\frac{x_2}{x_1}} \times \\ 
     & \phantom{= [x_3^0 x_2^0 x_1^0 z^n] \Biggl( \quad} \times 
     \undersetbrace{\widehat{E}_3(x_1)}{
     \frac{1}{3}\left( 
     e^{x_1^{1/3}} + 2 e^{-\frac{x_1^{1/3}}{2}} 
     \cos\left(\frac{\sqrt{3} \cdot x_1^{1/3}}{2}\right) 
     \right) 
     } 
     \Biggr). 
\end{align*} 
Similarly, the following related sequence cases forming particular 
expansions of these binomial coefficient variants are generated by 
\begin{align*} 
\binom{8n}{4n} & = 
     \frac{2^{16n}}{(4n)!} \times 
     \mathsmaller{\Pochhammer{\frac{1}{8}}{n} 
     \Pochhammer{\frac{3}{8}}{n} \Pochhammer{\frac{5}{8}}{n} 
     \Pochhammer{\frac{7}{8}}{n}} \\ 
     & = 
     [x_1^0 x_2^0 x_3^0 x_4^0 z^n] \Biggl( 
     \ConvGF{n}{8}{7}{\frac{16z}{x_4}} \ConvGF{n}{8}{5}{\frac{x_4}{x_3}} 
     \ConvGF{n}{8}{3}{\frac{x_3}{x_2}} \times \\ 
     & \phantom{= [x_1^0 x_2^0 x_3^0 x_4^0 z^n] \Biggl( \quad} \times 
     \ConvGF{n}{8}{1}{\frac{x_2}{x_1}} \times 
     \undersetbrace{\widehat{E}_4(x_1) = E_{4,1}(x_1)}{
     \frac{1}{2}\left( 
     \cos\left(x_1^{1/4}\right) + \cosh\left(x_1^{1/4}\right) 
     \right) 
     } 
     \Biggr) \\ 
     & = 
     \frac{2^{4n}}{(4n)!} \times 
     \AlphaFactorial{8n-7}{8} \AlphaFactorial{8n-5}{8} 
     \AlphaFactorial{8n-3}{8} \AlphaFactorial{8n-3}{8} \\ 
     & = 
     [x_1^0 x_2^0 x_3^0 x_4^0 z^n] \Biggl( 
     \ConvGF{n}{-8}{8n-7}{\frac{16z}{x_4}} 
     \ConvGF{n}{-8}{8n-5}{\frac{x_4}{x_3}} \times \\ 
     & \phantom{= [x_1^0 x_2^0 x_3^0 x_4^0 z^n] \Biggl( \quad} \times 
     \ConvGF{n}{-8}{8n-3}{\frac{x_3}{x_2}} 
     \ConvGF{n}{-8}{8n-1}{\frac{x_2}{x_1}} \times 
     \widehat{E}_4(x_1) 
     \Biggr). 
\end{align*} 

\sublabel{Generating the subfactorial function (sequence of derangements)} 
Another pair of convergent-based generating function identities 
enumerating the sequence of 
\emph{subfactorials}, $\left(!n\right)_{n \geq 1}$, 
or \emph{derangements}, $\left(n\?\right)_{n \geq 1}$, 
are expanded for $n \geq 1$ as follows 
(see Example \ref{remark_FactSumIdents_SubfactorialSums_ConvIdents} and the 
related examples cited in 
Section \ref{subsubSection_Apps_Example_SumsFactFn_Seqs}) 
(\cite[\S 5.3]{GKP}, \cite[\cf \S 8.4]{NISTHB}, \seqnum{A000166}): 
\begin{align*} 
\tagtext{Subfactorial Function} 
!n & \defequals 
     n! \times \sum_{i=0}^{n} \frac{(-1)^{i}}{i!} 
     \quad \seqmapsto{A000166} 
     \left(0, 1, 2, 9, 44, 265, 1854, 14833, \ldots \right) \\ 
     & \phantom{:} = 
     [z^n x^0] \left( 
     \frac{e^{-x}}{(1-x)} \times \ConvGF{n}{-1}{n}{\frac{z}{x}} 
     \right) \\ 
     & \phantom{:} = 
     [x^0 z^n] \left( 
     \frac{e^{-x}}{(1-x)} \times \ConvGF{n}{1}{1}{\frac{z}{x}} 
     \right). 
\end{align*} 

\begin{remark}[Laplace-Borel Transformations of Formal Power Series] 
\label{remark_Formal_Laplace-Borel_Transforms} 
The sequence of subfactorials is enumerated through the 
previous equations as the diagonals of generating function products 
where the rational convergent functions, $\ConvGF{n}{\alpha}{R}{z}$, 
generate the sequence multiplier of $n!$ 
corresponding to the (formal) \emph{Laplace transform}, 
$\mathcal{L}(f(t); z)$, defined by the integral transformations in the next 
equations 
(\cite[\cf \S 2.2]{FLAJOLET80B}, \cite[\S B.4]{ACOMB-BOOK},
\cite[p.\ 566]{GKP}). 
In this case the integral formulas are 
applied termwise to the power series given by the 
exponential generating function, $f(x) \defequals e^{-x} / (1-x)$, 
for this sequence. 
\begin{align*} 
\tagtext{Laplace-Borel Transformation Integrals} 
\mathcal{L}(\widehat{F}; z) & = 
     \phantom{\frac{1}{2\pi}} \int_0^{\infty} e^{-t} \widehat{F}(tz) dt \\ 
\mathcal{L}^{-1}(\widetilde{F}; z) & = 
     \frac{1}{2\pi} \int_0^{2\pi} 
     \widetilde{F}\left(-z e^{-\imath s}\right) 
     e^{-e^{\imath s}} ds. 
\end{align*} 
The applications cited in 
Section \ref{subsubSection_Apps_Example_SumsFactFn_Seqs} and 
Section \ref{subsubSection_Apps_Example_SumsOfPowers_Seqs} 
in this article below 
employ this particular generating function technique to enumerate the 
factorial function multipliers provided by 
these rational convergent functions in several particular cases of 
sequences involving finite sums over factorial functions, 
sums of powers sequences, and new forms of approximate 
generating functions for the binomial coefficients and 
sequences of binomials. 
\end{remark} 

\subsection{Examples: Expanding arithmetic progressions of the 
            single factorial function} 
\label{subsubSection_Apps_ArithmeticProgs_of_the_SgFactFns} 

One application suggested by the results in the previous subsection 
provides $a$-fold reductions of the 
$h$-order series approximations otherwise required to 
exactly enumerate 
arithmetic progressions of the single factorial function according to the 
next result. 
\begin{align} 
\label{eqn_AKPlusB_FactFn_Conv_Ident-stmt_v1} 
\left(an+r\right)! & = 
     [z^{an+r}] \ConvGF{h}{-1}{an+r}{z},\ 
     \forall n \geq 1, a \in \mathbb{Z}^{+}, 0 \leq r < a, 
     \forall h \geq an+r 
\end{align} 
The arithmetic progression sequences of the single factorial function 
formed in the 
particular special cases when $a \defequals 2, 3$ 
(and then for particular cases of $a \defequals 4,5$) 
expanded in the 
examples cited below illustrate the utility to these convergent-based 
formal generating function approximations. 

\begin{prop}[Factorial Function Multiplication Formulas]
The statement of 
\emph{Gauss's multiplication formula} for the gamma function yields the 
following decompositions of the single factorial functions, $(an+r)!$, 
into a finite product over $a$ of the integer-valued multiple factorial 
sequences defined by \eqref{eqn_nAlpha_Multifact_variant_rdef} for 
$n \geq 1$ and whenever $a \geq 2$ and $0 \leq r < a$ 
are fixed natural numbers 
(\cite[\S5.5(iii)]{NISTHB}, \cite[\S 2]{ATLASOFFUNCTIONS}, 
\cite{WOLFRAMFNSSITE-INTRO-FACTBINOMS}): 
\begin{align} 
\label{eqn_prop_factfn_mult_formulas} 
(an+r)! & = 
     \undersetbrace{ \mathsmaller{(an+r)! = 
     \prod\limits_{i=0}^{a-1} \AlphaFactorial{an+r-i}{a} = 
     \prod\limits_{i=0}^{a-1} \pn{n}{-a}{an+r-i} }}{ 
     \AlphaFactorial{an+r}{a} \times \AlphaFactorial{an+r-1}{a} 
     \times \cdots \times 
     \AlphaFactorial{an+r-a+1}{a} 
     } \\ 
\notag 
(an+r)! & = 
     \undersetbrace{ \mathsmaller{(an+r)! = 
     \prod\limits_{i=1}^{a} a^{n} \times \Pochhammer{\frac{r+i}{a}}{n} = 
     \prod\limits_{i=1}^{a} \pn{n}{a}{r+i} } }{ 
     r! \cdot a^{an} \Pochhammer{\frac{1+r}{a}}{n} 
     \Pochhammer{\frac{2+r}{a}}{n} \times \cdots \times 
     \Pochhammer{\frac{a-1+r}{a}}{n} 
     \Pochhammer{\frac{a+r}{a}}{n} 
     } \\ 
\notag 
     & = 
     r! \cdot a^{an} \Pochhammer{1+\frac{r}{a}}{n} 
     \Pochhammer{1+\frac{r-1}{a}}{n} \times \cdots \times 
     \Pochhammer{1+\frac{r-a+1}{a}}{n},\ 
     \forall a,n \in \mathbb{Z}^{+}, r \geq 0. 
\end{align} 
\end{prop} 
\begin{proof} 
The first identity corresponds to the expansions of the 
single factorial function by a product of $\alpha$ distinct 
$\alpha$-factorial functions for any fixed integer $\alpha \geq 2$ in the 
following forms: 
\begin{align*} 
n! & = n!! \cdot (n-1)!! \\ 
     & = 
     n!!! \cdot (n-1)!!! \cdot (n-2)!!! \\ 
     & = 
     \prod_{i=0}^{\alpha-1} \AlphaFactorial{n-i}{\alpha},\ 
     \alpha \in \mathbb{Z}^{+}. 
\end{align*} 
The expansions of the last two identities stated above also follow from the 
known \emph{multiplication formula} for the Pochhammer symbol 
expanded by the next equation \cite{WOLFRAMFNSSITE-INTRO-FACTBINOMS} 
for any fixed integers $a \geq 1$ and $r \geq 0$, and where 
$(an+r)! = \Pochhammer{1}{an+r}$ by 
Lemma \ref{lemma_GenConvFn_EnumIdents_pnAlphaRSeq_idents_combined_v1}. 
\begin{align*} 
\tagtext{Pochhammer Symbol Multiplication Formula} 
\Pochhammer{x}{an+r} & = 
     \Pochhammer{x}{r} \times a^{an} \times 
     \prod_{j=0}^{a-1} \Pochhammer{\frac{x+j+r}{a}}{n} 
\end{align*} 
The two identities involving the corresponding products of the 
sequences from \eqref{eqn_GenFact_product_form} 
provided by the braced formulas in 
\eqref{eqn_prop_factfn_mult_formulas} 
follow similarly from the lemma, and moreover, 
lead to several direct expansions of the convergent-function-based 
product sequences identities expanded in the next examples. 
\end{proof} 

\subsubsection{Expansions of arithmetic progression sequences involving the 
               double factorial function ($a := 2$)} 
In the particular cases where $a \defequals 2$ (with $r \defequals 0, 1$), 
we obtain the following forms of the corresponding alternate expansions of 
\eqref{eqn_AKPlusB_FactFn_Conv_Ident-stmt_v1} 
enumerated by the diagonal coefficients of the next 
convergent-based product generating functions for all $n \geq 1$ 
(\cite[\cf \S 2]{ATLASOFFUNCTIONS}, \seqnum{A010050}, \seqnum{A009445}): 
\begin{align*} 
\tagtext{Double Factorial Function Expansions} 
(2n)! 
      & = 2^{n} n! \times (2n-1)!! \\ 
      & = [z^n] [x^0] \left( 
          \ConvGF{n}{-1}{n}{\frac{2z}{x}} 
          \ConvGF{n}{-2}{2n-1}{x} 
          \right) \\ 
      & = 2^{n} n! \times 2^{n} \Pochhammer{1/2}{n} \\ 
      & = [x^0 z^n] \left( 
          \ConvGF{n}{1}{1}{\frac{2z}{x}} \ConvGF{n}{2}{1}{x} 
          \right) \\ 
(2n+1)! 
      & = 2^{n} n! \times (2n+1)!! \\ 
      & = [z^n] [x^0] \left( 
          \ConvGF{n}{-1}{n}{\frac{2z}{x}} 
          \ConvGF{n}{-2}{2n+1}{x} 
          \right) \\ 
      & = 2^{n} n! \times 2^{n} \Pochhammer{3/2}{n} \\ 
      & = [x^1 z^n] \left( 
          \ConvGF{n}{1}{1}{\frac{2z}{x}} \ConvGF{n}{2}{1}{x} 
          \right) \\ 
      & = [x^0 z^n] \left( 
          \ConvGF{n}{1}{1}{\frac{2z}{x}} \ConvGF{n}{2}{3}{x} 
          \right). 
\end{align*} 

\subsubsection{Expansions of arithmetic progression sequences involving the 
               triple factorial function ($a := 3$)}

When $a \defequals 3$ we similarly obtain the next few 
alternate expansions generating the triple factorial products 
for the arithmetic progression sequences in 
\eqref{eqn_AKPlusB_FactFn_Conv_Ident-stmt_v1} 
stated in the following equations for any $n \geq 2$ 
by extending the constructions of the identities for the 
expansions of the double factorial products in the previous equations 
(\cite[\S 2]{ATLASOFFUNCTIONS}, 
\seqnum{A100732}, \seqnum{A100089}, \seqnum{A100043}): 
\begin{align*} 
\tagtext{Triple Factorial Function Expansions} 
(3n)! & = (3n)!!! \times (3n-1)!!! \times (3n-2)!!! \\ 
      & = [z^n] [x_2^0 x_1^0] \left( 
          \ConvGF{n}{-1}{n}{\frac{3z}{x_2}} 
          \ConvGF{n}{-3}{3n-1}{\frac{x_2}{x_1}} 
          \ConvGF{n}{-3}{3n-2}{x_1} 
          \right) \\ 
      & = 3^n n! \times 3^n \Pochhammer{2/3}{n} \times 
          3^n \Pochhammer{1/3}{n} \\ 
      & = [x_1^{0} x_2^{0} z^{n}] \left( 
          \ConvGF{n}{1}{1}{\frac{3z}{x_2}} 
          \ConvGF{n}{3}{1}{\frac{x_2}{x_1}} 
          \ConvGF{n}{3}{2}{x_1} 
          \right) \\ 
(3n+1)! & = (3n)!!! \times (3n-1)!!! \times (3(n+1)-2)!!! \\ 
        & = [z^n] [x_2^0 x_1^{-1}] \Biggl( 
          \ConvGF{n}{-1}{n}{\frac{3z}{x_2}} 
          \ConvGF{n}{-3}{3n-1}{\frac{x_2}{x_1}} \times \\ 
        & \phantom{[z^n] [x_2^0 x_1^{-1}] \Biggl(\ } \times 
          \ConvGF{n}{-3}{3n+1}{x_1} 
          \Biggr) \\ 
      & = 3^n n! \times 3^n \Pochhammer{2/3}{n} \times 
          3^n \Pochhammer{4/3}{n} \\ 
      & = [x_1^{0} x_2^{0} z^{n}] \left( 
          \ConvGF{n}{1}{1}{\frac{3z}{x_2}} 
          \ConvGF{n}{3}{4}{\frac{x_2}{x_1}} 
          \ConvGF{n}{3}{2}{x_1} 
          \right) \\ 
(3n+2)! & = (3n)!!! \times (3(n+1)-1)!!! \times (3(n+1)-2)!!! \\ 
        & = [z^n] [x_2^{-1} x_1^0] \Biggl( 
          \ConvGF{n}{-1}{n}{\frac{3z}{x_2}} 
          \ConvGF{n}{-3}{3n+2}{\frac{x_2}{x_1}} \times \\ 
        & \phantom{= [z^n] [x_2^{-1} x_1^0] \Biggl(\ } \times 
          \ConvGF{n}{-3}{3n+1}{x_1} 
          \Biggr) \\ 
      & = 2 \times 3^n n! \times 3^n \Pochhammer{5/3}{n} \times 
          3^n \Pochhammer{4/3}{n} \\ 
      & = [x_1^{1} x_2^{0} z^{n}] \left( 
          \ConvGF{n}{1}{1}{\frac{3z}{x_2}} 
          \ConvGF{n}{3}{4}{\frac{x_2}{x_1}} 
          \ConvGF{n}{3}{2}{x_1} 
          \right). 
\end{align*} 

\subsubsection{Other special cases involving the 
               quadruple and quintuple factorial functions 
	       ($ a:= 4,5$)}
               
\label{subsubSection_SpCaseArithProgsOfSgFactFn_alphaEq45} 

The additional forms of the diagonal-coefficient generating functions 
corresponding to the special cases of the sequences in 
\eqref{eqn_AKPlusB_FactFn_Conv_Ident-stmt_v1} where 
$(a, r) \defequals (4, 2)$ and $(a, r) \defequals (5, 3)$, respectively 
involving the \emph{quadruple} and \emph{quintuple factorial} functions 
are also cited in the next equations to further illustrate the procedure 
outlined by the previous two example cases. 
\begin{align*} 
(4n+2)! & = (4n)!!!! \times (4n-1)!!!! \times (4(n+1)-2)!!!! \times 
            (4(n+1)-3)!!!! \\ 
      & = [z^n] [x_3^{0} x_2^{-1} x_1^0] \Biggl( 
          \ConvGF{n}{-1}{n}{\frac{4z}{x_3}} 
          \ConvGF{n}{-4}{4n-1}{\frac{x_3}{x_2}} \times \\ 
      & \phantom{= [z^n] [x_3^{0} x_2^0 x_1^0] \Biggl(} \times 
          \ConvGF{n}{-4}{4n+2}{\frac{x_2}{x_1}} 
          \ConvGF{n}{-4}{4n+1}{x_1} 
          \Biggr) \\ 
      & = 
      2 \times 4^{4n} \times 
      \Pochhammer{1}{n} \Pochhammer{3/4}{n} 
      \Pochhammer{3/2}{n} \Pochhammer{5/4}{n} \\ 
      & = [x_1^0 x_2^1 x_3^0 z^n] \Biggl( 
          \ConvGF{n}{1}{1}{\frac{4z}{x_3}} 
          \ConvGF{n}{4}{3}{\frac{x_3}{x_2}} 
          \ConvGF{n}{4}{2}{\frac{x_2}{x_1}} \times \\ 
      & \phantom{= [x_1^0 x_2^1 x_3^0 z^n] \Biggl(} \times 
          \ConvGF{n}{4}{1}{x_1} 
          \Biggr),\ n \geq 2 \\ 
(5n+3)! & = (5n)!_{(5)} \times (5n-1)!_{(5)} \times (5(n+1)-2)!_{(5)} \times 
            (5(n+1)-3)!_{(5)} \times (5n+1)!_{(5)} \\ 
        & = [z^n] [x_4^0 x_3^{-1} x_2^0 x_1^0] \Biggl( 
            \ConvGF{n}{-1}{n}{\frac{5z}{x_4}} 
            \ConvGF{n}{-5}{5n-1}{\frac{x_4}{x_3}} \times \\ 
        & \phantom{= [z^n] [x_4^0 x_3^{-1} x_2^0 x_1^0] \Biggl( } \times 
            \ConvGF{n}{-5}{5n+3}{\frac{x_3}{x_2}} 
            \ConvGF{n}{-5}{5n+2}{\frac{x_2}{x_1}} \times \\ 
        & \phantom{= [z^n] [x_4^0 x_3^{-1} x_2^0 x_1^0] \Biggl( } \times 
            \ConvGF{n}{-5}{5n+1}{x_1} 
            \Biggr) \\ 
      & = 
      6 \times 5^{5n} \times 
      \Pochhammer{1}{n} \Pochhammer{4/5}{n} \Pochhammer{8/5}{n} 
      \Pochhammer{7/5}{n} \Pochhammer{6/5}{n} \\ 
      & = [x_1^0 x_2^0 x_3^1 x_4^0 z^n] \Biggl( 
          \ConvGF{n}{1}{1}{\frac{5z}{x_4}} 
          \ConvGF{n}{5}{4}{\frac{x_4}{x_3}} 
          \ConvGF{n}{5}{3}{\frac{x_3}{x_2}} \times \\ 
      & \phantom{= [x_1^0 x_2^0 x_3^1 x_4^0 z^n] \Biggl(} \times 
          \ConvGF{n}{5}{2}{\frac{x_2}{x_1}} 
          \ConvGF{n}{5}{1}{x_1} 
          \Biggr),\ n \geq 2 
\end{align*} 
Convergent-based generating function identities enumerating 
specific expansions corresponding to other cases of 
\eqref{eqn_AKPlusB_FactFn_Conv_Ident-stmt_v1} 
when $a \defequals 4,5$ are given in 
\cite{SUMMARYNBREF-STUB}. 

\begin{remark}[Generating Arithmetic Progressions of a Sequence] 
\label{remark_GFArithmeticProg_formulas} 
The truncated power series approximations generating the 
single factorial functions 
formulated in the last few special case examples expanded in this section 
are also compared to the known results for 
extracting arithmetic progressions 
from any formal ordinary power series generating function, 
$F(z) = \sum_n f_n z^n$, of an arbitrary sequence, 
$\langle f_n \rangle_{n=0}^{\infty}$, 
through the primitive $a^{th}$ roots of unity, 
$\omega_a \defequals \exp\left(2\pi\imath / a\right)$, 
stated in the references as 
(\cite[\S 1.2.9]{TAOCPV1}, \cite[Ex.\ 1.26, p.\ 84]{ADVCOMB}) 
\begin{align*} 
F_{an+b}(z) & := \sum_{n \geq 0} f_{an+b} z^{an+b} = 
     \frac{1}{a} \sum_{0 \leq m < a} \omega_a^{-mb} 
     F\left(\omega_a^{m}\right), 
\end{align*} 
for integers $a \geq 2$ and $0 \leq b < a$ 
(compare to Remark \ref{remark_Lemmas_ArithmeticProgGFs_nOvermFloored} in 
Section \ref{subSection_Apps_and_Examples_StmtsOfLemmas}). 
This formula is also employed in the special cases of the 
exponential series generating functions defined in 
\eqref{eqn_HybridDiagCoeffHPGFs_BinomCoeff_Examples-exps_ExpGFs_v2} of the 
previous subsection. 
\end{remark} 

\subsection{Examples: Generalized superfactorial function products and 
            relations to the Barnes $G$-function} 

The \emph{superfactorial function}, $S_1(n)$, 
also denoted by $S_{1,0}(n)$ in the notation of 
\eqref{eqn_SAlphadn_GenSuperFactFnSeqs_product_based_def_v1} below, 
is defined for integers $n \geq 1$ 
by the factorial products (\seqnum{A000178}):
\begin{align*} 
\tagtext{Superfactorial Function} 
S_1(n) & \defequals \prod_{k \leq n} k! 
     \quad \seqmapsto{A000178} 
     \left(1, 2, 12, 288, 34560, 24883200, \ldots \right). 
\end{align*} 
These superfactorial functions are given in terms of the 
\emph{Barnes G-function}, $G(z)$, for $z \in \mathbb{Z}^{+}$ through the 
relation $S_1(n) = G(n+2)$. 
The Barnes G-function, $G(z)$, corresponds to a so-termed 
\quotetext{\emph{double gamma function}} satisfying a 
functional equation of the following form 
for natural numbers $n \geq 1$ 
\cite[\S 5.17]{NISTHB} \cite{CONTRIB-THEORY-BARNESGFN}: 
\begin{equation*} 
\tagtext{Barnes G-Function} 
G(n+2) = \Gamma(n+1) G(n+1) \Iverson{n > 1} + \Iverson{n = 1}. 
\end{equation*} 
We can similarly expand the superfactorial function, $S_1(n)$, by 
unfolding the factorial products in the previous definition recursively 
according to the formulas given in the next equation. 
\begin{align*} 
S_1(n) & = 
     n! \cdot (n-1)! \times \cdots \times (n-k+1)! \cdot S_1(n-k),\ 
     0 \leq k < n 
\end{align*} 
The product sequences over the single factorial functions 
formed by the last equations then 
lead to another application of the 
diagonal-coefficient product generating functions 
involving the rational convergent 
functions that enumerate the functions, $(n-k)!$, when $n-k \geq 1$. 
In particular, these particular cases of the 
diagonal coefficient, Hadamard-product-like sequences involving the 
single factorial function are generated as the coefficients 
\begin{align*} 
S_1(n) = 
     \left([z^{n}] \ConvGF{n}{-1}{n}{z}\right) & \times 
     \left([z^{n}] z \cdot \ConvGF{n}{-1}{n-1}{z}\right) 
     \times \\ 
     & \times 
     \left([z^{n}] z^2 \cdot \ConvGF{n}{-1}{n-2}{z}\right) 
     \times \cdots \times \\ 
     & \times 
     \left([z^{n}] z^{n} \cdot \ConvGF{n}{-1}{1}{z}\right) \\ 
S_1(n) = 
     \left([z^{n}] \ConvGF{n}{1}{1}{z}\right) 
     \phantom{_{+1}-} & \times 
     \left([z^{n}] z \cdot \ConvGF{n}{1}{1}{z}\right) 
     \times \\ 
     & \times 
     \left([z^{n}] z^2 \cdot \ConvGF{n}{1}{1}{z}\right) 
     \times \cdots \times 
     \left([z^{n}] z^{n} \cdot \ConvGF{n}{1}{1}{z}\right). 
\end{align*} 
Stated more precisely, the 
superfactorial sequence is generated by the following 
finite, rational products of the generalized convergent functions 
for any $n \geq 2$: 
\begin{align*} 
\tagonce\label{eqn_SuperFactFn_S1n_RationalHP-DiagGF_ProductIdents-stmts_v1} 
S_1(n) 
     & = 
     [x_1^{-1} x_2^{-1} \cdots x_{n-1}^{-1} x_n^{n}] 
     \Biggl( 
     \prod_{i=0}^{n-2} 
     \ConvGF{n}{-1}{n-i}{\frac{x_{n-i}}{x_{n-i-1}}} \times 
     \ConvGF{n}{-1}{1}{x_1} 
     \Biggr) \\ 
     & = 
     [x_1^{-1} x_2^{-1} \cdots x_{n-1}^{-1} x_n^{n}] 
     \Biggl( 
     \prod_{i=0}^{n-2} 
     \ConvGF{n}{1}{1}{\frac{x_{n-i}}{x_{n-i-1}}} \times 
     \ConvGF{n}{1}{1}{x_1} 
     \Biggr). 
\end{align*} 

\subsubsection{Generating generalized superfactorial product sequences} 

Let the more general superfactorial functions, $S_{\alpha,d}(n)$, 
forming the analogous products of the integer-valued 
multiple, $\alpha$-factorial function cases from 
\eqref{eqn_nAlpha_Multifact_variant_rdef} 
correspond to the expansions defined by the next equation. 
\begin{align*} 
\tagonce\label{eqn_SAlphadn_GenSuperFactFnSeqs_product_based_def_v1} 
S_{\alpha,d}(n) & \defequals 
     \prod_{j=1}^{n} (\alpha j - d)!_{(\alpha)},\ 
     n \geq 1, \alpha \in \mathbb{Z}^{+}, 0 \leq d < \alpha 
\end{align*} 
Observe that the corollary of 
Lemma \ref{lemma_GenConvFn_EnumIdents_pnAlphaRSeq_idents_combined_v1} 
cited in \eqref{eqn_AlphaFactFn_anm1_SpCase_SeqIdents-stmts_v0} 
implies that 
whenever $n \geq 1$, and for any fixed $\alpha \in \mathbb{Z}^{+}$, 
we immediately obtain the next identity 
corresponding to the so-termed 
\quotetext{ordinary} case of these 
superfactorial functions, $S_1(n) = S_{1,0}(n)$, 
in the notation for these sequences defined above. 
\begin{align*} 
S_1(n) = \alpha^{-\binom{n+1}{2}} 
     \prod_{j=1}^{n} (\alpha j)!_{(\alpha)} = 
     \alpha^{-\binom{n+1}{2}} S_{\alpha,0}(n),\ 
     \forall \alpha \in \mathbb{Z}^{+},\ n \geq 1. 
\end{align*} 
For other cases of the parameter $d > 0$, the 
generalized superfactorial function products defined by 
\eqref{eqn_SAlphadn_GenSuperFactFnSeqs_product_based_def_v1} 
are enumerated in a similar fashion to the previous 
constructions of the convergent-based generating function 
identities expanded by 
\eqref{eqn_SuperFactFn_S1n_RationalHP-DiagGF_ProductIdents-stmts_v1} 
in the following forms for $n \geq 1$, $\alpha \in \mathbb{Z}^{+}$, 
and any fixed $0 \leq d < \alpha$: 
\begin{align*} 
S_{\alpha,d}(n) 
     & = 
     [x_1^{-1} x_2^{-1} \cdots x_{n-1}^{-1} x_n^{n}] 
     \Biggl( 
     \prod_{i=0}^{n-2} 
     \ConvGF{n}{-\alpha}{\alpha(n-i)-d}{\frac{x_{n-i}}{x_{n-i-1}}} 
     \ConvGF{n}{-\alpha}{\alpha-d}{x_1} 
     \Biggr) \\ 
\tagonce\label{eqn_SuperFactFn_S1n_RationalHP-DiagGF_ProductIdents-stmts_v2} 
     & = 
     [x_1^{-1} x_2^{-1} \cdots x_{n-1}^{-1} x_n^{n}] 
     \Biggl( 
     \prod_{i=0}^{n-2} 
     \ConvGF{n}{\alpha}{\alpha-d}{\frac{x_{n-i}}{x_{n-i-1}}} \times 
     \ConvGF{n}{\alpha}{\alpha-d}{x_1} 
     \Biggr). 
\end{align*} 

\subsubsection{Special cases of the 
               generalized superfactorial products and their relations to the 
               Barnes G-function at rational $z$} 

The special case sequences formed by the double factorial products, 
$S_{2,1}(n)$, and the quadruple factorial products, $S_{4,2}(n)$, 
are simplified by \Mm{} to obtain the next 
closed-form expressions given by \cite{SUMMARYNBREF-STUB} 
\begin{align*} 
S_{2,1}(n) & \defequals \prod_{j=1}^{n} (2j-1)!! = 
     \frac{A^{3/2}}{2^{1/24} e^{1/8} \pi^{1/4}} \cdot 
     \frac{2^{n(n+1)/2}}{\pi^{n/2}} \times 
     G\left(n + \frac{3}{2}\right) \\ 
S_{4,2}(n) & \defequals \prod_{j=1}^{n} (4j-2)!!!! = 
     \frac{A^{3/2}}{2^{1/24} e^{1/8} \pi^{1/4}} \cdot 
     \frac{4^{n(n+1)/2}}{\pi^{n/2}} \times 
     G\left(n + \frac{3}{2}\right), 
\end{align*} 
where $A \approx 1.2824271$ denotes \emph{Glaisher's constant} 
\cite[\S 5.17]{NISTHB}, and where the particular constant multiples 
in the previous equation correspond to the special case values, 
$\Gamma(1/2) = \sqrt{\pi}$ and 
$G(3/2) = A^{-3/2} 2^{1/24} e^{1/8} \pi^{1/4}$ 
\cite{CONTRIB-THEORY-BARNESGFN}. 
In addition, 
since the sequences defined by 
\eqref{eqn_SAlphadn_GenSuperFactFnSeqs_product_based_def_v1} 
are also expanded as the products 
\begin{align*} 
S_{\alpha,d}(n) & = 
     \prod_{j=1}^{n} 
     \undersetbrace{= \alpha^{j} \times 
     \frac{\Gamma\left(j+1-\frac{d}{\alpha}\right)}{ 
     \Gamma\left(1-\frac{d}{\alpha}\right)}}{
     \left( 
     \alpha^{j} \times \Pochhammer{1 - \frac{d}{\alpha}}{j} 
     \right) 
     } = 
     \alpha^{\binom{n+1}{2}} G(n+2) \times 
     \prod_{j=1}^{n} \binom{j-\frac{d}{\alpha}}{j}, 
\end{align*} 
further computations with \Mm{} yield the next 
few representative special cases of these generalized superfactorial 
functions when $\alpha \defequals 3, 4, 5$
\cite[\cf \S 2]{CONTRIB-THEORY-BARNESGFN} \cite{SUMMARYNBREF-STUB}: 
\begin{align*} 
\tagtext{Special Case Products} 
S_{3,1}(n) \defequals 
     \prod_{j=1}^{n} (3j-1)!!! & = 
     3^{n(n-1)/2} \left( \frac{2 \cdot G\left(\frac{5}{3}\right)}{ 
     G\left(\frac{8}{3}\right)}\right)^{n} 
     \times \frac{G\left(n+\frac{5}{3}\right)}{G\left(\frac{5}{3}\right)} \\ 
S_{4,1}(n) \defequals  
\prod_{j=1}^{n} (4j-1)!!!! & = 
     4^{n(n-1)/2} \left( \frac{3 \cdot G\left(\frac{7}{4}\right)}{ 
     G\left(\frac{11}{4}\right)}\right)^{n} 
     \times \frac{G\left(n+\frac{7}{4}\right)}{G\left(\frac{7}{4}\right)} \\ 
S_{5,1}(n) \defequals 
\prod_{j=1}^{n} \AlphaFactorial{5j-1}{5} & = 
     5^{n(n-1)/2} \left( \frac{4 \cdot G\left(\frac{9}{5}\right)}{ 
     G\left(\frac{14}{5}\right)}\right)^{n} 
     \times \frac{G\left(n+\frac{9}{5}\right)}{G\left(\frac{9}{5}\right)} \\ 
S_{5,2}(n) \defequals 
\prod_{j=1}^{n} \AlphaFactorial{5j-2}{5} & = 
     5^{n(n-1)/2} \left( \frac{3 \cdot G\left(\frac{8}{5}\right)}{ 
     G\left(\frac{13}{5}\right)}\right)^{n} 
     \times \frac{G\left(n+\frac{8}{5}\right)}{G\left(\frac{8}{5}\right)}. 
\end{align*} 
We are then led to conjecture inductively, without proof 
given in this example, that these sequences satisfy the form of the 
next equation involving the Barnes G-function over the 
rational-valued inputs prescribed according to the 
next formula for $n \geq 1$. 
\begin{align*} 
\tagtext{Generalized Superfactorial Function Identity} 
S_{\alpha,d}(n) 
     & = 
     \frac{\alpha^{\binom{n}{2}} (\alpha - d)^{n}}{ 
     \Gamma\left(2 - \frac{d}{\alpha}\right)^{n}} \times 
     \frac{G\left(n + 2 - \frac{d}{\alpha}\right)}{ 
     G\left(2 - \frac{d}{\alpha}\right)}. 
\end{align*} 

\begin{remark}[Generating Rational-Valued Cases of the Barnes $G$-Function] 
The identities for the $\alpha$-factorial function products given in the 
previous examples suggest further avenues to 
generating other particular forms of the Barnes G-function formed by these 
generalized integer-parameter product sequence cases. 
These functions are generated by extending the constructions of the 
rational generating function methods outlined in this section 
\cite{CONTRIB-THEORY-BARNESGFN,ON-HYPGEOMFNS-PHKSYMBOL}, 
which then suggest additional identities for the Barnes G-functions, 
$G(z+2)$, over rational-valued $z > 0$ 
involving the special function zeros already defined by 
Section \ref{subSection_Intro_GenConvFn_Defs_and_Properties} and in 
Section \ref{subSection_Congruences_for_Series_ModuloIntegers_p}. 
The convergent-based generating function identities stated in the 
previous equations also suggest further applications to 
enumerating specific new identities corresponding to the 
special case constant formulas expanded in 
\cite[\S 2]{CONTRIB-THEORY-BARNESGFN}. 
\end{remark} 

\begin{remark}[Expansions of Hyperfactorial Function Products]  
The generalized superfactorial sequences defined by 
\eqref{eqn_SAlphadn_GenSuperFactFnSeqs_product_based_def_v1} in the 
previous example are also related to the \emph{hyperfactorial function}, 
$H_1(n)$, defined for $n \geq 1$ by the products (\seqnum{A002109}) 
\begin{align*} 
\tagtext{Hyperfactorial Products} 
H_1(n) & \defequals 
     \prod_{1 \leq j \leq n} j^j 
     \quad \seqmapsto{A002109} 
     \left(1, 4, 108, 27648, 86400000, \ldots \right). 
\end{align*} 
The exercises in the reference state additional known formulas 
establishing relations between these 
expansions of the hyperfactorial function defined above, and 
products of the binomial coefficients, including the following identities 
(\cite[\S 5; Ex.\ 5.13, p.\ 527]{GKP}, \seqnum{A001142}): 
\begin{align*} 
\tagtext{Binomial Coefficient Products} 
B_1(n) & \defequals 
     \prod_{k=0}^{n} \binom{n}{k} = 
     \frac{(n!)^{n+1}}{S_1(n)^{2}} = 
     \frac{H_1(n)}{S_1(n)} = 
     \frac{H_1(n)^{2}}{(n!)^{n+1}}. 
\end{align*} 
Statements of congruence properties and other relations 
connecting these sequences are considered in 
\cite{GENWTHM-DBLHYPERSUPER-FACTFNS,
       CONTRIB-THEORY-BARNESGFN,ON-HYPGEOMFNS-PHKSYMBOL}. 
\end{remark} 

\subsection{Examples: Enumerating sequences involving 
            sums of factorial-related functions, 
            sums of factorial powers, and more challenging 
            combinatorial sums involving factorial functions} 
\label{subsubSection_Apps_Example_SumsFactFn_Seqs} 

The coefficients of the convergent-based generating function 
constructions for the factorial product sequences 
given in the previous subsection are compared to the next several identities 
expanding the corresponding sequences of finite sums involving 
factorial functions 
(\cite[\cf \S 3; Ex.\ 3.30 p.\ 168]{ADVCOMB}, 
\seqnum{A003422}, \seqnum{A005165}, \seqnum{A033312},
\seqnum{A001044}, \seqnum{A104344}, \seqnum{A061062}). 
\begin{align*} 
\tagtext{Left Factorials} 
L!n & \defequals \sum_{k=0}^{n-1} k! = 
     [z^n] \left( 
     \frac{z}{(1-z)} \cdot \ConvGF{n}{1}{1}{z} 
     \right) \\ 
\tagtext{Alternating Factorials} 
\af(n) & \defequals \sum_{k=1}^{n} (-1)^{n-k} \cdot k! = 
     [z^n] \left( 
     \frac{1}{(1+z)} \cdot \left(\ConvGF{n}{1}{1}{z} - 1\right) 
     \right) \\ 
\tagonce\label{eqn_FactorialSumIdents_examples_afn_sf2n_sf3n-stmts_v1} 
\Sf_2(n) & \defequals \sum_{k=1}^{n} k \cdot k! = 
       (n+1)! - 1 \\ 
     & \phantom{:} = 
     [x^0 z^n] \left( 
     \frac{1}{(1-z)} \frac{x}{(1-x)^2} 
     \ConvGF{n}{1}{1}{\frac{z}{x}} 
     \right) \\ 
\tagtext{Sums of Single Factorial Powers} 
\Sf_3(n) & \defequals \sum_{k=1}^{n} (k!)^{2} \\ 
     & \phantom{:} = 
     [x^0 z^n] \left( 
     \frac{1}{(1-z)} \times 
     \left( 
     \ConvGF{n}{1}{1}{x} \ConvGF{n}{1}{1}{\frac{z}{x}} - 1 
     \right) 
     \right) \\ 
\Sf_4(n) & \defequals \sum_{k=0}^{n} (k!)^{3} \\ 
     & \phantom{:} = 
     [x_1^0 x_2^0 z^n] \left( 
     \frac{1}{(1-z)} \times 
     \ConvGF{n}{1}{1}{\frac{z}{x_2}} \ConvGF{n}{1}{1}{\frac{x_2}{x_1}} 
     \ConvGF{n}{1}{1}{x_1} + 1 
     \right) 
\end{align*} 
One generalization of the second identity given in 
\eqref{eqn_FactorialSumIdents_examples_afn_sf2n_sf3n-stmts_v1} 
due to Gould is stated in \cite[p.\ 168]{ADVCOMB} 
as the formula 
\begin{align*} 
\sum_{k=0}^{n} \binom{x}{k}^{p} \left(\frac{k!}{x^{k+1}}\right)^{p} 
     \left( (x-k)^{p} - x^{p} \right) & = 
     \binom{x}{n+1}^{p} \left(\frac{(n+1)!}{x^{n+1}}\right)^{p} - 1. 
\end{align*} 
The \href{http://mathworld.wolfram.com/FactorialSums.html}{MathWorld} 
site providing an overview of results for factorial-related sums contains 
references to definitions of several other factorial-related 
finite sums and series in addition to those special cases defined in the 
first few of the labeled equations in 
\eqref{eqn_FactorialSumIdents_examples_afn_sf2n_sf3n-stmts_v1}. 

\subsubsection{Generating sums of double and triple factorial powers} 
The expansion of the second to last sum, denoted $\Sf_3(n)$ in 
\eqref{eqn_FactorialSumIdents_examples_afn_sf2n_sf3n-stmts_v1}, is 
generalized to form the following variants of sums over the squares of the 
$\alpha$-factorial functions, $n!!$ and $n!!!$, 
through the generating function identities given in 
\eqref{eqn_MultFactFn_ConvSeq_def_v2} of the introduction 
(\seqnum{A184877}): 
\begin{align*} 
\tagtext{Sums of Double Factorial Squares} 
\Sf_{3,2}(n) & \defequals \sum_{k=0}^{n} (k!!)^{2} \\ 
\notag 
     & \phantom{:} = 
     [x^0 z^n] \Biggl( 
     \frac{1}{(1-z)} \times \biggl( 
     \ConvGF{n}{2}{2}{x} \ConvGF{n}{2}{2}{\frac{z^2}{x}} \\ 
     & \phantom{\defequals [x^0 z^n] \Biggl(\frac{1}{(1-z)} \times 
                \biggl( \quad} + 
     z \cdot \ConvGF{n}{2}{3}{x} \ConvGF{n}{2}{3}{\frac{z^2}{x}} 
     \biggr) 
     \Biggr) \\ 
\notag 
     & \phantom{:} = 
     [x^0 z^n] \Biggl( 
     \frac{1}{(1-z)} \times \biggl( 
     \ConvGF{n}{2}{2}{x} \ConvGF{n}{2}{2}{\frac{z^2}{x}} \\ 
     & \phantom{\phantom{:} = [x^0 z^n] \Biggl(\frac{1}{(1-z)} \times 
                \biggl( \quad } + 
     z^{-1} \cdot \ConvGF{n}{2}{1}{x} \ConvGF{n}{2}{1}{\frac{z^2}{x}} - 1 
     \biggr) 
     \Biggr) \\ 
\tagtext{Sums of Triple Factorial Squares} 
\Sf_{3,3}(n) & \defequals \sum_{k=0}^{n} (k!!!)^{2} \\ 
\notag 
     & \phantom{:} = 
     [x^0 z^n] \Biggl( 
     \frac{1}{(1-z)} \times \biggl( 
     \ConvGF{n}{3}{3}{x} \ConvGF{n}{3}{3}{\frac{z^3}{x}} \\ 
     & \phantom{\defequals [x^0 z^n] 
       \Biggl( \frac{1}{(1-z)} \times \biggl( \quad} + 
     z^{-1} \cdot \ConvGF{n}{3}{2}{x} \ConvGF{n}{3}{2}{\frac{z^3}{x}} \\ 
     & \phantom{\defequals [x^0 z^n] 
       \Biggl( \frac{1}{(1-z)} \times \biggl( \quad} + 
     z^{-2} \cdot \ConvGF{n}{3}{1}{x} \ConvGF{n}{3}{1}{\frac{z^3}{x}} - 
     1 - \frac{1}{z} 
     \biggr) 
     \Biggr). 
\end{align*} 
The next form of the cube-factorial-power 
sequences, $\Sf_4(n)$, defined in 
\eqref{eqn_FactorialSumIdents_examples_afn_sf2n_sf3n-stmts_v1} 
corresponding to the next sums taken over powers of the 
double factorial function are similarly generated by 
\begin{align*} 
\tagtext{Sums of Double Factorial Cubes} 
\Sf_{4,2}(n) & \defequals \sum_{k=0}^{n} (k!!)^{3} \\ 
     & \phantom{:} = 
     [x_1^0 x_2^0 z^n] \Biggl( 
     \frac{1}{(1-z)} \times \biggl( 
     \ConvGF{n}{2}{2}{\frac{z^2}{x_2}} \ConvGF{n}{2}{2}{\frac{x_2}{x_1}} 
     \ConvGF{n}{2}{2}{x_1} \\ 
     & \phantom{\defequals [x_1^0 x_2^0 z^n] \Biggl( 
                \biggl( \quad } + 
     z^{-1} \cdot 
     \ConvGF{n}{2}{1}{\frac{z^2}{x_2}} \ConvGF{n}{2}{1}{\frac{x_2}{x_1}} 
     \ConvGF{n}{2}{1}{x_1} - \frac{1}{z} 
     \biggr) \\ 
     & \phantom{\defequals [x_1^0 x_2^0 z^n] \Biggl( 
                \biggl( \quad } + 
     2\Biggr). 
\end{align*} 

\subsubsection{Another convergent-based generating function identity} 

The second variant of the factorial sums, denoted $\Sf_2(n)$ in 
\eqref{eqn_FactorialSumIdents_examples_afn_sf2n_sf3n-stmts_v1}, is 
enumerated through an alternate approach provided by the 
more interesting summation identities cited in 
\cite[\S 3, p.\ 168]{ADVCOMB}. 
In particular, we have another pair of identities generating these sums 
expanded as 
\begin{align*} 
(n+1)! - 1 & = (n+1)! \times \sum_{k=0}^{n} \frac{k}{(k+1)!} \\ 
           & = [z^n x^0] \left( 
     \left(\frac{1}{x \cdot (1-x)} - \frac{e^{x}}{x}\right) \times 
     \ConvGF{n+2}{-1}{n+1}{\frac{z}{x}} 
     \right) \\ 
           & = [x^0 z^{n+1}] \left( 
     \left(\frac{1}{(1-x)} - e^{x}\right) \times 
     \ConvGF{n+1}{1}{1}{\frac{z}{x}} 
     \right). 
\end{align*} 

\begin{remark}[Sums of Squares of the Binomial Coefficients] 
We can form related convergent-based generating functions 
enumerating the following polynomial sum 
which is also expanded in terms of the \emph{Legendre polynomials}, 
$P_n(x) = [z^n] (z^2 - 2xz + 1)^{-1/2}$, as follows 
\cite[\S 18.3]{NISTHB} \cite[p.\ 543]{GKP}: 
\begin{align*} 
(x-1)^n \cdot P_n\left(\frac{x+1}{x-1}\right) & = 
     \sum_{0 \leq k \leq n} \binom{n}{k}^2 x^k \\ 
     & = 
     (n!)^2 \times [z^n] I_0\left(2 \sqrt{xz}\right) 
     I_0\left(2 \sqrt{z}\right) \\ 
     & = 
     [x_1^0 x_2^0] [z^n] \Bigl( 
     \ConvGF{n}{1}{1}{\frac{z}{x_1}} \ConvGF{n}{1}{1}{\frac{x_1}{x_2}} 
     \times \\ 
     & \phantom{= [x_2^0 x_1^0] [z^n] \Bigl(\ } \times 
     I_0\left(2 \sqrt{x \cdot x_2}\right) I_0\left(2 \sqrt{x_2}\right) 
     \Bigr) \cdot \Iverson{n \geq 1} + \Iverson{n = 0}. 
\end{align*} 
Other related sums over the alternating squares and cubes of the 
binomial coefficients, $(-1)^k \binom{2n}{k}^2$ and $(-1)^k \binom{2n}{k}^3$, 
are given through products of $\alpha$-factorial functions and 
reciprocals of the single factorial function, and so may also be 
enumerated similarly to the sums in this and in the previous section. 
\end{remark} 

\subsubsection{Enumerating more challenging combinatorial sums involving 
               double factorials} 

The convergent-based generating function identities enumerating the 
sequences stated next in 
Example \ref{example_SumsInvolving_DblFactFns-exps_examples_v1} and 
Example \ref{remark_FactSumIdents_SubfactorialSums_ConvIdents} 
below provide additional examples of the termwise formal 
Laplace-Borel-like transform provided by coefficient extractions 
involving these rational convergent functions outlined by 
Remark \ref{remark_Formal_Laplace-Borel_Transforms}. 

\begin{example} 
\label{example_SumsInvolving_DblFactFns-exps_examples_v1} 
Since we know that $(2k-1)!! = [z^k] \ConvGF{n}{2}{1}{z}$ 
for all $0 \leq k < n$, the terms of the next 
modified product sequences are generated through the 
following related forms obtained from the formal series 
expansions of the convergent generating functions: 
\begin{align*} 
\frac{(k+1)}{k!} \cdot (2k-1)!! & = 
     [x^0] [z^k] \left( 
     \ConvGF{k}{-1}{2k-1}{\frac{z}{x}} \cdot (x+1) e^{x} 
     \right) \\ 
     & = 
     [x^0] [z^k] \left( 
     \ConvGF{k}{2}{1}{\frac{z}{x}} \cdot (x+1) e^{x} 
     \right). 
\end{align*} 
The convergent-based expansions of the next \quotetext{round number} identity 
generating the double factorial function given cited in the reference 
are then easily obtained from the 
previous equations in the following forms 
\cite[\S 4.3]{DBLFACTFN-COMBIDENTS-SURVEY}: 
\begin{align*} 
(2n-1)!! & = 
     \sum_{k=1}^{n} \frac{(n-1)!}{(k-1)!} \cdot k \cdot (2k-3)!! \\ 
     & = 
     (n-1)! \times [x_2^{n}] [x_1^0] \left( 
     \frac{x_2}{(1-x_2)} \times \ConvGF{n}{2}{1}{\frac{x_2}{x_1}} \times 
     (x_1+1) e^{x_1} 
     \right) \\ 
     & = 
     [x_1^0 x_2^0 x_3^{n-1}] \left( 
     \ConvGF{n}{1}{1}{\frac{x_3}{x_2}} 
     \ConvGF{n}{2}{1}{\frac{x_2}{x_1}} \times 
     \frac{(x_1+1)}{(1-x_2)} \cdot e^{x_1} 
     \right). 
\end{align*} 
Related challenges are posed in the statements of several other 
finite sum identities involving the double factorial function cited in the 
references \cite{MAA-FUN-WITH-DBLFACT,DBLFACTFN-COMBIDENTS-SURVEY}. 
\end{example} 

\subsubsection{Other examples of convergent-based generating function 
               identities enumerating the subfactorial function} 
The first convergent-based generating function expansions approximating the 
formal ordinary power series over the 
subfactorial sequence given in 
Section \ref{subsubSection_remark_HybridDiagonalHPGFs} 
are expanded as (\seqnum{A000166}) 
\begin{align*} 
\tagtext{Subfactorial OGF Identities} 
!n & = 
     n! \times \sum_{i=0}^{n} \frac{(-1)^{i}}{i!} = 
     [x^0 z^n] \left( 
     \frac{e^{-x}}{(1-x)} \times \ConvGF{n}{1}{1}{\frac{z}{x}} 
     \right) \\ 
   & = 
     \sum_{k=0}^{n} \binom{n}{k} (-1)^{n-k} k! = 
     [z^n x^n] \left( 
     \frac{(x+z)^{n}}{(1+z)} \times \ConvGF{n}{1}{1}{x} 
     \right). 
\end{align*} 
The constructions of the convergent-based formal power series for the 
ordinary generating functions of the 
subfactorial function, $!n$, outlined in the previous section 
are extended in the next example. 

\begin{example} 
\label{remark_FactSumIdents_SubfactorialSums_ConvIdents}
The next pair of 
alternate, factorial-function-like auxiliary recurrence relations 
from (\cite[\S 5.3-\S 5.4]{GKP}, \cite[\S 4.2]{ADVCOMB}) 
exactly define the subfactorial function when $n \geq 2$. 
The following generating function identities 
then correspond to the respective expansions of the 
generating functions in previous equation involving the 
first-order partial derivatives of the convergent functions, 
$\ConvGF{n}{\alpha}{R}{t}$, with respect to $t$ 
\cite[\cf \S 7.2]{GKP} \cite[\cf \S 2.2]{GFOLOGY}: 
\begin{align*} 
\tagtext{Factorial-Related Recurrence Relations} 
!n & = (n-1) \left( !(n-1) + !(n-2) \right) \\ 
     & = (n-1) \times !(n-1) + (n-2) \times !(n-2) + !(n-2) \\ 
     & = 
     [x_1^0 z^n] \left( 
     \frac{(z^2+z^3) \cdot e^{-x_1}}{x_1 \cdot (1-x_1)} \times 
     \Conv_n^{(1)}\left(1, 1; \frac{z}{x_1}\right) + 
     \frac{z^2 \cdot e^{-x_1}}{(1-x_1)} \times 
     \ConvGF{n}{1}{1}{\frac{z}{x_1}} + 1 
     \right) \\ 
!n & = 
     n \times !(n-1) + (-1)^{n} \\ 
     & = 
     (n-1) \times !(n-1) + !(n-1) + (-1)^{n} \\ 
     & = 
     [x_1^0 z^n] \left( 
     \frac{z^2 \cdot e^{-x_1}}{x_1 \cdot (1-x_1)} \times 
     \Conv_n^{(1)}\left(1, 1; \frac{z}{x_1}\right) + 
     \frac{z \cdot e^{-x_1}}{(1-x_1)} \times 
     \ConvGF{n}{1}{1}{\frac{z}{x_1}} + 
     \frac{1}{(1+z)} 
     \right). 
\end{align*} 
The next sums provide another summation-based recursive formula for the 
subfactorial function derived from the known exponential generating function, 
$\widehat{D}_{n\?}(z) = e^{-z} \cdot (1-z)^{-1}$, for this sequence 
(\cite[\S 5.4]{GKP}, \cite[\S 4.2]{ADVCOMB}). 
\begin{align*} 
\tagtext{Basic Subfactorial Recurrence} 
!n & = 
     n! - \sum_{i=1}^{n} \binom{n}{i} !(n-i) \\ 
     & = 
     n! \times \left(1 - \sum_{i=1}^{n} \frac{1}{i!} \cdot 
     \frac{!(n-i)}{(n-i)!}\right) \\ 
     & = 
     n! \times \left(1 - [x_1^0 x_2^0 x_3^n] \left( 
     (e^{x_3}-1) \ConvGF{n}{1}{1}{\frac{x_3}{x_2 x_1}} 
     \frac{e^{x_2-x_1}}{(1-x_1)} \right) 
     \right) \\ 
     & = 
     [x_x^0 x_2^0 x_3^0 z^n] \left( 
     \ConvGF{n}{1}{1}{\frac{z}{x_3}} \left( \frac{1}{(1-x_3)} - 
     \ConvGF{n}{1}{1}{\frac{x_3}{x_2 x_1}} 
     \frac{e^{x_2-x_1} \cdot (e^{x_3}-1)}{(1-x_1)}
     \right) 
     \right) 
\end{align*} 
The rational convergent-based expansions that generate the last 
equation immediately above then correspond to the effect of 
performing a termwise Laplace-Borel transformation approximating the 
complete integral transform, $\mathcal{L}(\widehat{D}_{n\?}(t); z)$, 
defined by 
Remark \ref{remark_Formal_Laplace-Borel_Transforms}, 
which is related to the regularized sums involving the 
incomplete gamma function given in the examples from 
Section \ref{subSection_Intro_Examples_DivergentCFracOGFs} 
of the introduction. 
\end{example} 

\subsection{Examples: Generating sums of powers of natural numbers, 
            binomial coefficient sums, and sequences of binomials} 
\label{subsubSection_Apps_Example_SumsOfPowers_Seqs} 

\subsubsection{Generating variants of sums of powers sequences} 

As a starting point for the next generating function identities 
that provide expansions of the sums of powers 
sequences defined by 
\eqref{eqn_defs_and_ConvFnExps_of_the_GenSUmsOfPowersSeqs} 
in this section below, 
let $p \geq 2$ be fixed, and 
suppose that $m \in \mathbb{Z}^{+}$. 
The convergent-based generating function 
series over the integer powers, $m^{p}$, 
are generated by an application of the binomial theorem to 
form the next sums: 
\begin{align} 
\label{eqn_mPowPm1_IntPows_BinomCoeffGF_exps-stmts_v1} 
m^{p} - 1 & = 
     (p-1)! \cdot \left( 
     p \times \sum_{k=0}^{p-1} \frac{(m-1)^{p-k}}{k! (p-k)!} 
     \right) \\ 
\label{eqn_mPowPm1_IntPows_BinomCoeffGF_exps-stmts_v2} 
m^{p} - 1 & = 
     [z^{p-1}] [x^0] \Biggl( 
     \ConvGF{p}{-1}{p-1}{\frac{z}{x}} \times (m e^{mx} - e^{x}) 
     \Biggr). 
\end{align} 
Next, consider the generating function expansions 
enumerating the finite sums of the $p^{th}$ power sequences in 
\eqref{eqn_mPowPm1_IntPows_BinomCoeffGF_exps-stmts_v1} summed over 
$0 \leq m \leq n$ as follows \cite[\cf \S 7.6]{GKP}: 
\begin{align*} 
\tagonce\label{eqn_Spn_Conv_and_BNumGF_ident_ex-stmt_v2.1} 
\widetilde{B}_u(w, x) & \defequals 
     \sum_{n \geq 0} \left( 
     \sum_{m=0}^{n} (me^{mx}-e^{x}) u^{m} 
     \right) w^{n} \\ 
     & \phantom{:} = 
     \sum_{n \geq 0} \Biggl( 
     \left(\frac{1}{1-u} + \frac{n e^{nx}}{u e^{x} - 1} - 
     \frac{e^{nx}}{(u e^{x} - 1)^2}\right) e^{x} u^{n+1} \\ 
     & \phantom{:= \sum_{n \geq 0} \Biggl( \quad\ } + 
     \left(\frac{u}{(u e^{x} - 1)^2} - \frac{1}{1-u}\right) e^{x} 
     \Biggr) w^{n} \\ 
     & \phantom{:} = 
     -\frac{u^2w^2 e^{3x} - 2 uw e^{2x} + (u^2w^2 -uw + 1) e^{x}}{ 
     (1-w) (1-uw) (uw e^{x} - 1)^2} \\ 
     & \phantom{:} = 
     \frac{e^{x}}{(1-w)(1-uw)} - 
     \frac{1}{(1-w) (e^{x} uw-1)^{2}} + 
     \frac{1}{(1-w) (e^{x} uw - 1)} \\ 
\widetilde{B}_{a,b,u}(w, x) & \defequals 
     \sum_{n \geq 0} \left( 
     \sum_{m=0}^{n} ((am+b) e^{(am+b) x}-e^{x}) u^{m} 
     \right) w^{n} \\ 
     = & 
     \frac{e^{x} - be^{bx} + 
     \left((b-a)e^{(a+b)x} - 2 e^{(a+1)x} + be^{bx}\right) uw + 
     \left((a-b)e^{(a+b)x} + e^{(2a+1)x} \right) (uw)^2}{ 
     (1-w) (1-uw) (uw e^{ax} - 1)^2} \\ 
     = & 
     \frac{be^{bx} + (a-b) e^{(a+b) x} uw}{(1-w) (uw e^{ax}-1)^{2}} - 
     \frac{e^{x} - 2 e^{(a+1) x} uw + e^{(2a+1) x} (uw)^2}{ 
     (1-w) (1-uw) (uw e^{ax}-1)^{2}}. 
\end{align*} 
We then obtain the next cases of the 
convergent-based generating function identities exactly 
enumerating the corresponding first variants of the 
sums of powers sequences obtained from 
\eqref{eqn_mPowPm1_IntPows_BinomCoeffGF_exps-stmts_v1} 
stated in the following forms \cite[\S 6.5, \S 7.6]{GKP}: 
\begin{align*} 
\tagonce\label{eqn_defs_and_ConvFnExps_of_the_GenSUmsOfPowersSeqs} 
S_p(n) & \defequals 
     \sum_{0 \leq m < n} m^{p} \\ 
\tagonce\label{eqn_Spn_Conv_and_BNumGF_ident_ex-stmt_v1} 
     & \phantom{:} = 
     n + 
     [w^{n-1}] [z^{p-1} x^0] \Biggl( 
     \ConvGF{p}{-1}{p-1}{\frac{z}{x}} \widetilde{B}_1(w, x) 
     \Biggr) \\ 
     & \phantom{:} = 
     n + 
     [w^{n-1}] [x^0 z^{p-1}] \Biggl( 
     \ConvGF{p}{1}{1}{\frac{z}{x}} \widetilde{B}_1(w, x) 
     \Biggr). 
\end{align*} 
A somewhat related set of results for 
variations of more general cases of the 
power sums expanded above is expanded similarly 
for $p \geq 1$, fixed scalars $a,b \neq 0$, and 
any non-zero indeterminate $u$ according to the 
next convergent function identities given by 
\begin{align*} 
\tagtext{Generalized Sums of Powers Sequences} 
S_{p}(u, n) & \defequals 
     \sum_{0 \leq m < n} m^{p} u^{m} \\ 
     & \phantom{:} = 
     \frac{u^{n} - 1}{u-1} + 
     [w^{n-1}] [z^{p-1} x^{0}] \Biggl( 
     \ConvGF{p}{-1}{p-1}{\frac{z}{x}} 
     \widetilde{B}_u(w, x)
     \Biggr) \\ 
\tagonce\label{eqn_Spn_Conv_and_BNumGF_ident_ex-stmt_v2} 
     & \phantom{:} = 
     \frac{u^{n} - 1}{u-1} + 
     [w^{n-1}] [x^0 z^{p-1}] \Biggl( 
     \ConvGF{p}{1}{1}{\frac{z}{x}} 
     \widetilde{B}_u(w, x)
     \Biggr) \\ 
S_{p}(a, b; u, n) & \defequals 
     \sum_{0 \leq m < n} (am+b)^{p} u^{m} \\ 
     & \phantom{:} = 
     \frac{u^{n} - 1}{u-1} + 
     [w^{n-1}] [z^{p-1} x^{0}] \Biggl( 
     \ConvGF{p}{-1}{p-1}{\frac{z}{x}} 
     \widetilde{B}_{a,b,u}(w, x)
     \Biggr) \\ 
     & \phantom{:} = 
     \frac{u^{n} - 1}{u-1} + 
     [w^{n-1}] [x^0 z^{p-1}] \Biggl( 
     \ConvGF{p}{1}{1}{\frac{z}{x}} 
     \widetilde{B}_{a,b,u}(w, x)
     \Biggr). 
\end{align*} 

\begin{remark}[Relations to the Bernoulli and Euler Polynomials] 
For fixed $n \geq 0$, integers $a,b$, and some $u \neq 0$, 
exponential generating functions for the 
generalized sums, $S_p(a,b; u, n+1)$, with respect to $p$ 
are given by the following sums \cite[\S 7.6]{GKP}: 
\begin{align*} 
\frac{S_p(a,b; u, n+1)}{p!} & = [z^p] \left( 
     \sum_{0 \leq k \leq n} e^{(ak+b) z} u^{k}
     \right) = 
     [z^p] \left( 
     e^{bz} \times \frac{e^{a(n+1) z} u^{n+1} -1}{u e^{az} - 1} 
     \right). 
\end{align*} 
The bivariate, two-variable exponential generating functions, 
$\widetilde{B}_{u}(w, x)$ and $\widetilde{B}_{a,b,u}(w, x)$, 
involved in enumerating the respective sequences in each of 
\eqref{eqn_Spn_Conv_and_BNumGF_ident_ex-stmt_v1} and 
\eqref{eqn_Spn_Conv_and_BNumGF_ident_ex-stmt_v2} are related to the 
generating functions for the 
\emph{Bernoulli} and \emph{Euler polynomials}, $B_n(x)$ and $E_n(x)$, 
defined in the references when the parameter $u \defmapsto \pm 1$ 
(\cite[\S 24.2]{NISTHB}, \cite[\S 4.2.2, \S 4.2.3]{UC}). 
For $u \defequals \pm 1$, the sums defined by the left-hand-sides of the 
previous two equations in \eqref{eqn_Spn_Conv_and_BNumGF_ident_ex-stmt_v2} 
also correspond to special cases of the following 
known identities involving these polynomial sequences 
\cite[\S 6.5]{GKP} \cite[\S 24.4(iii)]{NISTHB}: 
\begin{align*} 
\tagtext{Sums of Powers Formulas} 
\sum_{m=0}^{n} (am+b)^{p} & = 
     \frac{a^{p}}{p+1} \left( 
     B_{p+1}\left(n+1+\frac{b}{a}\right) - 
     B_{p+1}\left(\frac{b}{a}\right) 
     \right) \\ 
\sum_{m=0}^{n} (-1)^{m} (am+b)^{p} & = 
     \frac{a^{p}}{2} \left( 
     (-1)^{n} \cdot E_{p}\left(n+1+\frac{b}{a}\right) + 
     E_{p}\left(\frac{b}{a}\right) 
     \right). 
\end{align*} 
The results in the previous equations are also compared to the 
forms of other well-known sequence generating functions 
involving the \emph{Bernoulli numbers}, $B_n$, the 
\emph{first-order Eulerian numbers}, $\gkpEI{n}{m}$, and the 
\emph{Stirling numbers of the second kind}, $\gkpSII{n}{k}$, 
in the next few cases of the established identities for these 
sequences expanded in 
Remark \ref{remark_SumsOfPowers_CompsToOtherKnownSeqGFs} 
(\cite[\cf \S 6]{GKP}, \seqnum{A027641}, \seqnum{A027642}, 
\seqnum{A008292}, \seqnum{A008277}). 
\end{remark} 

\begin{remark}[Comparisons to Other Formulas and Special Generating Functions] 
\label{remark_SumsOfPowers_CompsToOtherKnownSeqGFs}
The sequences generated by 
\eqref{eqn_Spn_Conv_and_BNumGF_ident_ex-stmt_v1} are first compared 
to the following known expansions that exactly generate these 
finite sums over $n \geq 0$ and $p \geq 1$ 
(\cite[\S 24.4(iii), \S 24.2]{NISTHB}, \cite[\S 6.5, \S 7.4]{GKP}): 
\begin{align*} 
\tagtext{Relations to the Bernoulli Numbers} 
S_p(n+1) & = 
     \frac{B_{p+1}(n+1) - B_{p+1}(0)}{p+1} \\ 
     & = 
     \sum_{s=0}^{p} \binom{p+1}{s} \frac{B_s \cdot (n+1)^{p+1-s}}{(p+1)} \\ 
\tagtext{Expansions by the Stirling Numbers} 
S_p(n+1) & = 
     \sum_{j=0}^{p} \gkpSII{p}{j} \frac{\FFactII{(n+1)}{j+1}}{(j+1)} \\ 
     & = 
     \sum_{0 \leq j,k \leq p+1} \gkpSII{p}{j} \gkpSI{j+1}{k} 
     \frac{(-1)^{j+1-k} (n+1)^{k}}{j+1} \\ 
\tagtext{Relations to Special Generating Functions} 
S_p(n+1) & = 
     [z^{n}] \left( 
     \sum_{j=0}^{p} \gkpSII{p}{j} \frac{z^{j} \cdot j!}{(1-z)^{j+2}} 
     \right) \\ 
     & = 
     [z^{n}] \left( 
     \sum_{i \geq 0} \gkpEI{p}{i} \frac{z^{i+1}}{(1-z)^{p+2}} 
     \right) \\ 
     & = 
     p! \cdot [w^{n} z^{p}] \left( 
     \frac{w \cdot e^{z}}{(1-w) (1 - w e^{z})} 
     \right). 
\end{align*} 
Two bivariate \quotetext{super} generating function for the 
\emph{first-order Eulerian numbers}, $\gkpEI{n}{m}$, 
employed in formulating variants of the last identity 
are given by the following equations 
where $\gkpEI{n}{m} = \gkpEI{n}{n-1-m}$ for all 
$n \geq 1$ and $0 \leq m < n$ by the row-wise symmetry in the triangle 
\cite[\S 7.4, \S 6.2]{GKP} \cite[\S 26.14(ii)]{NISTHB}: 
\begin{align*} 
\tagtext{First-Order Eulerian Number EGFs} 
\sum_{m,n \geq 0} \gkpEI{n}{m} \frac{w^{m} z^{n}}{n!} & = 
     \frac{1-w}{e^{(w-1) z} - w} \\ 
\sum_{m,n \geq 0} \gkpEI{m+n+1}{m} \frac{w^{m} z^{n}}{(m+n+1)!} & = 
     \frac{e^{w} - e^{z}}{w e^{z} - z e^{w}}. 
\end{align*} 
Similarly, the generalized forms of the sums generated by 
\eqref{eqn_Spn_Conv_and_BNumGF_ident_ex-stmt_v2} are related to the 
more well-known combinatorial sequence identities expanded as follows 
(\cite[\S 26.8]{NISTHB}, \cite[\S 7.4]{GKP}): 
\begin{align*} 
S_p(u, n+1) & = 
     \sum_{j=0}^{p} \gkpSII{p}{j} u^{j} \times 
     \frac{\partial^{(j)}}{{\partial u}^{(j)}} 
     \Biggl( 
     \frac{1}{1-u} - \frac{u^{n+1}}{1-u} 
     \Biggr) \\ 
   & = 
     [w^{n}] \left( 
     \sum_{j=0}^{p} \gkpSII{p}{j} \frac{(uw)^{j} \cdot j!}{(1-w) (1-uw)^{j+1}} 
     \right) \\ 
   & = 
     p! \cdot [w^{n} z^{p}] \left( 
     \frac{uw \cdot e^{z}}{(1-w) (1 - uw e^{z})} 
     \right). 
\end{align*} 
As in the examples of termwise applications of the formal 
Laplace-Borel transformations noted above, 
the role of the parameter $p$ corresponding to the forms of the 
special sequence triangles in the identities given above is 
phrased through the implicit dependence of the convergent functions on the 
fixed $p \geq 1$ in each of 
\eqref{eqn_mPowPm1_IntPows_BinomCoeffGF_exps-stmts_v2}, 
\eqref{eqn_Spn_Conv_and_BNumGF_ident_ex-stmt_v1}, and 
\eqref{eqn_Spn_Conv_and_BNumGF_ident_ex-stmt_v2}. 
\end{remark} 

\subsubsection{Semi-rational generating function constructions 
               enumerating sequences of binomials} 

A second motivating application highlighting the procedure outlined 
in the examples above 
expands the binomial power sequences, $2^{p} - 1$, for $p \geq 1$ through an 
extension of the first result given in 
\eqref{eqn_mPowPm1_IntPows_BinomCoeffGF_exps-stmts_v1} when $m \defequals 2$. 
The finite sums for the integer powers provided by the binomial theorem in 
these cases correspond to removing, or selectively peeling off, the 
$r$ uppermost-indexed terms from the 
first sum for subsequent choices of the $p \geq r \geq 1$ in the 
following forms \cite[\cf \S 2.2, \S 2.4]{PRIMEREC}: 
\begin{align*} 
m^{p} - 1 & = \sum_{i=0}^{r} \binom{p}{p+1-i} (m-1)^{i} \\ 
     & \phantom{=\ } + 
     (p-r-1)! \cdot \left( 
     p(p-1) \cdots (p-r) \times \sum_{k=0}^{p-r-1} 
     \frac{(m-1)^{k+1}}{(k+1)! (p-1-k)!} 
     \right). 
\end{align*} 
The generating function identities phrased in terms of 
\eqref{eqn_mPowPm1_IntPows_BinomCoeffGF_exps-stmts_v1} 
in the previous examples 
are then modified slightly according to this equation for the 
next few special cases of $r \geq 1$. 
For example, these sums are employed to obtain the next 
forms of the convergent-based generating function 
expansions generalizing the result in 
\eqref{eqn_mPowPm1_IntPows_BinomCoeffGF_exps-stmts_v2} above 
(\seqnum{A000225}). 
\begin{align*} 
2^{p} - 1 = 
     [z^{p-2}] [x^0] \Biggl( & 
     \frac{1}{(1-z)} + 
     \left(4 e^{2x} - 2 e^{x}\right) \times 
     \ConvGF{p}{-1}{p-2}{\frac{z}{x}} 
     \Biggr),\ p \geq 2 \\ 
     = 
     [z^{p-3}] [x^0] \Biggl( & 
     \frac{4-3z}{(1-z)^2} + 
     \left(8 e^{2x} - e^{x} \cdot (x+5) \right) \times 
     \ConvGF{p}{-1}{p-3}{\frac{z}{x}} 
     \Biggr),\ p \geq 3 \\ 
     = 
     [z^{p-4}] [x^0] \Biggl( & 
     \frac{11 - 17 z + 7 z^2}{(1-z)^3} \\ 
     & + 
     \left(16 e^{2x} - \frac{e^{x}}{2} \cdot (x^2 + 10x + 24) \right) \times 
     \ConvGF{p}{-1}{p-4}{\frac{z}{x}} 
     \Biggr),\ p \geq 4 \\ 
     = [z^{p-5}] [x^0] \Biggl( & 
     \frac{26 - 62z + 52 z^2 - 15 z^3}{(1-z)^4} \\ 
     & + 
     \frac{x^5 e^{x}}{6} (192 e^x - (x^3 + 18 x^2 + 96 x + 162)) \times 
     \ConvGF{p}{-1}{p-5}{\frac{z}{x}} 
     \Biggr), \\ 
     & \phantom{=\ } p \geq 5. 
\end{align*} 
The special cases of these generating functions for the 
$p^{th}$ powers defined above are also expanded in the next more general 
forms of these convergent function identities for $p > m \geq 1$. 
\begin{align*} 
2^p - 1 
       & = [z^{p-m-1} x^0] \Biggl( 
       \frac{\widetilde{\ell}_{m,2}(z)}{(1-z)^{m}} + 
       \left(2^{m+1} \cdot e^{2x} - 
       \frac{e^{x}}{(m-1)!} \cdot \widetilde{p}_{m,2}(x) \right) \times \\ 
       & \phantom{= [z^{p-m-1} x^0] \Biggl(\ } \times 
       \ConvGF{p}{-1}{p-m-1}{\frac{z}{x}} 
       \Biggr) \\ 
       & = [x^0 z^{p-m-1}] \left( 
       \frac{\widetilde{\ell}_{m,2}(z)}{(1-z)^{m}} + 
       \left(2^{m+1} \cdot e^{2x} - 
       \frac{e^{x}}{(m-1)!} \cdot \widetilde{p}_{m,2}(x) \right) \times 
       \ConvGF{p}{1}{1}{\frac{z}{x}} 
       \right) 
\end{align*} 
The listings provided in 
\tableref{table_ConvGF_Examples_for_PthPowerSeqs} 
cite the particular special cases of the polynomials, 
$\ell_{m,2}(z)$ and $p_{m,2}(x)$, that provide the generalizations of the 
first cases expanded in the previous equations. 
The constructions of these new identities, 
including the variations for the sequences 
formed by the binomial coefficient sums for the powers, 
$2^{p} - 2$, are motivated in the context of divisibility modulo $p$ by the 
reference (\cite[\S 8]{HARDYWRIGHTNUMT}, \seqnum{A000918}). 

Further cases of the more general $p^{th}$ power sequences 
of the form $(s+1)^{p} - 1$ for any fixed $s > 0$ are enumerated 
similarly through the next formulas. 
\begin{align*} 
(s+1)^p - 1 = 
       [z^{p-m-1} x^0] \Biggl( 
       \frac{s^2 \ell_{m,s+1}(z)}{(1-sz)^{m}} - & 
       \left(e^{x} - (s+1)^{m+1} e^{(s+1) x} + 
       \frac{s^2 e^{sx}}{(m-1)!} p_{m,s+1}(sx) \right) \times \\ 
       & \phantom{\Biggl( } \times 
       \ConvGF{p}{-1}{p-m-1}{\frac{z}{x}} 
       \Biggr),\ 
       p > m \geq 1 \\ 
       = 
       [x^0 z^{p-m-1}] \Biggl( 
       \frac{s^2 \ell_{m,s+1}(z)}{(1-sz)^{m}} - & 
       \left(e^{x} - (s+1)^{m+1} e^{(s+1) x} + 
       \frac{s^2 e^{sx}}{(m-1)!} p_{m,s+1}(sx) \right) \times \\ 
       & \phantom{\Biggl( } \times 
       \ConvGF{p}{1}{1}{\frac{z}{x}} 
       \Biggr),\ 
       p > m \geq 1 
\end{align*} 
The second set of listings provided in 
\tableref{table_ConvGF_Examples_for_PthPowerSeqs} 
expand several additional special cases corresponding to the 
polynomial sequences, $\ell_{m,s+1}(z)$ and $p_{m,s+1}(x)$, 
required to generate the more general cases of these 
particular $p^{th}$ power sequences when $p > m \geq 1$ 
(\seqnum{A000225}, \seqnum{A024023}, \seqnum{A024036}, \seqnum{A024049}). 
Related expansions of the \emph{sequences of binomials} of the forms 
$a^{n} \pm 1$ and $a^{n} \pm b^{n}$ are considered in the references 
\cite[\cf \S 2.2, \S 2.4]{PRIMEREC}. 

\section{Conclusions} 
\label{Section_ConcludingRemarks} 

We have defined several new forms of 
ordinary power series approximations to the 
typically divergent ordinary generating functions of 
generalized multiple, or $\alpha$-factorial, function sequences. 
The generalized forms of these convergent functions 
provide partial truncated approximations to the sequence generating 
functions which enumerate the factorial products generated by 
these divergent formal power series. 
The exponential generating functions for the 
special case product sequences, $p_n(\alpha, s-1)$, 
are studied in the reference \cite[\S 5]{MULTIFACTJIS}. 
The exponential generating functions that enumerate the 
cases corresponding to the more general factorial-like sequences, 
$p_n(\alpha, \beta n + \gamma)$, are less obvious in form. 
We have also suggested a number of new, alternate 
approaches to generating the factorial function sequences that arise 
in applications, including 
classical identities involving the single and double factorial functions, and 
in the forms of several other noteworthy special cases. 

The key ingredient to the short proof given in 
Section \ref{Section_Proofs_of_the_GenCFracReps} 
employs known characterizations of the 
Pochhammer symbols, $\Pochhammer{x}{n}$, by 
generalized Stirling number triangles as 
polynomial expansions in the indeterminate, $x$, 
each with predictably small finite-integral-degree at any fixed $n$. 
The more combinatorial proof in the spirit of Flajolet's articles 
suggested by the discussions in 
Section \ref{subSection_GenCFrac_Reps_for_GenFactFns} 
may lead to further interesting interpretations of the 
$\alpha$-factorial functions, $(s-1)!_{(\alpha)}$, 
which motivate the investigations of the coefficient-wise 
symbolic polynomial expansions of the functions first considered in the 
article \cite{MULTIFACTJIS}. 
A separate proof of the expansions of these new continued fractions 
formulated in terms of the 
generalized $\alpha$-factorial function coefficients defined by 
\eqref{eqn_Fa_rdef}, and by their strikingly Stirling-number-like 
combinatorial properties motivated in the introduction, 
is notably missing from this article. 

The rationality of these convergent functions for all $h$ 
suggests new insight to generating numeric sequences of interest, 
including several specific new congruence properties, derivations of 
finite difference equations that hold for these exact sequences 
modulo any integers $p$, and perhaps more interestingly, 
exact expansions of the classical single and double factorial functions 
by the special zeros of the generalized Laguerre polynomials and 
confluent hypergeometric functions. 
The techniques behind the specific identities given here 
are easily generalized and extended to other specific applications. 
The particular examples cited within this article 
are intended as suggestions at new 
starting points to tackling the expansions that arise in 
many other practical situations, both implicitly and explicitly 
involving the generalized variants of the factorial-function-like 
product sequences, $p_n(\alpha, R)$. 

\section{Acknowledgments} 
\label{Section_Acks} 

The original work on the article started as an extension of the 
forms of the generalized factorial function variants considered in my article 
published in the \emph{Journal of Integer Sequences} in 2010. 
The research on the continued fraction representations for the 
generalized factorial functions considered in this article 
began as the topic for my final project in the 
\emph{Introduction to Mathematical Research} course at the 
University of Illinois at Urbana-Champaign 
around the time of the first publication. 
Thanks especially to Professor Bruce Reznick at the 
University of Illinois at Urbana-Champaign, and 
also to Professor Jimmy McLaughlin, 
for their helpful input on revising previous drafts of the article. 

\section{Appendix: Tables}

\label{Section_appendix_StartOfTableData} 
\label{page_StartOfTableData} 

\begin{table}[H] 

\smaller 
\centering 

\begin{subtable}{\subtablewidth} 
\centering 

\begin{tabular}{|c|l|} \hline 
\hline\tabletopstrut  
$h$ & $\ConvFP{h}{\alpha}{R}{z}$ \tablebottomstrut \\ \hline 
1 & $1$ \\ 
2 & $1 - (2\alpha+R)z$ \\ 
3 & $1 - (6\alpha+2R)z + (6\alpha^2+4\alpha R+R^2)z^2$ \\ 
4 & $1 - (12\alpha+3R)z + (36\alpha^2+19\alpha R+3R^2)z^2 - 
           (24\alpha^3 + 18\alpha^2R + 7\alpha R^2 + R^3) z^3$ \\ 
5 & $1 - (20 \alpha + 4 R) z + (120 \alpha^2 + 51 \alpha R + 6 R^2) z^2 - 
     (240 \alpha^3 + 158 \alpha^2 R + 42 \alpha R^2 + 4 R^3) z^3$ \\ 
  & $\phantom{1} + 
     (120 \alpha^4 + 96 \alpha^3 R + 46 \alpha^2 R^2 + 11 \alpha R^3 + R^4) z^4$ \\ 
\hline 
\end{tabular} 
\caption{The convergent numerator functions, $\ConvFP{h}{\alpha}{R}{z}$}

\begin{tabular}{|c|l|} \hline 
\hline\tabletopstrut 
$h$ & $\ConvFQ{h}{\alpha}{R}{z}$ \\ \hline 
0 & $1$ \\ 
1 & $1-Rz$ \\ 
2 & $1 - 2 (\alpha + R) z + R (\alpha + R) z^2$ \\ 
3 & $1 - 3 (2 \alpha + R) z + 3 (\alpha + R) (2 \alpha + R) z^2 - 
      R (\alpha + R) (2 \alpha + R) z^3$ \\ 
4 & $1 - 4 (3 \alpha + R) z + 6 (2 \alpha + R) (3 \alpha + R) z^2 - 
       4 (\alpha + R) (2 \alpha + R) (3 \alpha + R) z^3$ \\ 
  & $\phantom{1 } + 
     R (\alpha + R) (2 \alpha + R) (3 \alpha + R) z^4$ \\ 
5 & $1 - 5 (4 \alpha + R) z + 10 (3 \alpha + R) (4 \alpha + R) z^2 - 
     10 (2 \alpha + R) (3 \alpha + R) (4 \alpha + R) z^3$ \\ 
  & $\phantom{1 } + 
     5 (\alpha + R) (2 \alpha + R) (3 \alpha + R) (4 \alpha + R) z^4 - 
     R (\alpha + R) (2 \alpha + R) (3 \alpha + R) (4 \alpha + R) z^5$ \\ 
\hline 
\end{tabular} 

\caption{The convergent denominator functions, $\ConvFQ{h}{\alpha}{R}{z}$}

\begin{tabular}{|c|l|l|} \hline 
\hline\tabletopstrut  
$h$ & $\ConvFP{h}{1}{1}{z}$ & $\ConvFQ{h}{1}{1}{z}$ 
\tablebottomstrut \\ \hline 
1 & $1$ & $1-z$ \\ 
2 & $1-3 z$ & $1-4 z+2 z^2$ \\ 
3 & $1-8 z+11 z^2$ & $1-9 z+18 z^2-6 z^3$ \\ 
4 & $1-15 z+58 z^2-50 z^3$ & $1-16 z+72 z^2-96 z^3+24 z^4$ \\ 
5 & $1-24 z+177 z^2-444 z^3+274 z^4$ & 
    $1-25 z+200 z^2-600 z^3+600 z^4-120 z^5$ \\
6 & $1-35 z+416 z^2-2016 z^3$ & $1-36 z+450 z^2-2400 z^3+5400 z^4$ \\ 
  & $\quad + 3708 z^4-1764 z^5$ & $\quad - 4320 z^5+720 z^6$ \\ 
\hline\hline 
\end{tabular} 
\caption{The $h^{th}$convergent generating functions, 
         $\ConvGF{h}{1}{1}{z}$, 
         generating the single factorial function, $n! = (1)_n$}

\begin{tabular}{|c|l|l|} \hline 
\hline\tabletopstrut  
$h$ & $\ConvFP{h}{2}{1}{z}$ & $\ConvFQ{h}{2}{1}{z}$ 
\tablebottomstrut \\ \hline 
1 & $1$ & $1-z$ \\ 
2 & $1-5 z$ & $1-6 z+3 z^2$ \\ 
3 & $1-14 z+33 z^2$ & $1-15 z+45 z^2-15 z^3$ \\ 
4 & $1-27 z+185 z^2-279 z^3$ & $1-28 z+210 z^2-420 z^3+105 z^4$ \\ 
5 & $1-44 z+588 z^2-2640 z^3+2895 z^4$ & 
    $1-45 z+630 z^2-3150 z^3+4725 z^4-945 z^5$ \\ 
6 & $1-65 z+1422 z^2-12558 z^3$ & 
    $1-66 z+1485 z^2-13860 z^3+51975 z^4$ \\
  & $\quad + 41685 z^4-35685 z^5$ & $\quad - 62370 z^5+10395 z^6$ \\ 
\hline\hline 
\end{tabular} 
\caption{The $h^{th}$ convergent generating functions, 
         $\ConvGF{h}{2}{1}{z}$, 
         generating the double factorial function, 
         $(2n-1)!! = 2^{n} \times \Pochhammer{\frac{1}{2}}{n}$} 

\end{subtable} 

\caption{The generalized convergent numerator and 
         denominator function sequences} 
\label{table_SpCase_Listings_Of_Phz_ConvFn} 
\label{table_SpCase_Listings_Of_Qhz_ConvFn} 
\label{table_SpCase_Listings_Of_PhzQhz_ConvFn} 

\end{table} 

\begin{table}[H] 
\smaller\centering 

\begin{subtable}{\subtablewidth} 
\centering 

\begin{tabular}{|c|l|} \hline 
\hline\tabletopstrut 
$h$ & $z^{h-1} \cdot \ConvFP{h}{\alpha}{R}{z^{-1}}$ \\ \hline 
1 & $1$ \\ 
2 & $-(2\alpha+R) + z$ \\ 
3 & $6 \alpha ^2+\alpha  (4 R-6 z)+(R-z)^2$ \\ 
4 & $-24 \alpha ^3-18 \alpha ^2 (R-2 z)-\alpha  (7 R-12 z) (R-z)-(R-z)^3$ \\ 
5 & $120 \alpha ^4+2 \alpha ^2 \left(23 R^2-79 R z+60 z^2\right) + 
     48 \alpha ^3 (2 R-5 z)+\alpha  (11 R-20 z) (R-z)^2+(R-z)^4$ \\ 
6 & $-720 \alpha ^5-2 \alpha ^3 \left(163 R^2-678 R z+600 z^2\right) - 
     \alpha ^2 \left(101 R^2-368 R z+300 z^2\right) (R-z)$ \\ 
   & $\phantom{-720 \alpha ^5} - 
      600 \alpha ^4 (R-3 z)-2 \alpha  (8R-15 z) (R-z)^3-(R-z)^5$ \\ 
7  & $5040 \alpha ^6+36 \alpha ^4 \left(71 R^2-347 R z+350 z^2\right) + 
      \alpha ^2 \left(197 R^2-740 R z+630 z^2\right) (R-z)^2$ \\ 
   & $\phantom{5040 \alpha ^6} + 
      \alpha ^3 \left(932 R^3-5102 R^2z+8322 R z^2-4200 z^3\right) + 
      2160 \alpha ^5 (2 R-7 z)$ \\ 
   & $\phantom{5040 \alpha ^6} + 
      2 \alpha  (11 R-21 z) (R-z)^4 + (R-z)^6$ \\ 
8  & $-40320 \alpha ^7-36 \alpha ^5 \left(617 R^2-3466 R z+3920 z^2\right) - 
      \alpha ^2 \left(351 R^2-1342 R z+1176 z^2\right) (R-z)^3$ \\ 
   & $\phantom{-40320 \alpha ^7} + 
      \alpha ^4 \left(-9080 R^3+57286 R^2 z-105144 R z^2+58800 z^3\right)$ \\ 
   & $\phantom{-40320 \alpha ^7} - 
       \alpha ^3 \left(2311 R^3-13040 R^2 z+22210 R z^2-11760 z^3\right) 
       (R-z)-35280 \alpha ^6 (R-4 z)$ \\ 
   & $\phantom{-40320 \alpha ^7} - 
      \alpha (29 R-56 z) (R-z)^5-(R-z)^7$ \\ \hline 
\hline 
\end{tabular} 
\subcaption{The reflected numerator polynomials, 
          $\widetilde{\FP}_h(\alpha, R; z) \defequals z^{h-1} \cdot 
          \ConvFP{h}{\alpha}{R}{z^{-1}}$}

\begin{tabular}{|c|l|} \hline 
\hline\tabletopstrut 
$h$ & $z^{h-1} \cdot \ConvFP{h}{\alpha}{z-w}{z^{-1}}$ \\ \hline 
2 & $w-2 \alpha$ \\ 
3 & $6 \alpha ^2+w^2-4 \alpha  w-2 \alpha  z$ \\ 
4 & $-24 \alpha ^3+w^3-7 \alpha  w^2+w \left(18 \alpha ^2- 
    5 \alpha  z\right)+18 \alpha ^2 z$ \\ 
5 & $120 \alpha ^4+w^4-11 \alpha  w^3+ 
     w^2 \left(46 \alpha ^2-9 \alpha  z\right)+ 
     w \left(66 \alpha ^2 z-96 \alpha ^3\right)+8 \alpha ^2 z^2- 
     144 \alpha ^3 z$ \\ 
6 & $-720 \alpha ^5+w^5-16 \alpha  w^4+ 
     w^3 \left(101 \alpha ^2-14 \alpha  z\right)+ 
     w^2 \left(166 \alpha ^2 z-326 \alpha ^3\right)$ \\ 
  & $\qquad + 
     w \left(600 \alpha ^4+33 \alpha ^2 z^2-704 \alpha ^3 z\right)- 
     170 \alpha ^3 z^2+1200 \alpha ^4 z$ \\ 
7 & $5040 \alpha ^6+w^6-22 \alpha  w^5+ 
     w^4 \left(197 \alpha ^2-20 \alpha  z\right)+ 
     w^3 \left(346 \alpha ^2 z-932 \alpha ^3\right)$ \\ 
  & $\qquad + 
     w^2 \left(2556 \alpha ^4+87 \alpha ^2 z^2-2306 \alpha ^3 z\right)$ \\ 
  & $\qquad + 
     w \left(-4320 \alpha ^5-914 \alpha ^3 z^2+7380 \alpha ^4 z\right)- 
     48 \alpha ^3 z^3+2664 \alpha ^4 z^2-10800 \alpha ^5 z$ \\ 
8 & $-40320 \alpha ^7+w^7-29 \alpha  w^6+ 
     w^5 \left(351 \alpha ^2-27 \alpha  z\right)+ 
     w^4 \left(640 \alpha ^2 z-2311 \alpha ^3\right)$ \\ 
  & $\qquad + 
     w^3 \left(9080 \alpha ^4+185 \alpha ^2 z^2-6107 \alpha ^3 z\right)+ 
     w^2 \left(-22212 \alpha ^5-3063 \alpha ^3 z^2+30046 \alpha ^4 z\right)$ \\ 
  & $\qquad + 
     w \left(35280 \alpha ^6-279 \alpha ^3 z^3+17812 \alpha ^4 z^2- 
     80352 \alpha ^5 z\right)$ \\ 
  & $\qquad + 
     1862 \alpha ^4 z^3-38556 \alpha ^5 z^2+105840 \alpha ^6 z$ \\ \hline 
\hline 
$h$ & $z^{h-1} \cdot \ConvFP{h}{-\alpha}{z-w}{z^{-1}}$ \\ \hline 
3 & $6 \alpha ^2+w^2+\alpha  (4 w+2 z)$ \\ 
4 & $24 \alpha ^3+w^3+\alpha  \left(7 w^2+5 w z\right)+ 
     \alpha ^2 (18 w+18 z)$ \\ 
5 & $120 \alpha ^4+w^4+\alpha ^2 \left(46 w^2+66 w z+8 z^2\right)+ 
     \alpha  \left(11 w^3+9 w^2 z\right)+\alpha ^3 (96 w+144 z)$ \\ 
6 & $720 \alpha ^5+w^5+\alpha ^3 \left(326 w^2+704 w z+170 z^2\right)+ 
     \alpha \left(16 w^4+14 w^3 z\right)$ \\ 
  & $\qquad + 
     \alpha ^2 \left(101 w^3+166 w^2 z+33 w z^2\right)+ 
     \alpha ^4 (600 w+1200 z)$ \\ 
7 & $5040 \alpha ^6+w^6+ 
     \alpha ^4 \left(2556 w^2+7380 w z+2664 z^2\right)+ 
     \alpha  \left(22 w^5+20 w^4 z\right)$ \\ 
  & $\qquad + 
     \alpha ^3 \left(932 w^3+2306 w^2 z+914 w z^2+48 z^3\right)+ 
     \alpha ^2 \left(197 w^4+346 w^3 z+87 w^2 z^2\right)$ \\ 
  & $\qquad + 
     \alpha ^5 (4320 w+10800 z)$ \\ 
8 & $40320 \alpha ^7+w^7+ 
     \alpha ^5 \left(22212 w^2+80352 w z+38556 z^2\right)+ 
     \alpha  \left(29 w^6+27 w^5 z\right)$ \\ 
  & $\qquad + 
     \alpha ^4 \left(9080 w^3+30046 w^2 z+17812 w z^2+1862 z^3\right)+ 
     \alpha ^2 \left(351 w^5+640 w^4 z+185 w^3 z^2\right)$ \\ 
  & $\qquad + 
     \alpha ^3 \left(2311 w^4+6107 w^3 z+3063 w^2 z^2+279 w z^3\right)+ 
     \alpha ^6 (35280 w+105840 z)$ \\ \hline\hline
\end{tabular} 
\subcaption{Modified forms of the reflected numerator polynomials, 
          $\widetilde{\FP}_h(\pm \alpha, z-w; z)$} 

\end{subtable} 

\caption{The reflected convergent numerator function sequences} 
\label{table_RelfectedConvNumPolySeqs_sp_cases} 
\end{table} 

\addtocounter{table}{1}
\setcounter{subtable}{0} 

\begin{sidewaystable}
\centering 
\smaller 

\resizebox{0.95\textwidth}{!}{
\begin{tabular}{|l|l|lcc|lcc|lcc|lcc|} \hline 
\hline\tabletopstrut 
$n$ & $\MultiFactorial{n}{2}$ & 
$\widetilde{R}_2^{(2)}(n)$ & $\pod{2}$ & $\pod{4}$ &
$\widetilde{R}_3^{(2)}(n)$ & $\pod{3}$ & $\pod{6}$ &
$\widetilde{R}_4^{(2)}(n)$ & $\pod{4}$ & $\pod{8}$ &
$\widetilde{R}_5^{(2)}(n)$ & $\pod{5}$ & $\pod{10}$ \\ \hline 
 0 & 1 & 1 & 1 & 1 & 1 & 1 & 1 & 1 & 1 & 1 & 1 & 1 & 1 \\
 1 & 1 & 1 & 1 & 1 & 1 & 1 & 1 & 1 & 1 & 1 & 1 & 1 & 1 \\
 2 & 2 & 2 & 0 & 2 & 2 & 2 & 2 & 2 & 2 & 2 & 2 & 2 & 2 \\
 3 & 3 & 3 & 1 & 3 & 3 & 0 & 3 & 3 & 3 & 3 & 3 & 3 & 3 \\
 4 & 8 & 8 & 0 & 0 & 8 & 2 & 2 & 8 & 0 & 0 & 8 & 3 & 8 \\
 5 & 15 & 15 & 1 & 3 & 15 & 0 & 3 & 15 & 3 & 7 & 15 & 0 & 5 \\
 6 & 48 & 48 & 0 & 0 & 48 & 0 & 0 & 48 & 0 & 0 & 48 & 3 & 8 \\
 7 & 105 & -175 & 1 & 1 & 105 & 0 & 3 & 105 & 1 & 1 & 105 & 0 & 5 \\
 8 & 384 & 0 & 0 & 0 & 384 & 0 & 0 & 384 & 0 & 0 & 384 & 4 & 4 \\
 9 & 945 & -13671 & 1 & 1 & 945 & 0 & 3 & 945 & 1 & 1 & 945 & 0 & 5 \\
 10 & 3840 & -17920 & 0 & 0 & 3840 & 0 & 0 & 3840 & 0 & 0 & 3840 & 0 & 0 \\
 11 & 10395 & -633501 & 1 & 3 & 43659 & 0 & 3 & 10395 & 3 & 3 & 10395 & 0 & 5 \\
 12 & 46080 & -960000 & 0 & 0 & 92160 & 0 & 0 & 46080 & 0 & 0 & 46080 & 0 & 0 \\
 13 & 135135 & -28498041 & 1 & 3 & 3532815 & 0 & 3 & 135135 & 3 & 7 & 135135 & 0 & 5 \\
 14 & 645120 & -45480960 & 0 & 0 & 5644800 & 0 & 0 & 645120 & 0 & 0 & 645120 & 0 & 0 \\
 15 & 2027025 & -1343937855 & 1 & 1 & 257161905 & 0 & 3 & -5386095 & 1 & 1 & 2027025 & 0 & 5 \\
 16 & 10321920 & -2202927104 & 0 & 0 & 401522688 & 0 & 0 & 0 & 0 & 0 & 10321920 & 0 & 0 \\
 17 & 34459425 & -67747539375 & 1 & 1 & 17642360385 & 0 & 3 & -1211768415 & 1 & 1 & 34459425 & 0 & 5 \\
 18 & 185794560 & -112925343744 & 0 & 0 & 27994595328 & 0 & 0 & -1634992128 & 0 & 0 & 185794560 & 0 & 0 \\
 19 & 654729075 & -3664567145437 & 1 & 3 & 1200706189875 & 0 & 3 & -141536175885 & 3 & 3 & 3315215475 & 0 & 5 \\
 20 & 3715891200 & -6182061834240 & 0 & 0 & 1941606236160 & 0 & 0 & -211558072320 & 0 & 0 & 7431782400 & 0 & 0 \\
 21 & 13749310575 & -212363430514977 & 1 & 3 & 83236453970607 & 0 & 3 & -14054409745425 & 3 & 7 & 679112772975 & 0 & 5 \\
\hline\hline
\end{tabular} 
} 

\captionof{subtable}{Congruences for the double factorial function, 
                     $n!! = \MultiFactorial{n}{2}$, 
                     modulo $h$ (and $2h$) for $h \defequals 2,3,4,5$. 
        Supplementary listings containing computational data for the 
        congruences, $n!! \equiv R_h^{(2)}(n) \pmod{2^i h}$, 
        for $2 \leq i \leq h \leq 5$ are tabulated in the 
        summary notebook reference. 
        } 

\end{sidewaystable} 

\begin{sidewaystable}
\centering 
\smaller 

\resizebox{0.95\textwidth}{!}{
\begin{tabular}{|l|l|lcc|lcc|lcc|lcc|} \hline 
\hline\tabletopstrut 
$n$ & $\MultiFactorial{n}{3}$ & 
$\widetilde{R}_2^{(3)}(n)$ & $\pod{2}$ & $\pod{6}$ &
$\widetilde{R}_3^{(3)}(n)$ & $\pod{3}$ & $\pod{9}$ &
$\widetilde{R}_4^{(3)}(n)$ & $\pod{4}$ & $\pod{12}$ &
$\widetilde{R}_5^{(3)}(n)$ & $\pod{5}$ & $\pod{15}$ \\ \hline 
 0 & 1 & 1 & 1 & 1 & 1 & 1 & 1 & 1 & 1 & 1 & 1 & 1 & 1 \\
 1 & 1 & 1 & 1 & 1 & 1 & 1 & 1 & 1 & 1 & 1 & 1 & 1 & 1 \\
 2 & 2 & 2 & 0 & 2 & 2 & 2 & 2 & 2 & 2 & 2 & 2 & 2 & 2 \\
 3 & 3 & 3 & 1 & 3 & 3 & 0 & 3 & 3 & 3 & 3 & 3 & 3 & 3 \\
 4 & 4 & 4 & 0 & 4 & 4 & 1 & 4 & 4 & 0 & 4 & 4 & 4 & 4 \\
 5 & 10 & 10 & 0 & 4 & 10 & 1 & 1 & 10 & 2 & 10 & 10 & 0 & 10 \\
 6 & 18 & 18 & 0 & 0 & 18 & 0 & 0 & 18 & 2 & 6 & 18 & 3 & 3 \\
 7 & 28 & 28 & 0 & 4 & 28 & 1 & 1 & 28 & 0 & 4 & 28 & 3 & 13 \\
 8 & 80 & 80 & 0 & 2 & 80 & 2 & 8 & 80 & 0 & 8 & 80 & 0 & 5 \\
 9 & 162 & 162 & 0 & 0 & 162 & 0 & 0 & 162 & 2 & 6 & 162 & 2 & 12 \\
 10 & 280 & -980 & 0 & 4 & 280 & 1 & 1 & 280 & 0 & 4 & 280 & 0 & 10 \\
 11 & 880 & -704 & 0 & 4 & 880 & 1 & 7 & 880 & 0 & 4 & 880 & 0 & 10 \\
 12 & 1944 & 0 & 0 & 0 & 1944 & 0 & 0 & 1944 & 0 & 0 & 1944 & 4 & 9 \\
 13 & 3640 & -92300 & 0 & 4 & 3640 & 1 & 4 & 3640 & 0 & 4 & 3640 & 0 & 10 \\
 14 & 12320 & -115192 & 0 & 2 & 12320 & 2 & 8 & 12320 & 0 & 8 & 12320 & 0 & 5 \\
 15 & 29160 & -136080 & 0 & 0 & 29160 & 0 & 0 & 29160 & 0 & 0 & 29160 & 0 & 0 \\
 16 & 58240 & -6186752 & 0 & 4 & 395200 & 1 & 1 & 58240 & 0 & 4 & 58240 & 0 & 10 \\
 17 & 209440 & -8349992 & 0 & 4 & 633556 & 1 & 1 & 209440 & 0 & 4 & 209440 & 0 & 10 \\
 18 & 524880 & -10935000 & 0 & 0 & 1049760 & 0 & 0 & 524880 & 0 & 0 & 524880 & 0 & 0 \\
 19 & 1106560 & -411766784 & 0 & 4 & 51684256 & 1 & 1 & 1106560 & 0 & 4 & 1106560 & 0 & 10 \\
 20 & 4188800 & -572266240 & 0 & 2 & 70505120 & 2 & 2 & 4188800 & 0 & 8 & 4188800 & 0 & 5 \\
 21 & 11022480 & -777084840 & 0 & 0 & 96446700 & 0 & 0 & 11022480 & 0 & 0 & 11022480 & 0 & 0 \\
 22 & 24344320 & -28922921456 & 0 & 4 & 5645314048 & 1 & 4 & -144674816 & 0 & 4 & 24344320 & 0 & 10 \\
 23 & 96342400 & -40807520000 & 0 & 4 & 7668245080 & 1 & 1 & -116486720 & 0 & 4 & 96342400 & 0 & 10 \\
 24 & 264539520 & -56458612224 & 0 & 0 & 10290587328 & 0 & 0 & 0 & 0 & 0 & 264539520 & 0 & 0 \\
 25 & 608608000 & -2177450514800 & 0 & 4 & 577086766300 & 1 & 1 & -41321139200 & 0 & 4 & 608608000 & 0 & 10 \\
 26 & 2504902400 & -3101148709984 & 0 & 2 & 793943072000 & 2 & 8 & -52040160640 & 0 & 8 & 2504902400 & 0 & 5 \\
 27 & 7142567040 & -4341229572096 & 0 & 0 & 1076206288752 & 0 & 0 & -62854589952 & 0 & 0 & 7142567040 & 0 & 0 \\
 28 & 17041024000 & -176120000000000 & 0 & 4 & 58548072721600 & 1 & 1 & -7074936915200 & 0 & 4 & 153556480000 & 0 & 10 \\
\hline\hline
\end{tabular} 
}

\captionof{subtable}{Congruences for the triple factorial function, 
                     $n!!! = \MultiFactorial{n}{3}$, 
                     modulo $h$ (and $3h$) for $h \defequals 2,3,4,5$. 
        Supplementary listings containing computational data for the 
        congruences, $n!!! \equiv R_h^{(3)}(n) \pmod{3^i h}$, 
        for $2 \leq i \leq h \leq 5$ are tabulated in the 
        summary notebook reference. 
     } 

\end{sidewaystable} 

\begin{sidewaystable}
\centering 
\smaller 

\resizebox{0.95\textwidth}{0.275\textwidth}{
\begin{tabular}{|l|l|lcc|lcc|lcc|lcc|} \hline 
\hline\tabletopstrut 
$n$ & $\MultiFactorial{n}{4}$ & 
$\widetilde{R}_2^{(4)}(n)$ & $\pod{2}$ & $\pod{8}$ &
$\widetilde{R}_3^{(4)}(n)$ & $\pod{3}$ & $\pod{12}$ &
$\widetilde{R}_4^{(4)}(n)$ & $\pod{4}$ & $\pod{16}$ &
$\widetilde{R}_5^{(4)}(n)$ & $\pod{5}$ & $\pod{20}$ \\ \hline 
 0 & 1 & 1 & 1 & 1 & 1 & 1 & 1 & 1 & 1 & 1 & 1 & 1 & 1 \\
 1 & 1 & 1 & 1 & 1 & 1 & 1 & 1 & 1 & 1 & 1 & 1 & 1 & 1 \\
 2 & 2 & 2 & 0 & 2 & 2 & 2 & 2 & 2 & 2 & 2 & 2 & 2 & 2 \\
 3 & 3 & 3 & 1 & 3 & 3 & 0 & 3 & 3 & 3 & 3 & 3 & 3 & 3 \\
 4 & 4 & 4 & 0 & 4 & 4 & 1 & 4 & 4 & 0 & 4 & 4 & 4 & 4 \\
 5 & 5 & 5 & 1 & 5 & 5 & 2 & 5 & 5 & 1 & 5 & 5 & 0 & 5 \\
 6 & 12 & 12 & 0 & 4 & 12 & 0 & 0 & 12 & 0 & 12 & 12 & 2 & 12 \\
 7 & 21 & 21 & 1 & 5 & 21 & 0 & 9 & 21 & 1 & 5 & 21 & 1 & 1 \\
 8 & 32 & 32 & 0 & 0 & 32 & 2 & 8 & 32 & 0 & 0 & 32 & 2 & 12 \\
 9 & 45 & 45 & 1 & 5 & 45 & 0 & 9 & 45 & 1 & 13 & 45 & 0 & 5 \\
 10 & 120 & 120 & 0 & 0 & 120 & 0 & 0 & 120 & 0 & 8 & 120 & 0 & 0 \\
 11 & 231 & 231 & 1 & 7 & 231 & 0 & 3 & 231 & 3 & 7 & 231 & 1 & 11 \\
 12 & 384 & 384 & 0 & 0 & 384 & 0 & 0 & 384 & 0 & 0 & 384 & 4 & 4 \\
 13 & 585 & -3159 & 1 & 1 & 585 & 0 & 9 & 585 & 1 & 9 & 585 & 0 & 5 \\
 14 & 1680 & -2800 & 0 & 0 & 1680 & 0 & 0 & 1680 & 0 & 0 & 1680 & 0 & 0 \\
 15 & 3465 & -1815 & 1 & 1 & 3465 & 0 & 9 & 3465 & 1 & 9 & 3465 & 0 & 5 \\
 16 & 6144 & 0 & 0 & 0 & 6144 & 0 & 0 & 6144 & 0 & 0 & 6144 & 4 & 4 \\
 17 & 9945 & -364871 & 1 & 1 & 9945 & 0 & 9 & 9945 & 1 & 9 & 9945 & 0 & 5 \\
 18 & 30240 & -437472 & 0 & 0 & 30240 & 0 & 0 & 30240 & 0 & 0 & 30240 & 0 & 0 \\
 19 & 65835 & -508725 & 1 & 3 & 65835 & 0 & 3 & 65835 & 3 & 11 & 65835 & 0 & 15 \\
 20 & 122880 & -573440 & 0 & 0 & 122880 & 0 & 0 & 122880 & 0 & 0 & 122880 & 0 & 0 \\
 21 & 208845 & -32086803 & 1 & 5 & 1990989 & 0 & 9 & 208845 & 1 & 13 & 208845 & 0 & 5 \\
 22 & 665280 & -40544064 & 0 & 0 & 2794176 & 0 & 0 & 665280 & 0 & 0 & 665280 & 0 & 0 \\
 23 & 1514205 & -50324483 & 1 & 5 & 4031325 & 0 & 9 & 1514205 & 1 & 13 & 1514205 & 0 & 5 \\
 24 & 2949120 & -61440000 & 0 & 0 & 5898240 & 0 & 0 & 2949120 & 0 & 0 & 2949120 & 0 & 0 \\
 25 & 5221125 & -2829930075 & 1 & 5 & 358222725 & 0 & 9 & 5221125 & 1 & 5 & 5221125 & 0 & 5 \\
 26 & 17297280 & -3647749248 & 0 & 0 & 452200320 & 0 & 0 & 17297280 & 0 & 0 & 17297280 & 0 & 0 \\
 27 & 40883535 & -4637561553 & 1 & 7 & 570989007 & 0 & 3 & 40883535 & 3 & 15 & 40883535 & 0 & 15 \\
 28 & 82575360 & -5821562880 & 0 & 0 & 722534400 & 0 & 0 & 82575360 & 0 & 0 & 82575360 & 0 & 0 \\
 29 & 151412625 & -264205859375 & 1 & 1 & 52114215825 & 0 & 9 & -1438808175 & 1 & 1 & 151412625 & 0 & 5 \\
 30 & 518918400 & -344048090880 & 0 & 0 & 65833447680 & 0 & 0 & -1378840320 & 0 & 0 & 518918400 & 0 & 0 \\
 31 & 1267389585 & -442855631151 & 1 & 1 & 82524474513 & 0 & 9 & -979895151 & 1 & 1 & 1267389585 & 0 & 5 \\
 32 & 2642411520 & -563949338624 & 0 & 0 & 102789808128 & 0 & 0 & 0 & 0 & 0 & 2642411520 & 0 & 0 \\
 33 & 4996616625 & -26469713463567 & 1 & 1 & 7078405640625 & 0 & 9 & -516689348175 & 1 & 1 & 4996616625 & 0 & 5 \\
 34 & 17643225600 & -34686740160000 & 0 & 0 & 9032888517120 & 0 & 0 & -620425428480 & 0 & 0 & 17643225600 & 0 & 0 \\
\hline\hline 
\end{tabular} 
}

\captionof{subtable}{Congruences for the 
                     quadruple factorial ($4$-factorial) function, 
                     $n!!!! = \MultiFactorial{n}{4}$, 
                     modulo $h$ (and $4h$) for $h \defequals 2,3,4,5$. 
        Supplementary listings containing computational data for the 
        congruences, $n!!!! \equiv R_h^{(4)}(n) \pmod{4^i h}$, 
        for $2 \leq i \leq h \leq 5$ are tabulated in the 
        summary notebook reference. 
     } 

\addtocounter{table}{-1} 
\smallskip\hrule
\captionof{table}{The $\alpha$-factorial functions modulo 
                  $h$ (and $h\alpha$) 
                  for $h \defequals 2,3,4,5$ defined by the special case 
                  expansions from 
                  Section \ref{subsubSection_Examples_NewCongruences} of the 
                  introduction and in 
                  Section \ref{subSection_NewCongruence_Relations_Modulo_Integer_Bases} 
                  where 
                  $\widetilde{R}_p^{(\alpha)}(n) \defequals 
                   \left[z^{\lfloor (n+\alpha-1) / \alpha \rfloor}\right] 
                   \ConvGF{p}{-\alpha}{n}{z}$. 
                 }
\label{table245}
\end{sidewaystable} 

\begin{table}[H] 
\centering 

\smaller 

\begin{subtable}{\textwidth} 
\centering 

\begin{tabular}{|c|l|l|} \hline 
\hline\tabletopstrut 
$m$ & $\ell_{m,2}(z)$ & $p_{m,2}(x)$ \\ \hline 
1 & $1$ & $1$ \\ 
2 & $4 - 3z$ & $x+4$ \\ 
3 & $11 - 17z + 7z^2$ & $x^2+10x+22$ \\ 
4 & $26 - 62z + 52z^2  - 15z^3$ & 
    $x^3+18x^2+96x+156$ \\ 
5 & $57 - 186z + 238z^2 - 139z^3 + 31z^4$ & 
    $x^4+28x^3+264x^2+1008x+1368$ \\ \hline 
\hline\tabletopstrut 
$m$ & $\ell_{m,3}(z)$ & $p_{m,3}(x)$ \\ \hline 
1 & $1$ & $1$ \\ 
2 & $5 - 8z$ & $x+5$ \\ 
3 & $18 - 60z + 52 z^2$ & $x^2+12x+36$ \\ 
4 & $58 - 300 z + 532 z^2  - 320 z^3$ & $x^3+21x^2+144x+348$ \\ 
5 & $179 - 1268 z + 3436 z^2 - 4192 z^3 + 1936 z^4$ & 
    $x^4+32x^3+372x^2+1968x+4296$ \\ \hline 
\hline\tabletopstrut 
$m$ & $\ell_{m,4}(z)$ & $p_{m,4}(x)$ \\ \hline 
1 & $1$ & $1$ \\ 
2 & $6 - 15z$ & $x+6$ \\ 
3 & $27 - 141z + 189 z^2$ & $x^2+14x+54$ \\ 
4 & $112 - 906 z + 2484 z^2  - 2295 z^3$ & 
    $x^3+24x^2+204x+672$ \\ 
5 & $453 - 4998 z + 20898 z^2 - 39123 z^3 + 27621 z^4$ & 
    $x^4+36x^3+504x^2+3504x+10872$ \\ \hline 
\hline\tabletopstrut 
$m$ & $\ell_{m,5}(z)$ & $p_{m,5}(x)$ \\ \hline 
1 & $1$ & $1$ \\ 
2 & $7 - 24z$ & $x+7$ \\ 
3 & $38 - 272z + 496 z^2$ & 
    $x^2+16x+76$ \\ 
4 & $194 - 2144 z + 7984 z^2  - 9984 z^3$ & 
    $x^3+27x^2+276x+1164$ \\ 
5 & $975 - 14640 z + 82960 z^2 - 209920 z^3 + 199936 z^4$ & 
    $x^4+40x^3+660x^2+5760x+23400$ \\ \hline 
\hline 
\end{tabular} 
\subcaption{Generating the sequences of binomials, 
     $2^{p}-1$, $3^{p}-1$, $4^{p}-1$, and $5^{p}-1$}

\begin{tabular}{|c|l|l|} \hline 
\hline\tabletopstrut 
$m$ & $\ell_{m,s+1}(z)$ & $p_{m,s+1}(x)$ \\ \hline 
1 & $1$ & $1$ \\ 
2 & $3+s (1-2 z)-s^2 z$ & $3+s (1+x)$ \\ 
3 & $6+s^4 z^2-4 s (-1+2 z)+s^3 z (-2+3 z)$ & 
    $12+8 s (1+x)+s^2 (2+2 x+x^2)$ \\ 
  & $\quad + 
     s^2 (1-7 z+3 z^2)$ & \\ 
4 & $10-s^6 z^3-10 s (-1+2 z)-s^5 z^2 (-3+4 z)$ & 
    $60+60 s (1+x)+15 s^2 (2+2 x+x^2)$ \\ 
  & $\quad + 
     5 s^2 (1-5 z+3 z^2) - 
     s^4 z (3-13 z+6 z^2)$ & 
    $\quad + 
     s^3 (6+6 x+3 x^2+x^3)$ \\ 
  & $\quad  + 
     s^3 (1-14 z+21 z^2-4 z^3)$ & \\ 
5 & $15+s^8 z^4-20 s (-1+2 z)+s^7 z^3 (-4+5 z)$ & 
    $360+480 s (1+x) + 180 s^2 (2+2 x+x^2)$ \\ 
  & $\quad + 
     5 s^2 (3-13 z+9 z^2)+s^6 z^2 (6-21 z+10 z^2)$ & 
    $\quad + 
     24 s^3 (6+6 x+3 x^2+x^3)$ \\ 
  & $\quad - 
     3 s^3 (-2+18 z-27 z^2+8 z^3)$ & 
    $\quad + 
     s^4 (24+24 x+12 x^2+4 x^3+x^4)$ \\ 
  & $\quad + 
     s^5 z (-4+33 z-44 z^2+10 z^3)$ & \\ 
  & $\quad + 
     s^4 (1-23 z+73 z^2-46 z^3+5 z^4)$ & \\ \hline 
\hline 
\end{tabular} 
\subcaption{Generating the $p^{th}$ power sequences, $(s+1)^{p}-1$} 

\end{subtable} 

\caption{Convergent-based generating function identities for the 
         binomial $p^{th}$ power sequences 
         generated by the examples in 
         Section \ref{subsubSection_Apps_Example_SumsOfPowers_Seqs}} 
\label{table_ConvGF_Examples_for_PthPowerSeqs} 

\end{table} 

\addtocounter{table}{1}
\setcounter{subtable}{0} 

\begin{sidewaystable}

\centering 
\smaller 

\resizebox{0.95\textwidth}{!}{
\begin{tabular}{|c|lll|} 
\hline\tabletopstrut 
$n$ & 
$p_{n,0}(h)$ & $p_{n,1}(h)$ & $p_{n,2}(h)$ \\ \hline 
0 & 
    1 & 0 & 0 \\ 
1 & 
    $h$ & $1$ & 0 \\ 
2 & 
    $h (h-1)^2$ & $h(h-2)$ & $h-1$ \\ 
3 & 
    $h (h-1)^2 (h-2)$ & $h(h-1)(h-3)$ & $h(h-3)$ \\ 
4 & 
    $h (h-1)^2 (h-2)^2 (h-3)$ & $h(h-1) (h-2)^2 (h-4)$ & 
    $h(h-1)(h-3)(h-4)$ \\ 
5 & 
    $h (h-1)^2 (h-2)^2 (h-3)(h-4)$ & 
    $h(h-1) (h-2)^2 (h-3)(h-5)$ & 
    $h(h-1)(h-2)(h-4)(h-5)$ \\ 
6 & 
    $h (h-1)^2 (h-2)^2 (h-3)^2 (h-4)(h-5)$ & 
    $h(h-1) (h-2)^2 (h-3)^2 (h-4)(h-6)$ & 
    $h(h-1)(h-2) (h-3)^2 (h-5)(h-6)$ \\ 
7 & $h (h-1)^2 (h-2)^2 (h-3)^2 (h-4)(h-5)(h-6)$ & 
    $h(h-1) (h-2)^2 (h-3)^2 (h-4)(h-5)(h-7)$ & 
    $h(h-1)(h-2) (h-3)^2 (h-4)(h-6)(h-7)$ \\ 
\hline\hline  
$n$ & $p_{n,3}(h)$ & $p_{n,4}(h)$ & $p_{n,5}(h)$ \\ \hline 
0 & 0 & 0 & 0 \\ 
1 & 0 & 0 & 0 \\ 
2 & 0 & 0 & 0 \\ 
3 & $h-1$ & 0 & 0 \\ 
4 & $h(h-2)(h-4)$ & $(h-1)(h-2)$ & 0 \\ 
5 & $h(h-1)(h-2)(h-4)(h-5)$ & 
    $h(h-2)(h-5)$ & $(h-1)(h-2)$ \\ 
6 & $h(h-1)(h-2)(h-4)(h-5)(h-6)$ & 
    $h(h-1)(h-3)(h-5)(h-6)$ & 
    $h(h-2)(h-3)(h-6)$ \\ 
7 & $h(h-1)(h-2)(h-3)(h-5)(h-6)(h-7)$ & 
    $h(h-1)(h-2)(h-5)(h-6)(h-7)$ & 
    $h(h-1)(h-3)(h-6)(h-7)$ \\ 
\hline\hline 
$n$ & $p_{n,6}(h)$ & $p_{n,7}(h)$ & $m_{n,h}$ \\ \hline 
0 & 0 & 0 & 1 \\ 
1 & 0 & 0 & $h-1$ \\ 
2 & 0 & 0 & $h-2$ \\ 
3 & 0 & 0 & $(h-2)(h-3)$ \\ 
4 & 0 & 0 & $(h-3)(h-4)$ \\ 
5 & 0 & 0 & $(h-3)(h-4)(h-5)$ \\ 
6 & $(h-1)(h-2)(h-3)$ & 0 & $(h-4)(h-5)(h-6)$ \\ 
7 & $h(h-2)(h-3)(h-7)$ & $(h-1)(h-2)(h-3)$ & $(h-4)(h-5)(h-6)(h-7)$ \\ 
\hline\hline 
\end{tabular} 
} 

\captionof{subtable}{The auxiliary numerator subsequences, 
         $C_{h,n}(\alpha, R) \defmapsto 
          \frac{(-\alpha)^{n} m_{n,h}}{n!} \times 
          \sum_{i=0}^{n} \binom{n}{i} p_{n,i}(h) \Pochhammer{R / \alpha}{i}$, 
          expanded by the finite-degree polynomial sequence terms 
          defined by the Stirling number sums in 
          \eqref{eqn_Chn_formula_stmts-exp_v5} of 
          Section \ref{subsubSection_Properties_Of_ConvFn_Phz-AuxNumFn_Subsequences}. 
          } 
\label{table_ConvNumFnSeqs_Chn_AlphaR_SpCaseListings-first_subtable_pageref} 

\end{sidewaystable} 

\addtocounter{table}{-1} 

\begin{table}[H] 
\centering 
\smaller 

\begin{subtable}{\textwidth} 
\centering 

\addtocounter{subtable}{1} 

\resizebox{0.95\textwidth}{!}{
\begin{tabular}{|c|l|l|} \hline 
\hline\tabletopstrut 
$n$ & $m_{h}$ & $(-1)^{n} n! \cdot m_h^{-1} \cdot C_{h,n}(\alpha, R)$ \\ \hline 
1 & $1$ & $-(h-1) (R+ h\alpha)$ \\ 
2 & $(h-2)$ & 
    $\alpha  \left(2 h^2-3 h-1\right) R+\alpha ^2 (h-1)^2 h+(h-1) R^2$ \\ 
3 & $(h-3)$ & 
    $3 \alpha  (h-2) \left(h^2-2 h-1\right) R^2+ 
     \alpha ^2 (h-2) \left(3 h^3-9 h^2+2 h-2\right) R$ \\ 
  & & 
    $\qquad + 
     \alpha ^3 (h-2)^2 (h-1)^2 h+(h-2) (h-1) R^3$ \\ 
4 & $(h-3)(h-4)$ & 
    $\alpha ^2 \left(6 h^4-36 h^3+53 h^2-9 h+22\right) R^2+2 \alpha ^3 \
     \left(2 h^5-15 h^4+36 h^3-36 h^2+19 h+6\right) R$ \\ 
  & & 
    $\qquad + 
     \alpha ^4 (h-3) \
     (h-2)^2 (h-1)^2 h+(h-2) (h-1) R^4+2 \alpha  (h-3) (h-2) (2 h+1) R^3$ \\ 
5 & $(h-3)(h-4)(h-5)$ & 
    $5 \alpha  (h-2) \left(h^2-3 h-2\right) R^4+5 \alpha ^2 \left(2 h^4-14 h^3+23 h^2-h+14\right) R^3$ \\ 
  & & 
    $\qquad + 
     5 \alpha ^3 \left(2 h^5-18 h^4+49 h^3-49 h^2+40 h+20\right) R^2$ \\ 
  & & 
    $\qquad + 
     \alpha ^4 \left(5 h^6-55 h^5+215 h^4-395 h^3+374 h^2-72 h+48\right) R$ \\ 
  & & 
    $\qquad + 
     \alpha ^5 (h-4) (h-3) (h-2)^2 (h-1)^2 h+(h-2) (h-1) R^5$ \\ 
\hline\hline\tabletopstrut 
$n$ & $m_{h}$ & $(-1)^{n} n! \cdot m_h^{-1} \cdot C_{h,n}(\alpha, R)$ \\ \hline 
1 & $1$ & $\alpha  h^2+(R-\alpha ) h-R$ \\ 
2 & $(h-2)$ & 
$\alpha ^2 h^3+2 \alpha  h^2 (R-\alpha )+h \left(\alpha ^2+R^2-3 \alpha  R\right)-R (\alpha +R)$ \\ 
3 & $(h-3)$ & $\alpha ^3 h^5+3 \alpha ^2 h^4 (R-2 \alpha )+ 
     \alpha  h^3 \left(13 \alpha ^2+3 R^2-15 \alpha  R\right)$ \\ 
  & & 
    $\qquad + 
     h^2 \left(-12 \alpha ^3+R^3-12 \alpha  R^2+20 \alpha ^2 R\right)+ 
     h \left(4 \alpha ^3-3 R^3+9 \alpha  R^2-6 \alpha ^2 R\right)$ \\ 
  & & 
    $\qquad + 
     2 R (\alpha +R) (2 \alpha +R)$ \\ 
4 & $(h-3)(h-4)$ & 
    $\alpha ^4 h^6+\alpha ^3 h^5 (4 R-9 \alpha )+ 
     \alpha ^2 h^4 \left(31 \alpha ^2+6 R^2-30 \alpha  R\right)$ \\ 
  & & 
    $\qquad + 
     \alpha  h^3 \left(-51 \alpha ^3+4 R^3-36 \alpha  R^2+72 \alpha ^2 R\right)$ \\ 
  & & 
    $\qquad + 
     h^2 \left(40 \alpha ^4+R^4-18 \alpha  R^3+53 \alpha ^2 R^2- 
     72 \alpha ^3 R\right)$ \\ 
  & & 
    $\qquad + 
     h \left(-12 \alpha ^4-3 R^4+14 \alpha  R^3-9 \alpha ^2 R^2+38 \alpha ^3 R\right)+2 R (\alpha +R) (2 \alpha +R) (3 \alpha +R)$ \\ 
5 & $(h-3)(h-4)(h-5)$ & 
    $\alpha ^5 h^7+\alpha ^4 h^6 (5 R-13 \alpha )+ 
     \alpha ^3 h^5 \left(67 \alpha ^2+10 R^2-55 \alpha  R\right)$ \\ 
  & & 
    $\qquad + 
     5 \alpha ^2 h^4 \left(-35 \alpha ^3+2 R^3-18 \alpha  R^2+ 
     43 \alpha ^2 R\right)$ \\ 
  & & 
    $\qquad + 
     \alpha h^3 \left(244 \alpha ^4+5 R^4- 
     70 \alpha  R^3+245 \alpha ^2 R^2-395 \alpha ^3 R\right)$ \\ 
  & & 
    $\qquad + 
     h^2 \left(-172 \alpha ^5+R^5-25 \alpha  R^4+115 \alpha ^2 R^3- 
     245 \alpha ^3 R^2+374 \alpha ^4 R\right)$ \\ 
 & & 
    $\qquad + 
     h \left(48 \alpha ^5- 
     3 R^5+20 \alpha  R^4-5 \alpha ^2 R^3+200 \alpha ^3 R^2- 
     72 \alpha ^4 R\right)$ \\ 
 & & 
    $\qquad + 
     2 R (\alpha +R) (2 \alpha +R) (3 \alpha +R) (4 \alpha +R)$ \\ \hline 
\end{tabular} 
} 
\subcaption{Alternate factored forms of the 
            convergent function subsequences, 
            $C_{h,n}(\alpha, R)\defequals [z^n] \ConvFP{h}{\alpha}{R}{z}$, 
            gathered with respect to powers of $R$ and $h$.} 

\resizebox{0.95\textwidth}{!}{
\begin{tabular}{|c|l|} 
\hline\tabletopstrut 
$n$ & $n! \cdot C_{h,n}(\alpha, R)$ \\ \hline 
2 & $\alpha ^2 h^4+2 \alpha  h^3 (R-2 \alpha )+h^2 \left(5 \alpha ^2+R^2-7 
     \alpha  R\right)-h (3 R-2 \alpha ) (R-\alpha )+2 R (\alpha +R)$ \\ 
3 & $-\alpha ^3 h^6-3 \alpha ^2 h^5 (R-3 \alpha )-\alpha  h^4 \left(31 
     \alpha ^2+3 R^2-24 \alpha  R\right)+h^3 \left(51 \alpha ^3-R^3+21 
     \alpha  R^2-65 \alpha ^2 R\right)$ \\ 
  & $\qquad + 
     h^2 \left(-40 \alpha ^3+6 R^3-45 
     \alpha  R^2+66 \alpha ^2 R\right)-h (R-\alpha ) \left(12 \alpha ^2+11 
     R^2-10 \alpha R\right)+6 R (\alpha +R) (2 \alpha +R)$ \\ 
4 & $\alpha ^4 h^8+4 \alpha ^3 h^7 (R-4 \alpha )+2 \alpha ^2 h^6 
     \left(53 \alpha ^2+3 R^2-29 \alpha  R\right)+2 \alpha  h^5 
     \left(-188 \alpha ^3+2 R^3-39 \alpha  R^2+165 \alpha ^2 R\right)$ \\ 
  & $\qquad + 
     h^4 \left(769 \alpha ^4+R^4-46 \alpha R^3+377 \alpha ^2 R^2-936 
     \alpha ^3 R\right)-2 h^3 \left(452 \alpha ^4+5 R^4-94 \alpha  R^3+406 
     \alpha ^2 R^2-703 \alpha ^3 R\right)$ \\ 
  & $\qquad + 
     h^2 \left(564 \alpha ^4+35 R^4-302 \alpha R^3+721 \alpha ^2 R^2- 
     1118 \alpha ^3 R\right)-2 h (R-\alpha ) \left(-72 \alpha ^3+25 R^3-17 
     \alpha  R^2+114 \alpha ^2 R\right)$ \\ 
  & $\qquad + 
     24 R (\alpha +R) (2 \alpha +R) (3 \alpha +R)$ \\ 
5 & $-\alpha ^5 h^{10}-5 \alpha ^4 h^9 (R-5 \alpha )-5 \alpha ^3 h^8 
     \left(54 \alpha ^2+2 R^2-23 \alpha  R\right)-10 \alpha ^2 h^7 
     \left(-165 \alpha ^3+R^3-21 \alpha  R^2+111 \alpha ^2 R\right)$ \\ 
  & $\qquad - 
     \alpha  h^6 \left(6273 \alpha ^4+5 R^4-190 \alpha  R^3+1795 
     \alpha ^2 R^2-5860 \alpha ^3 R\right)$ \\ 
  & $\qquad + 
     h^5 \left(15345 \alpha ^5-R^5+85 \alpha  R^4-1425 \alpha ^2 R^3+ 
     8015 \alpha ^3 R^2-18519 \alpha ^4 R\right)$ \\ 
  & $\qquad + 
     5 h^4 \left(-4816 \alpha ^5+3 R^5-111 \alpha  R^4+1055 \alpha ^2 R^3- 
     4011 \alpha ^3 R^2+7205 \alpha ^4 R\right)$ \\ 
  & $\qquad - 
     5 h^3 \left(-4660 \alpha ^5+17 R^5-339 \alpha  R^4+1947 \alpha ^2 R^3- 
     5703 \alpha ^3 R^2+8438 \alpha ^4 R\right)$ \\ 
  & $\qquad + 
     h^2 \left(-12576 \alpha ^5+225 R^5-2200 \alpha  R^4+ 
     7975 \alpha ^2 R^3-22900 \alpha ^3 R^2+26400 \alpha ^4 R\right)$ \\ 
  & $\qquad - 
     2 h (R-\alpha ) \left(1440 \alpha ^4+137 R^4+7 \alpha  R^3+1802 \alpha ^2 R^2-1848 \alpha ^3 R\right)+120 R (\alpha +R) (2 \alpha +R) (3 \alpha +R) (4 \alpha +R)$ \\ \hline 
\hline 
\end{tabular} 
} 
\end{subtable} 

\caption{The auxiliary convergent function subsequences, 
         $C_{h,n}(\alpha, R) \defequals [z^n] \ConvFP{h}{\alpha}{R}{z}$,  
         defined in Section \ref{subsubSection_Properties_Of_ConvFn_Phz}.}
\label{table_ConvNumFnSeqs_Chn_AlphaR_SpCaseListings} 

\end{table} 

\setcounter{subtable}{0} 

\begin{table}[H] 
\centering\smaller 

\begin{subtable}{\textwidth} 
\centering 
\resizebox{0.95\textwidth}{!}{
\begin{tabular}{|c|l|} \hline 
\hline\tabletopstrut 
$k$ & $(-1)^{h-k} z^{-(h-k)} \cdot R_{h,h-k}(\alpha; z)$ \\ \hline 
1 & $1$ \\ 
2 & $-\frac{1}{2} \alpha \left(h^2-h+2\right) z+h-1$ \\ 
3 & $\frac{1}{2} (h-2) (h-1) -\frac{1}{2} \alpha (h-2) 
     \left(h^2+3\right) z+\frac{1}{24} \alpha ^2 
     \left(3 h^4-10 h^3+21 h^2-14 h+24\right) z^2$ \\ 
4 & $-\frac{1}{4} \alpha  (h-3) (h-2) \left(h^2+h+4\right) z + 
     \frac{1}{24} \alpha ^2 (h-3) \left(3 h^4-4 h^3+19 h^2-2 h+56\right) z^2$ \\ 
  & $\phantom{-\frac{1}{4}} - 
     \frac{1}{48} \alpha ^3\left(h^6-7 h^5+23 h^4-37 h^3+48 h^2-28 h+48 
     \right) z^3$ \\ 
  & $\phantom{-\frac{1}{4}} + 
      \frac{1}{6} (h-3) (h-2) (h-1)$ \\ 
5 & $-\frac{1}{12} \alpha  (h-4) (h-3) (h-2) \left(h^2+2 h+5\right) z + 
     \frac{1}{48} \alpha ^2 (h-4) (h-3) \left(3 h^4+2 h^3+23 h^2+16 h+100 
     \right) z^2$ \\ 
  & $\phantom{-\frac{1}{12}} - 
     \frac{1}{48} \alpha ^3 (h-4) 
     \left(h^6-4 h^5+14 h^4-16 h^3+61 h^2-12 h+180\right) z^3$ \\ 
  & $\phantom{-\frac{1}{12}} + 
     \frac{\alpha ^4}{5760} \left(15 h^8-180 h^7+950 h^6-2688 h^5+4775
      h^4-5340 h^3+5780 h^2-3312 h+5760\right) z^4$ \\ 
  & $\phantom{-\frac{1}{12}} + 
     \frac{1}{24} (h-4) (h-3) (h-2) (h-1)$ \\ \hline 
\hline\tabletopstrut 
$k$ & $k! (-1)^{h-k} z^{-(h-k)} \cdot R_{h,h-k}(\alpha; z)$ \\ \hline 
1 & $1$ \\ 
2 & $-\alpha h^2 z+h (\alpha  z+2)-2 (\alpha  z+1)$ \\ 
3 & $\frac{3}{4} \alpha ^2 h^4 z^2- 
     \frac{1}{2} \alpha  h^3 z (5 \alpha  z+6)+ 
     \frac{3}{4} h^2 \left(7 \alpha ^2 z^2+8 \alpha  z+4\right)$ \\ 
  & $\qquad + 
     \frac{1}{2} h \left(-7 \alpha ^2 z^2-18 \alpha z- 
     18\right)+6 \left(\alpha ^2 z^2+3 \alpha  z+1\right)$ \\ 
4 & $-\frac{1}{2} \alpha ^3 h^6 z^3 + 
     \frac{1}{2} \alpha ^2 h^5 z^2 (7 \alpha  z+6) - 
     \frac{1}{2} \alpha  h^4 z \left(23 \alpha ^2 z^2+26 \alpha  z+12\right)$ \\ 
  & $\qquad + 
     \frac{1}{2} h^3 \left(37 \alpha ^3 z^3+62 \alpha ^2 z^2+ 
     48 \alpha  z+8\right)$ \\ 
  & $\qquad + 
     h^2 \left(-24 \alpha ^3 z^3-59 \alpha ^2 z^2-30 \alpha  z-24\right)$ \\ 
  & $\qquad + 
     2 h \left(7 \alpha ^3 z^3+31 \alpha ^2 z^2+42 \alpha  z+ 
     22\right)-24 \left(\alpha ^3 z^3+7 \alpha ^2 z^2+6 \alpha  z+1\right)$ \\ 
5 & $\frac{5}{16} \alpha ^4 h^8 z^4- 
     \frac{5}{4} \alpha ^3 h^7 z^3 (3 \alpha  z+2)+ 
     \frac{5}{24} \alpha ^2 h^6 z^2 \left(95 \alpha ^2 z^2+96 \alpha z+ 
     36\right)$ \\ 
  & $\qquad - 
     \frac{1}{2} \alpha  h^5 z \left(112 \alpha ^3 z^3+150 \alpha ^2 z^2+ 
     95 \alpha  z+20\right)$ \\ 
  & $\qquad + 
     \frac{5}{48} h^4 \left(955 \alpha ^4 z^4+1728 \alpha ^3 z^3+ 
     1080 \alpha ^2 z^2+672 \alpha  z+48\right)$ \\ 
  & $\qquad - 
     \frac{5}{4} h^3 \left(89 \alpha ^4 z^4+250 \alpha ^3 z^3+ 
     242 \alpha ^2 z^2+104 \alpha  z+40\right)$ \\ 
  & $\qquad + 
     \frac{5}{12} h^2 \left(289 \alpha ^4 z^4+1536 \alpha ^3 z^3+ 
     1584 \alpha ^2 z^2+408 \alpha  z+420\right)$ \\ 
  & $\qquad + 
     h \left(-69 \alpha ^4 z^4-570 \alpha ^3 z^3-1270 \alpha ^2 z^2- 
     820 \alpha  z-250\right)$ \\ 
  & $\qquad + 
     120 \left(\alpha ^4 z^4+15 \alpha ^3 z^3+25 \alpha ^2 z^2+ 
     10 \alpha  z+1\right)$ \\ \hline 
\hline 
\end{tabular} 
} 
\end{subtable} 

\caption{The auxiliary convergent numerator function subsequences, 
         $R_{h,k}(\alpha; z) \defequals [R^k] \ConvFP{h}{\alpha}{R}{z}$, 
         defined by 
         Section \ref{subsubSection_Properties_Of_ConvFn_Phz-AuxNumFn_Subsequences}.}
\label{table_ConvNumFnSeqs_Rhk_Alphaz_SpCaseListings} 

\end{table}

\bigskip
\hrule
\bigskip

\noindent \textit{2010 Mathematics Subject Classification}: 
Primary 05A10; Secondary 05A15, 11A55, 11Y55, 11Y65, 11B65. \\ 

\noindent\textit{Keywords}: 
continued fraction, J-fraction, S-fraction, 
Pochhammer symbol, factorial function, 
multifactorial, multiple factorial, 
double factorial, superfactorial, 
rising factorial, Pochhammer k-symbol, 
Barnes G-function, hyperfactorial, subfactorial, triple factorial, 
generalized Stirling number, Stirling number of the first kind, 
confluent hypergeometric function, Laguerre polynomial, 
ordinary generating function, diagonal generating function, 
Hadamard product, divergent ordinary generating function, 
formal Laplace-Borel transform, Stirling number congruence. 

\bigskip
\hrule
\bigskip 

\noindent 
(Concerned with sequences
\seqnum{A000043}, \seqnum{A000108}, \seqnum{A000142}, \seqnum{A000165}, 
\seqnum{A000166}, \seqnum{A000178}, \seqnum{A000215}, \seqnum{A000225}, 
\seqnum{A000407}, \seqnum{A000668}, \seqnum{A000918}, \seqnum{A000978}, 
\seqnum{A000984}, \seqnum{A001008}, \seqnum{A001044}, \seqnum{A001097}, 
\seqnum{A001142}, \seqnum{A001147}, \seqnum{A001220}, \seqnum{A001348}, 
\seqnum{A001359}, \seqnum{A001448}, \seqnum{A002109}, \seqnum{A002144}, 
\seqnum{A002234}, \seqnum{A002496}, \seqnum{A002805}, \seqnum{A002981}, 
\seqnum{A002982}, \seqnum{A003422}, \seqnum{A005109}, \seqnum{A005165}, 
\seqnum{A005384}, \seqnum{A006512}, \seqnum{A006882}, \seqnum{A007406}, 
\seqnum{A007407}, \seqnum{A007408}, \seqnum{A007409}, \seqnum{A007540}, 
\seqnum{A007559}, \seqnum{A007619}, \seqnum{A007661}, \seqnum{A007662}, 
\seqnum{A007696}, \seqnum{A008275}, \seqnum{A008277}, \seqnum{A008292}, 
\seqnum{A008544}, \seqnum{A008554}, \seqnum{A009120}, \seqnum{A009445}, 
\seqnum{A010050}, \seqnum{A019434}, \seqnum{A022004}, \seqnum{A022005}, 
\seqnum{A023200}, \seqnum{A023201}, \seqnum{A023202}, \seqnum{A023203}, 
\seqnum{A024023}, \seqnum{A024036}, \seqnum{A024049}, \seqnum{A027641}, 
\seqnum{A027642}, \seqnum{A032031}, \seqnum{A033312}, \seqnum{A034176}, 
\seqnum{A046118}, \seqnum{A046124}, \seqnum{A046133}, \seqnum{A047053}, 
\seqnum{A061062}, \seqnum{A066802}, \seqnum{A077800}, \seqnum{A078303}, 
\seqnum{A080075}, \seqnum{A085157}, \seqnum{A085158}, \seqnum{A087755}, 
\seqnum{A088164}, \seqnum{A094638}, \seqnum{A100043}, \seqnum{A100089}, 
\seqnum{A100732}, \seqnum{A104344}, \seqnum{A105278}, \seqnum{A123176}, 
\seqnum{A130534}, \seqnum{A157250}, \seqnum{A166351} and \seqnum{A184877}.
) 

\bigskip
\hrule
\bigskip

\vspace*{+.1in}
\noindent
Received January 5 2016;
revised versions received  March 13 2016; April 7 2016; December 24 2016; December 29 2016.
Published in {\it Journal of Integer Sequences}, January 8 2017.

\bigskip
\hrule
\bigskip

\noindent
Return to
\htmladdnormallink{Journal of Integer Sequences home page}{http://www.cs.uwaterloo.ca/journals/JIS/}.
\vskip .1in


\begin{thebibliography}{10}

\bibitem{CONTRIB-THEORY-BARNESGFN}
V.~S. Adamchik, Contributions to the theory of the {B}arnes function, 2003, 
  \url{https://arxiv.org/abs/math/0308086}. 

\bibitem{GENWTHM-DBLHYPERSUPER-FACTFNS}
C. Aebi and G. Cairns, Generalizations of {W}ilson's theorem for
  double--, hyper--, sub--, and superfactorials, 
  {\em Amer. Math. Monthly} {\bf 122} (2015), 433--443.

\bibitem{PROPS-ZEROS-CHYPFNS80}
S.~Ahmed, Properties of the zeros of confluent hypergeometric functions, {\em
  J. Approx. Theory} {\bf 34} (1982), 335--347.

\bibitem{DBLFACTFN-COMBIDENTS-SURVEY}
D. Callan, A combinatorial survey of combinatorial identities for the double
  factorial,  2009, \url{https://arxiv.org/abs/0906.1317}. 

\bibitem{ADVCOMB}
L.~Comtet, {\em Advanced Combinatorics: The Art of Finite and Infinite
  Expressions}, Reidel Publishing Company, 1974.

\bibitem{ON-HYPGEOMFNS-PHKSYMBOL}
R.~Diaz and E.~Pariguan, On hypergeometric functions and $k$--{P}ochhammer
  symbol, 2005, \url{https://arxiv.org/abs/math/0405596}. 

\bibitem{ACOMB-BOOK}
P.~Flajolet and R.~Sedgewick, {\em Analytic Combinatorics}, Cambridge
  University Press, 2010.

\bibitem{FLAJOLET80B}
P. Flajolet, Combinatorial aspects of continued fractions, {\em Discrete
  Math.} {\bf 32} (1980), 125--161.

\bibitem{FLAJOLET82}
P. Flajolet, On congruences and continued fractions for some classical
  combinatorial quantities, {\em Discrete Math.} {\bf 41} (1982),
  145--153.

\bibitem{LGWORKS-ASYMP-SPFNZEROS2008}
W.~Gautschi and C.~Giordano, {L}uigi {G}atteschi's work on asymptotics of
  special functions and their zeros, {\em Numer. Algorithms} {\bf 49}
  (2008), 11--31.

\bibitem{MAA-FUN-WITH-DBLFACT}
H. Gould and J. Quaintance, Double fun with double factorials, {\em
  Math. Mag.} {\bf 85} (2012), 177--192.

\bibitem{GKP}
R.~L. Graham, D.~E. Knuth, and O.~Patashnik, {\em Concrete Mathematics: A
  Foundation for Computer Science}, Addison-Wesley, 1994.

\bibitem{HARDYWRIGHTNUMT}
G.~H. Hardy and E.~M. Wright, {\em An Introduction to the Theory of
  Numbers}, Oxford University Press, 2008.

\bibitem{CVLPOLYS}
D.~E. Knuth, Convolution polynomials, {\em Math J} {\bf 2}
  (1992), 67--78.

\bibitem{TAOCPV1}
D.~E. Knuth, {\em The Art of Computer Programming: Fundamental Algorithms},
  Vol.~1, Addison-Wesley, 1997.

\bibitem{GFLECT}
S.~K. Lando, {\em Lectures on Generating Functions}, American Mathematical
  Society, 2002.

\bibitem{NISTHB}
F. W.~J. Olver, D.~W. Lozier, R.~F. Boisvert, and C.~W. Clark,
  {\em {NIST} Handbook of Mathematical Functions}, Cambridge
  University Press, 2010.

\bibitem{PRIMEREC}
P.~Ribenboim, {\em The New Book of Prime Number Records}, Springer, 1996.

\bibitem{UC}
S.~Roman, {\em The Umbral Calculus}, Dover, 1984.

\bibitem{MULTIFACTJIS}
M.~D. Schmidt, Generalized $j$--factorial functions, polynomials, and
  applications, {\em J. Integer Seq.} {\bf 13} (2010).

\bibitem{SUMMARYNBREF-STUB}
M.~D. Schmidt, Mathematica summary notebook, 
supplementary reference, and computational documentation, 2016,  \\
{\small\url{\TheSummaryNBFileGoogleDriveLink}}. 

\bibitem{OEIS}
N.~J.~A. Sloane, The {O}nline {E}ncyclopedia of {I}nteger {S}equences, 2010, 
\url{http://oeis.org}. 

\bibitem{ATLASOFFUNCTIONS}
J. Spanier and K.~B. Oldham, {\em An Atlas of Functions}, Taylor \&
  Francis, 1987.

\bibitem{ECV2}
R.~P. Stanley, {\em Enumerative Combinatorics}, Vol.~2, 
Cambridge University Press, 1999.

\bibitem{GFOLOGY}
H.~S. Wilf, {\em Generatingfunctionology}, Academic Press, 1994.

\bibitem{WOLFRAMFNSSITE-INTRO-FACTBINOMS}
{W}olfram Functions~Site, 
Introduction to the factorials and binomials, 2016, 
{\footnotesize\url{http://functions.wolfram.com/GammaBetaErf/Factorial/introductions/FactorialBinomials/05/}}.


\end{thebibliography}
\end{document}